\pdfoutput=1
\documentclass[12pt]{amsart}
\usepackage{amsmath,amsthm,amssymb}
\usepackage{enumitem}
\usepackage{graphicx}
\usepackage{float}
\usepackage{lmodern}
\usepackage{cite}
\usepackage[margin=1in]{geometry}
\usepackage{tikz}
\usetikzlibrary{matrix}
\usetikzlibrary{arrows}
\usetikzlibrary{positioning}
\usepackage{overpic}
\usepackage{mathtools}
\usepackage{subcaption}
\theoremstyle{plain}
\newtheorem{thm}{Theorem}[section]
\newtheorem{cor}[thm]{Corollary}
\newtheorem{lem}[thm]{Lemma}

\newtheorem{conj}[thm]{Conjecture}
\theoremstyle{definition}
\newtheorem{defn}[thm]{Definition}
\theoremstyle{remark}
\newtheorem{rem}[thm]{Remark}

\newcommand{\Z}{\mathbb{Z}}
\newcommand{\Q}{\mathbb{Q}}

\begin{document}

\title[A Categorification of the HOMFLY-PT Polynomial]{A Categorification of the HOMFLY-PT Polynomial with a Spectral Sequence to Knot Floer Homology}

\author{Nathan Dowlin}
\date{}

\begin{abstract}

Let $E_{k}^{F}(D)$ be the spectral sequence induced by the oriented cube of resolutions on knot Floer homology. We prove that $E_{2}^{F}(D)$ is a triply graded link invariant whose graded Euler characteristic is the HOMFLY-PT polynomial and that the higher pages are link invariants. By construction, the spectral sequence converges to knot Floer homology. We show that the rank of the torsion-free part of $E_{2}^{F}(D)$ is the rank of HOMFLY-PT homology.

\end{abstract}

\maketitle

\section{Introduction}

The goal of this paper is to better understand the relationship between HOMFLY-PT homology and knot Floer homology. A relationship was first conjectured by Dunfield, Gukov and Rasmussen in \cite{Gukov}:

\begin{conj} \label{Guk} There is a spectral sequence from HOMFLY-PT homology to knot Floer homology.

\end{conj}

\noindent
Evidence for this relationship arose when Ozsv\'{a}th and Szab\'{o} developed the oriented cube of resolutions for knot Floer homology (\hspace{1sp}\cite{Szabo}) and noted that the algebra involved was reminiscent of HOMFLY-PT homology.

These similarities have been studied extensively by Manolescu (\hspace{1sp}\cite{Manolescu}), Gilmore (\hspace{1sp}\cite{Gilmore}), and the author (\hspace{1sp}\cite{Me}, \cite{Me2}). In particular, Manolescu defined an untwisted version of the cube of resolutions and conjectured a possible construction of the spectral sequence.

\begin{conj} \label{ManCon} Let $D$ be a braid diagram for a knot $K$, and let $S$ be a singular resolution of $D$. Define $H_{H}(S)$ to be the HOMFLY-PT homology of $S$ and $H_{F}(S)$ to be the homology of the complex assigned to $S$ by the knot Floer cube of resolutions.

\begin{enumerate}[label=(\alph*)] 
\item There is an isomorphism between $H_{F}(S)$ and $H_{H}(S)$.
\item This isomorphism commutes with the induced edge maps in the cube of resolutions.
\end{enumerate}

\end{conj}

Note that Conjecture \ref{ManCon}a $+$ Conjecture \ref{ManCon}b $\implies$ Conjecture \ref{Guk}. In particular, if $E^{F}_{k}(D)$ denotes the $k$th page of the spectral sequence induced by the cube filtration on the knot Floer cube of resolutions and $E^{H}_{k}(D)$ denotes the $k$th page for the HOMFLY-PT cube of resolutions, then 

\[ \text{Conjecture } \ref{ManCon}\text{a} \implies E_{1}^{F}(D) \cong E_{1}^{H}(D) \]
\[ \text{Conjecture } \ref{ManCon}\text{b} \implies E_{2}^{F}(D) \cong E_{2}^{H}(D) \]\vspace{0mm}

\noindent
Since HOMFLY-PT homology is defined to be the $E_{2}$ page $E_{2}^{H}(D)$, this conjecture would prove that $E_{2}^{F}(D)$ is HOMFLY-PT homology and $E_{\infty}^{F}(D)$ is knot Floer homology.

In \cite{Me2}, we proved part (a) of the conjecture, but unfortunately the isomorphism was not canonical. Thus, we were unable to approach part (b). However, in this paper we will show that $E_{2}^{F}(D)$ satisfies the same properties that make HOMFLY-PT homology an interesting theory. 

\begin{thm}

Let $D$ be a braid diagram for a knot $K$. Then $E_{2}^{F}(D)$ is a triply-graded link invariant that categorifies the HOMFLY-PT polynomial. The higher pages of the spectral sequence are all link invariants, and by construction, the $E_{\infty}$ page is knot Floer homology.

\end{thm}

This spectral sequence seems to have a lot in common with Rasmussen's spectral sequences $E_{k}(n)$ in \cite{Rasmussen}, where the $E_{2}$ page is HOMFLY-PT homology, the higher pages are link invariants, and the $E_{\infty}$ is $sl_{n}$ homology. If $E_{2}^{F}(D)$ turns out to be HOMFLY-PT homology, then this spectral sequence would fit into Rasmussen's framework as the $E_{k}(0)$ spectral sequence. 

We show that $E_{2}^{F}(D)$ does have more common with HOMFLY-PT homology than just the graded Euler characteristic. The middle HOMFLY-PT homology $H_{H}(D)$ is known to be a free $\mathbb{F}[U]$ module. We show that the $U$-torsion-free part of $E_{2}^{F}(D)$ is isomorphic to $H_{H}(D)$, where by $U$-torsion we mean the set of $x \in E_{2}^{F}(D)$ such that $U^{k}x=0$ for some $k$.

\begin{thm}

Let $T$ be the $U$-torsion part of $E_{2}^{F}(D)$. Then

\[ E_{2}^{F}(D) \cong H_{H}(D) \oplus T \]

\noindent
This isomorphism preserves two of the three gradings.

\end{thm}

\noindent
Thus, the $E_{2}$ page being torsion-free would prove that it is in fact HOMFLY-PT homology.

\begin{cor}

The reduced theory is given by 

\[ \overline{E_{2}^{F}}(D) \cong \overline{H}_{H}(D) \oplus T' \]

\noindent
where $\overline{H}_{H}(D)$ is the reduced HOMFLY-PT homology and $T'$ is the homology of the mapping cone 

\[T \xrightarrow{U} T \]

\noindent
In particular, there is a rank inequality 

\[ dim(\overline{E_{2}^{F}}(D)) \ge dim( \overline{H}_{H}(D)) \]

\end{cor}

The paper will be organized as follows. Section 2 will give a background on the knot Floer cube of resolutions and some of the tools that will be used in the invariance proof. Section 3 will describe a general type of algebraic proof that can be used to prove invariance for theories with certain local properties, and Section 4 will show that the knot Floer cube of resolutions satisfies these properties. Section 5 proves the relationships with HOMFLY-PT homology.

\section{Background}

The knot Floer cube of resolutions was first defined with twisted coefficients by Ozsv\'{a}th and Szab\'{o} in \cite{Szabo}. It was later modified by Manolescu to give an untwisted version (\hspace{1sp}\cite{Manolescu}), which is the one we will be discussing here.

Let $D$ be a decorated braid diagram for a link $L$. The cube complex can be defined for any diagram with a distinguished edge, but like HOMFLY-PT homology, we will need our diagram to be a braid in order to obtain invariance. We will choose the convention of having our braids oriented from the bottom to the top, with the loops oriented clockwise. An example is given in Figure \ref{rht}.

\begin{figure}[h!]

 \centering
   \begin{overpic}[width=.35\textwidth]{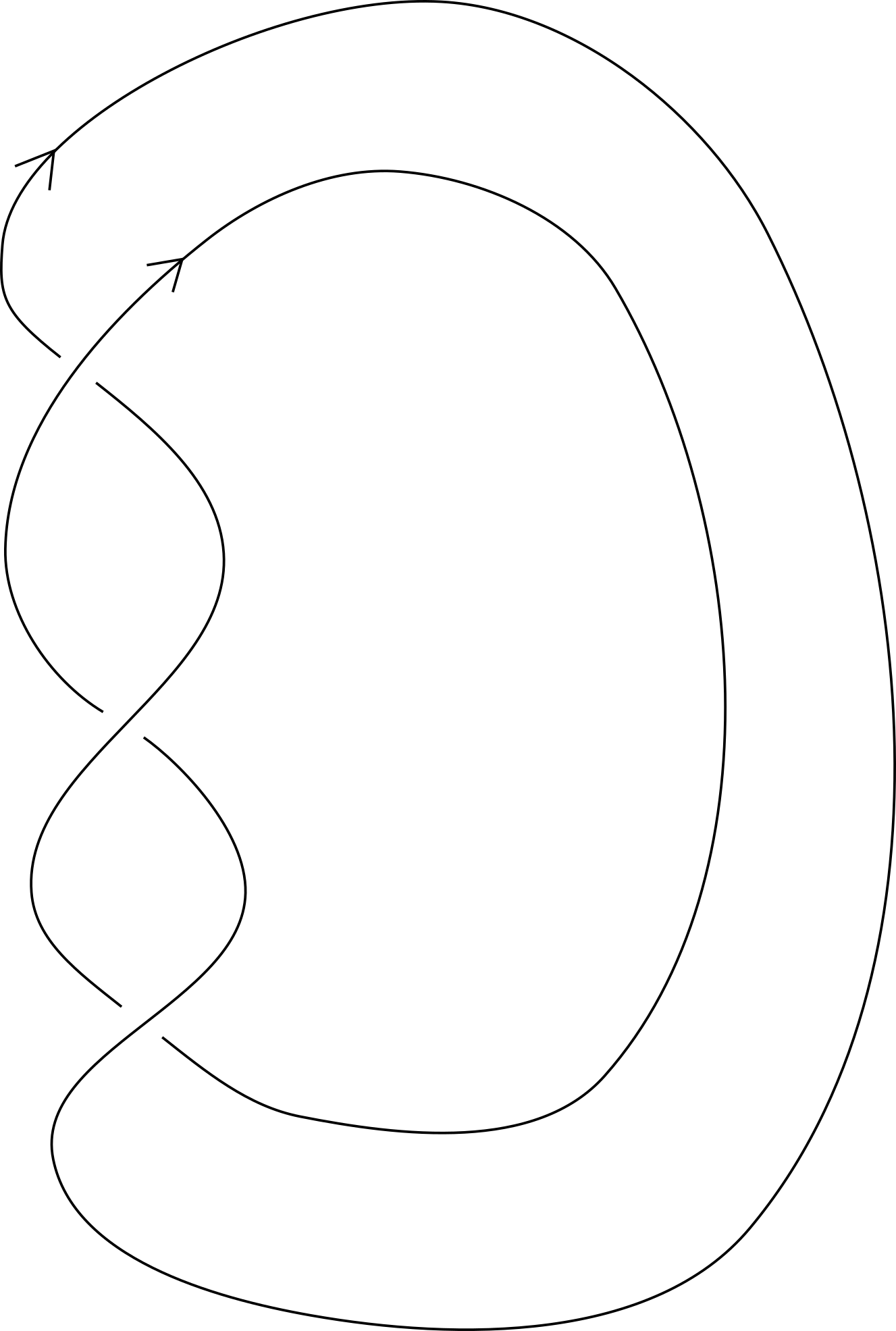}
   \put(-.9,80){$\bullet$}
   \end{overpic}
   \caption{A decorated braid diagram for the right-handed trefoil.}
   \label{rht}
\end{figure}

\subsection{The Knot Floer Cube of Resolutions}

We will assume that the reader is familiar with the construction of knot Floer homology - for resources, see \cite{OS1}, \cite{OS2}, \cite{Rasmussen2}. To simplify our computations, we will work with $\Z_{2}$ coefficients. Working without signs makes sense in this context because HOMFLY-PT homology, unlike $sl_{n}$ homology, gives an invariant for any choice of coefficients \cite{Krasner}. 

\begin{rem}
Technically, Krasner only shows invariance with $\Z$ coefficients. However, since the $E_{1}$ page is known to be torsion-free due to the MOY decomposition, we can apply the universal coefficient theorem to the $E_{2}$ page. This proves that the HOMFLY-PT homology over any ground ring is a well-defined invariant.

\end{rem}

The oriented cube of resolutions is defined using a multi-pointed Heegaard diagram on $S^{2}$ for a knot $K$. We will be using the minus theory $\mathit{HFK}^{-}(K)$, where we count all discs which pass through $O_{i}$ with coefficient $U_{i}$, and we do not count any discs passing through an $X$ basepoint. The ground ring $R$ is therefore given by $\Z_{2}[U_{1}, ..., U_{n}]$. Given a decorated braid projection $D$, we assign the Heegaard diagram in Figure \ref{fig2a} to each positive crossing, and the diagram in Figure \ref{fig2b} to each negative crossing. 

Knot Floer homology is typically defined using just the information in a Heegaard diagram. Unfortunately, there is no known diagram for which the associated chain complex splits as an oriented cube of resolutions. However, using the diagrams in Figure \ref{crossings}, we will construct a larger complex $C_{F}(D)$ which is chain homotopy equivalent to $\mathit{CFK}^{-}(D)$ such that $C_{F}(D)$ does split into a cube of resolutions.

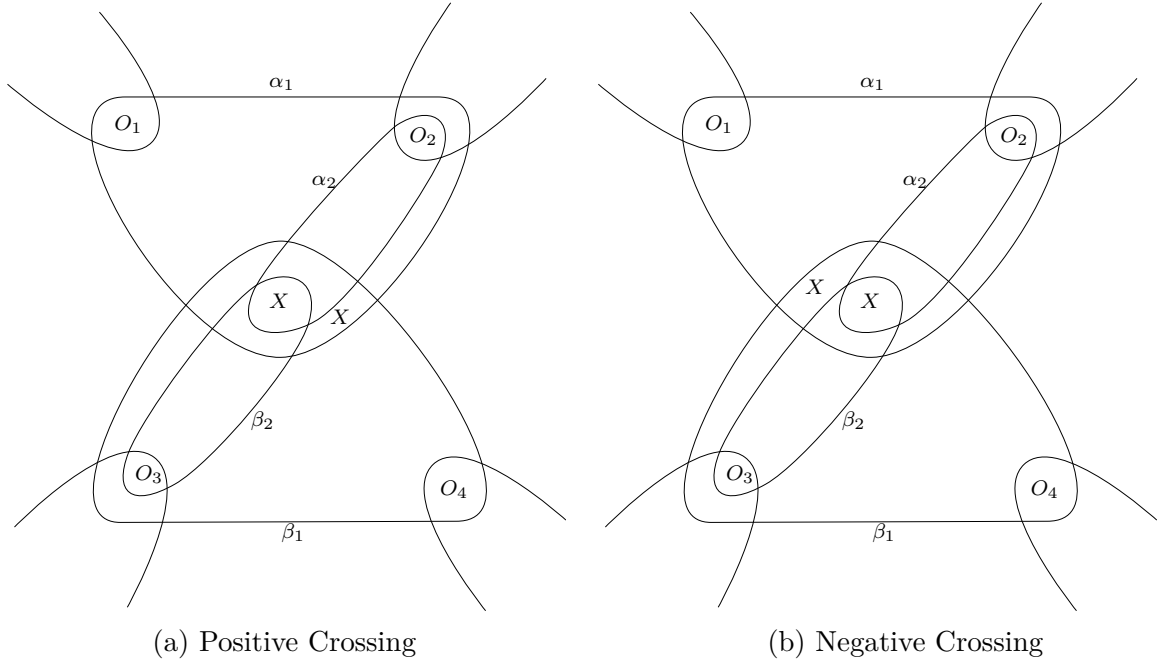
\begin{figure}[h!]
\tiny
\begin{subfigure}{.5\textwidth}
 \centering
   \begin{overpic}[width=.9\textwidth]{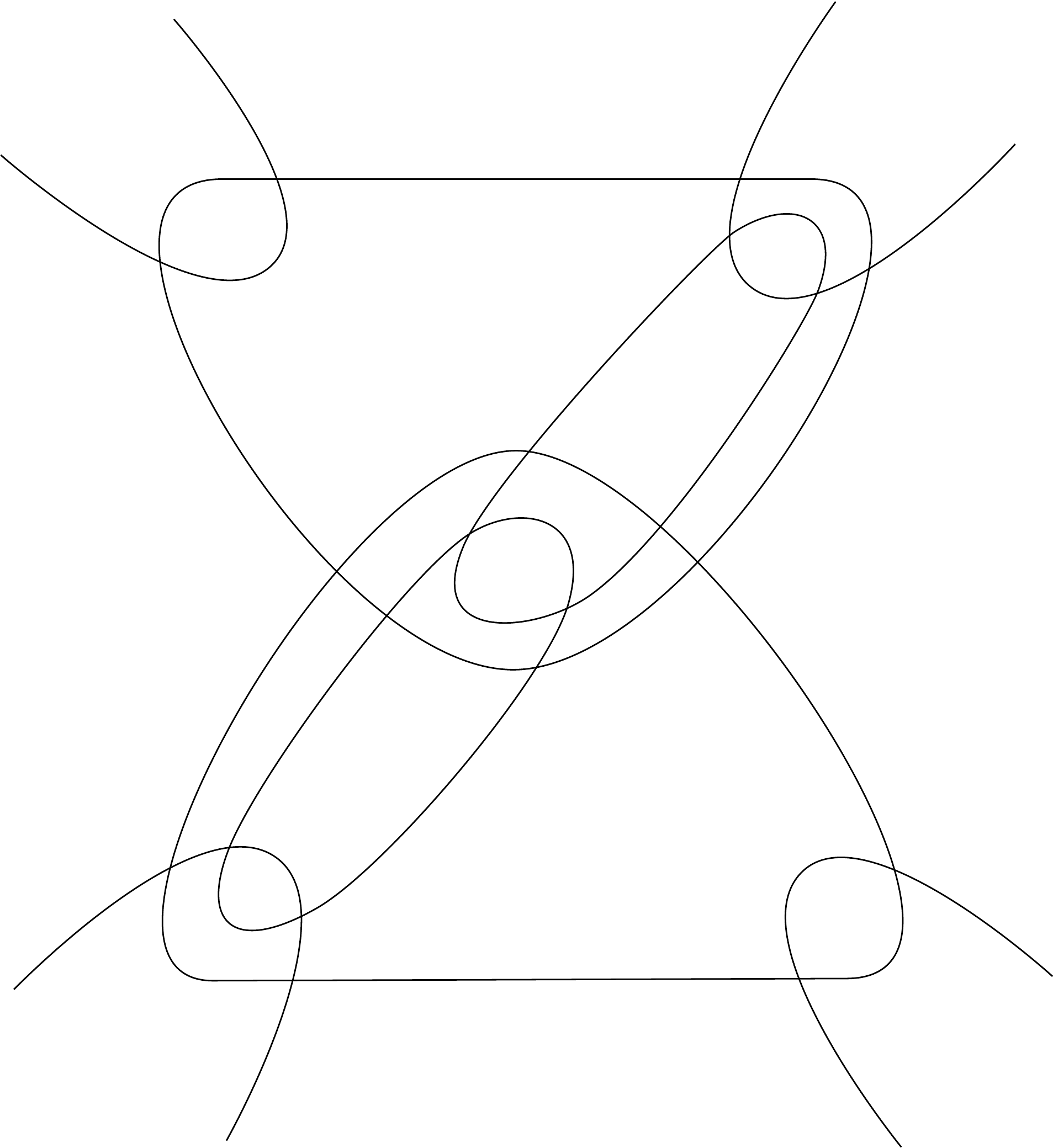}
   \put(17.6,79){$O_{1}$}
   \put(66,77){$O_{2}$}
   \put(43,49.7){$X$}
   \put(53,47.2){$X$}
   \put(21,21.6){$O_{3}$}
   \put(71,19){$O_{4}$}
   \put(43,86){$\alpha_{1}$}
   \put(50,70){$\alpha_{2}$}
   \put(45,12){$\beta_{1}$}
   \put(40,30.4){$\beta_{2}$}
   \end{overpic}
  \caption{Positive Crossing}  \label{fig2a}
\end{subfigure}%
\begin{subfigure}{.5\textwidth}
  \centering
   \begin{overpic}[width=.9\textwidth]{initial_diagram.pdf}
   \put(17.6,79){$O_{1}$}
   \put(66,77){$O_{2}$}
   \put(43,49.7){$X$}
   \put(34,52){$X$}
   \put(21,21.6){$O_{3}$}
   \put(71,19){$O_{4}$}
   \put(43,86){$\alpha_{1}$}
   \put(50,70){$\alpha_{2}$}
   \put(45,12){$\beta_{1}$}
   \put(40,30.4){$\beta_{2}$}
   \end{overpic}
  \hspace{7mm} \caption{Negative Crossing}
  \label{fig2b}
\end{subfigure}
\caption{The Heegaard diagrams at positive and negative crossings}\label{HDCrossing}
\label{crossings}
\end{figure}

Note that these two diagrams use the same $\alpha$ and $\beta$ curves, but differ in the location of one of the $X$ basepoints. The diagram in Figure 3 is thus able to include the data from both the positive and the negative crossings, as well as the oriented smoothing. More specifically, taking the $A_{0}$ and $A^{+}$ labels to denote $X$ basepoints gives a positive crossing, taking $A_{0}$ and $A^{-}$ to denote $X$ basepoints gives a negative crossing, and taking the two $B$ labels to be $X$ basepoints gives the oriented smoothing. Given a link diagram $D$ with a choice of crossing $c$, let $D_{+}$ denote the diagram where that crossing is a positive crossing, $D_{-}$ the diagram where it is a negative crossing, and $D_{s}$ the oriented resolution.

\begin{figure}[h!] 

 \centering
   \begin{overpic}[width=.75\textwidth]{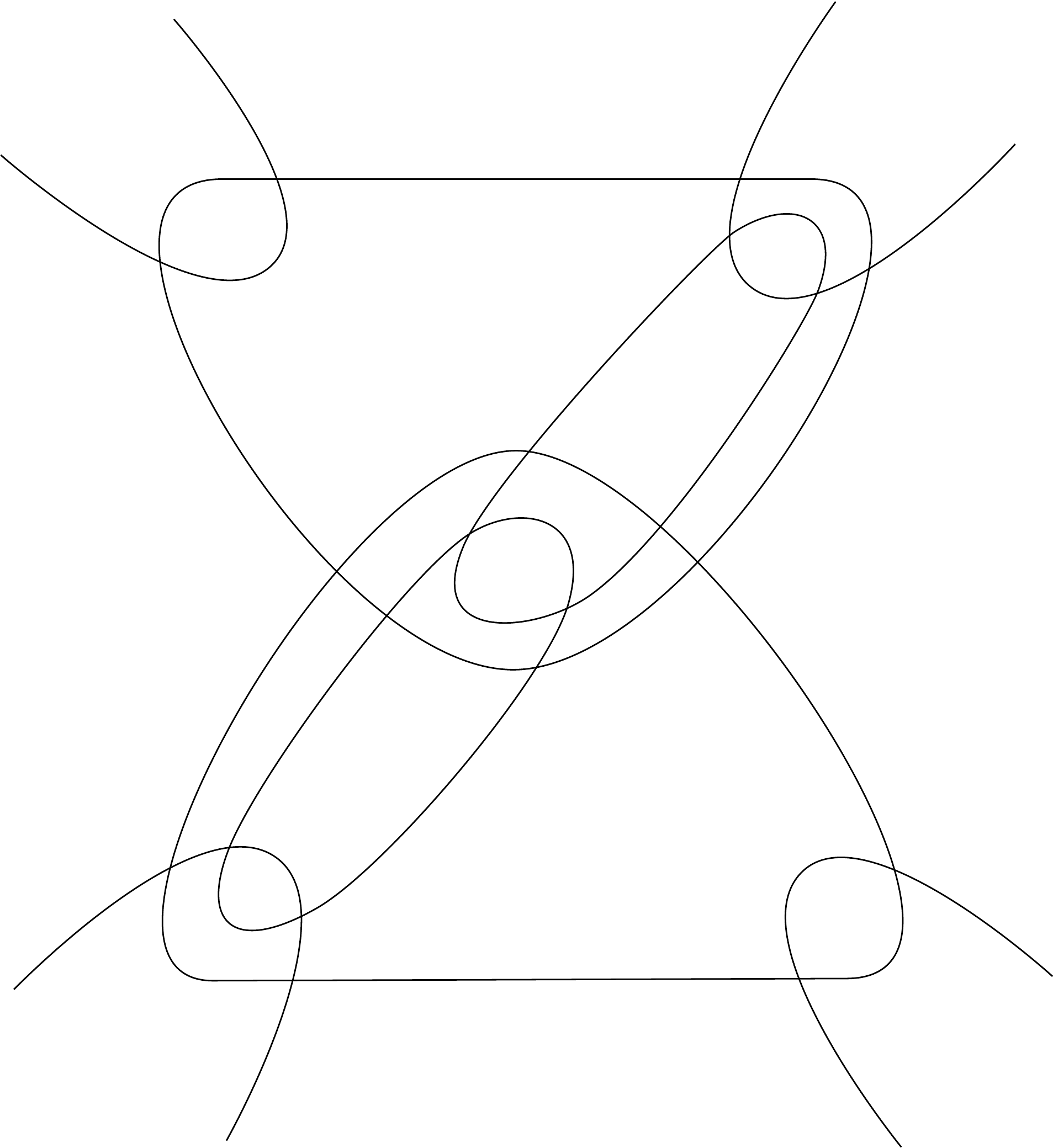}
   \put(39,54){$x$}
   \put(40.4,53){$\bullet$}
   \put(49.3,45){$x'$}
   \put(48.5,46.3){$\bullet$}
   \put(17.6,79){$O_{1}$}
   \put(66,77){$O_{2}$}
   \put(48,55.7){$B$}
   \put(43,49.7){$A_{0}$}
   \put(38.2,44){$B$}
   \put(34,52){$A_{-}$}
   \put(53,48){$A_{+}$}
   \put(21,21.6){$O_{3}$}
   \put(71,19){$O_{4}$}
   \put(43,86){$\alpha_{1}$}
   \put(50,70){$\alpha_{2}$}
   \put(45,12){$\beta_{1}$}
   \put(40,30.4){$\beta_{2}$}
   \end{overpic}
\caption{The Diagram at a Crossing}\label{HDCrossing}
\end{figure}

We will start by defining the cube of resolutions for a negative crossing. Since there are $X$ basepoints at $A_{0}$ and $A^{-}$, the intersection point $x$ becomes special, as it induces a filtration on the complex. Let $X$ denote the component of the knot Floer complex with generators containing the $x$ intersection point, and let $Y$ denote the component of the complex with generators which do not contain $x$. If we let $\Phi_{B}$ denote those differentials with multiplicity $0$ at $A_{0}$ and $A^{-}$ and multiplicity $1$ at one of the $B$'s, then 

\[ \mathit{CFK}^{-}(D_{-}) \cong Y \xrightarrow{\hspace{3mm}\Phi_{B}\hspace{3mm}} X \]

We then define the cube complex $C_{F}(D_{-})$ to be the complex in Figure \ref{negcomplex}, where $\Phi_{A^{-}}$ counts discs with multiplicity $0$ at the $B$ basepoints and $1$ at $A_{0}$ or $A^{-}$ and $\Phi_{A^{-}B}$ counts discs with multiplicity $1$ at one of the $B$ basepoints and multiplicity $1$ at $A_{0}$ or $A^{-}$ (and multiplicity $0$ at the other in both cases). The proof that $d^{2}=0$ is comes from counting Maslov index $2$ degenerations (see \cite{Szabo}, p. 33).

\begin{figure}[!h]
\centering
\begin{tikzpicture}
  \matrix (m) [matrix of math nodes,row sep=5em,column sep=6em,minimum width=2em] {
     X & X \\
     Y & X \\};
  \path[-stealth]
    (m-1-1) edge node [left] {$\Phi_{A^{-}}$} (m-2-1)
            edge node [above] {$1$} (m-1-2)
            edge node [right]{$\Phi_{A^{-}B}$} (m-2-2)
    (m-2-1.east|-m-2-2) edge node [above] {$\Phi_{B}$} (m-2-2)
    (m-1-2) edge node [right] {$U_{1}+U_{2}+U_{3}+U_{4}$ \hspace{ 15mm}} (m-2-2);
\end{tikzpicture}
\caption{Complex for the Negative Crossing} \label{negcomplex}
\end{figure}

Note that $C_{F}(D_{-})$ admits two filtrations - a horizontal filtration and a vertical filtration (the filtrations induced by the horizontal and vertical coordinates, respectively). Using the vertical filtration, we see that 

\[C_{F}(D_{-}) \cong \mathit{CFK}^{-}(D_{-}) \]

\noindent
where the chain homotopy equivalence comes from contracting the isomorphism

\[ X \xrightarrow{\hspace{3mm}1\hspace{3mm}} X \]

The interesting filtration on this complex is the horizontal filtration - this is the one that gives us the cube of resolutions. The complex in the lower filtration level is given by 

\[ Y \xrightarrow{\hspace{3mm}\Phi_{A^{-}}\hspace{3mm}} X \]

\noindent
We define $C_{F}(D_{s})$ to be this complex. Note that this is exactly the complex obtained by placing $X$ basepoints at the $B$ markings, so it gives the knot Floer homology of the oriented smoothing, $\mathit{CFK}^{-}(D_{s})$. 

The complex in the higher filtration level is given by 

\[ X \xrightarrow{\hspace{1mm}U_{1}+U_{2}+U_{3}+U_{4}\hspace{1mm}} X \]

\noindent
Observe that the chain complex $X$ is the complex obtained by deleting $\alpha_{2}$ and $\beta_{2}$ from the Heegaard diagram, giving the diagram in Figure \ref{singdiagram}. This is the Heegaard diagram used to describe the singularization $D_{x}$ (\hspace{1sp}\cite{SingularKnots}). Since the linear term $U_{1}+U_{2}+U_{3}+U_{4}$ depends only on the four edges adjacent to the crossing and not the sign of the crossing, we define 

\[ C_{F}(D_{x}) = X \xrightarrow{\hspace{1mm}U_{1}+U_{2}+U_{3}+U_{4}\hspace{1mm}} X \]

This gives the cube of resolutions decomposition 

\[ C_{F}(D_{-}) = C_{F}(D_{s}) \xrightarrow{\hspace{10mm}} C_{F}(D_{x}) \]

\begin{figure}[!h]
\centering
   \begin{overpic}[width=.7\textwidth]{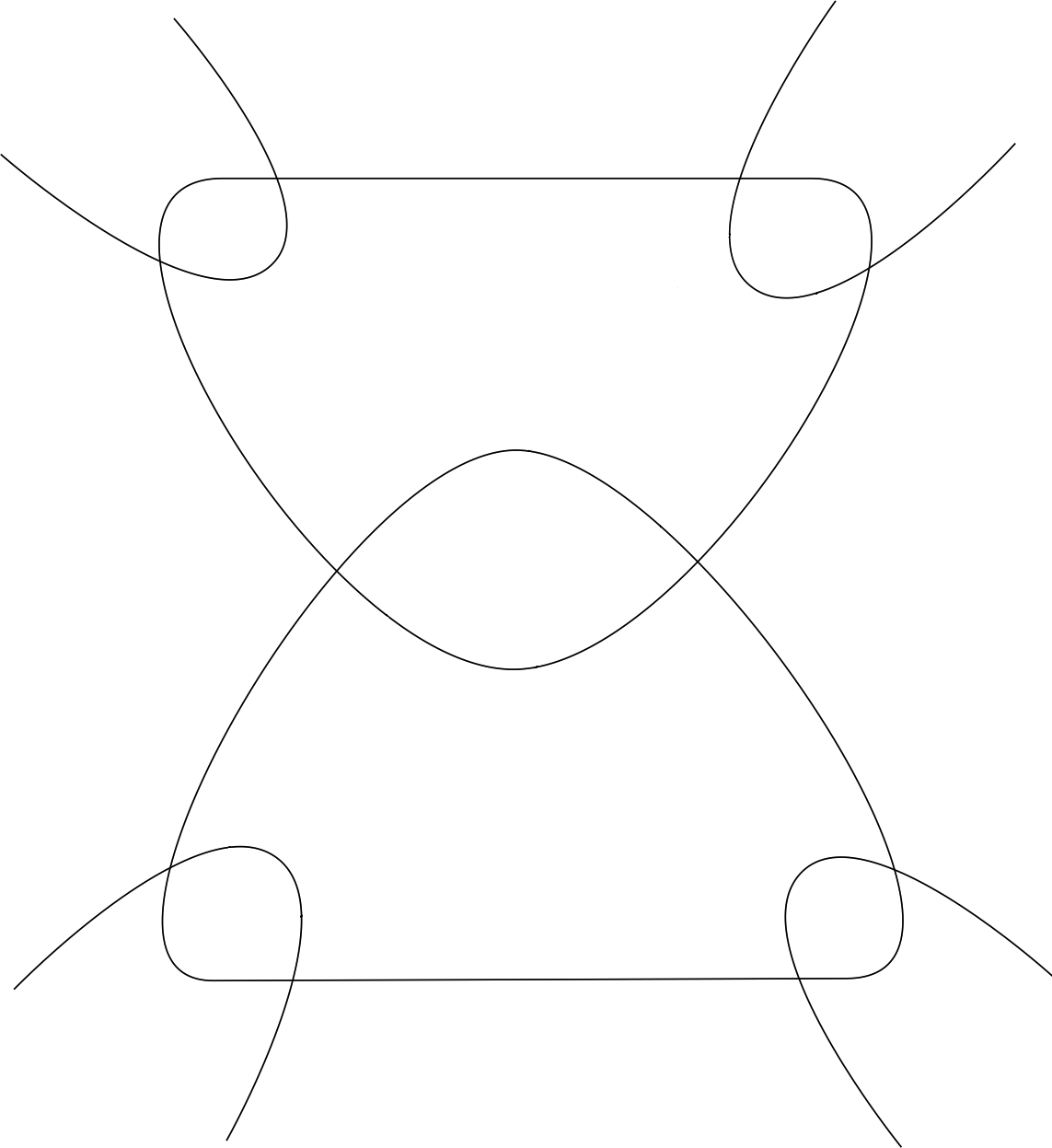} 
   \put(17.6,79){$O_{1}$}
   \put(68,79){$O_{2}$}
   \put(19,19.5){$O_{3}$}
   \put(72,19){$O_{4}$}
   \put(41.5,50){$XX$}
   \put(43,86){$\alpha$}
   \put(1,17){$\alpha$}
   \put(89,18){$\alpha$}
   \put(1,81){$\beta$}
   \put(87,83){$\beta$}
   \put(45,11.5){$\beta$}
   \put(13.2,76.4){$\bullet$}
   \put(23.25,83.7){$\bullet$}
   \put(24,85.8){$a_{2}$}
   \put(10.5,75.2){$a_{1}$}
   \put(63.8,83.7){$\bullet$}
   \put(75,76){$\bullet$}
   \put(61,85.5){$b_{1}$}
   \put(76.5,74.5){$b_{2}$}
   \put(28.55,49.55){$\bullet$}
   \put(25,49.55){$c_{1}$}
   \put(60,50.3){$\bullet$}
   \put(62,50.3){$c_{2}$}
   \put(14.1,23.5){$\bullet$}
   \put(11.35,25.5){$d_{1}$}
   \put(24.6,13.8){$\bullet$}
   \put(25.9,11.34){$d_{2}$}
   \put(68.7,14){$\bullet$}
   \put(66.6,12.3){$e_{1}$}
   \put(77.15,23.3){$\bullet$}
   \put(78.5,25){$e_{2}$}
   \end{overpic}
   \caption{Heegaard diagram for a singularization}\label{singdiagram}
\end{figure}
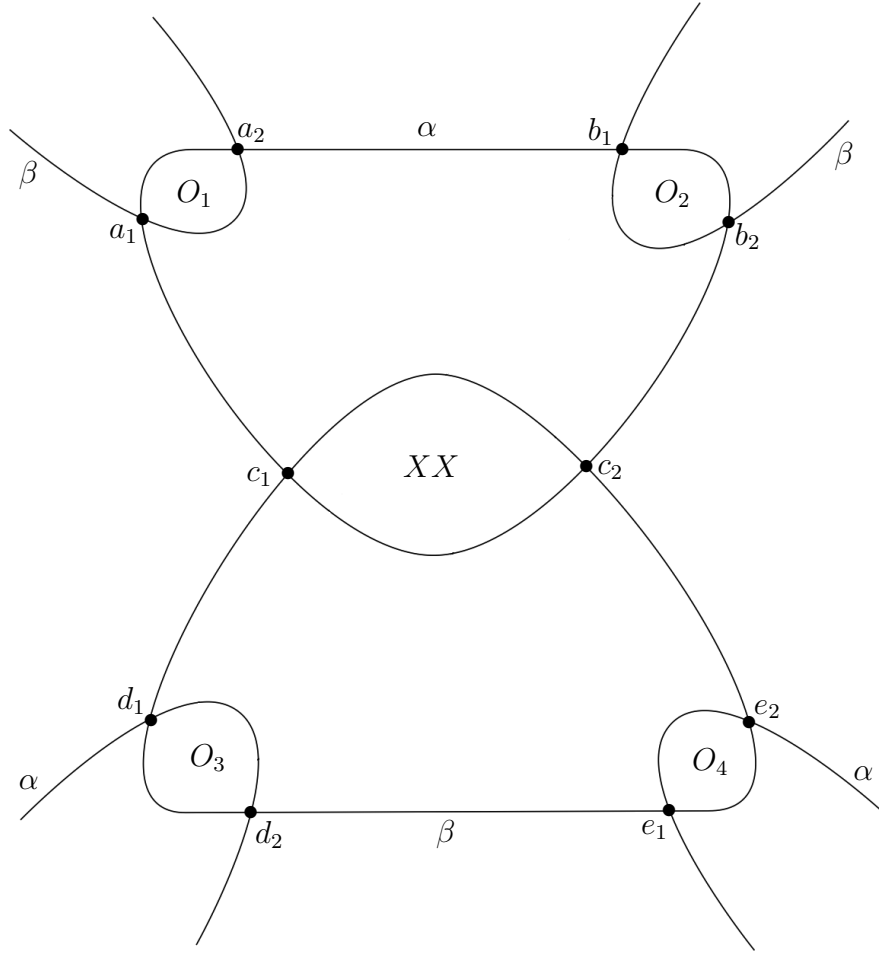

The positive crossing has a similar story, except that the focus is on the intersection point $x'$. Let $X'$ denote the component of the complex with generators containing the $x'$ intersection point, and let $Y'$ denote the the component of the complex with generators that do not contain $x'$. Then we have 

\[ \mathit{CFK}^{-}(D_{+}) \cong X' \xrightarrow{\hspace{3mm}\Phi_{B'}\hspace{3mm}} Y' \]

\noindent
where $\Phi_{B'}$ counts discs with multiplicity $1$ at one of the $B$ markings and $0$ at $A_{0}$ and $A^{+}$. We can now define the cube complex for the positive crossing, $C_{F}(D_{+})$, to be the complex in Figure \ref{poscomplex}. The maps are defined analogously to the positive crossing diagram, where $\Phi_{A^{+}}$ counts discs with multiplicity $0$ at the $B$ basepoints and $1$ at $A_{0}$ or $A^{+}$ and $\Phi_{A^{+}B}$ counts discs with multiplicity $1$ at one of the $B$ basepoints and multiplicity $1$ at $A_{0}$ or $A^{+}$ (and multiplicity $0$ at the other in both cases).

\begin{figure}[!h]
\centering
\begin{tikzpicture}
  \matrix (m) [matrix of math nodes,row sep=5em,column sep=6em,minimum width=2em] {
     X' & Y' \\
     X' & X' \\};
  \path[-stealth]
    (m-1-1) edge node [left] {$U_{1}+U_{2}+U_{3}+U_{4}$} (m-2-1)
            edge node [above] {$\Phi_{B'}$} (m-1-2)
            edge node [right]{$\Phi_{A^{+}B}$} (m-2-2)
    (m-2-1.east|-m-2-2) edge node [above] {1} (m-2-2)
    (m-1-2) edge node [right] {$\Phi_{A^{+}}$ \hspace{ 15mm}} (m-2-2);
\end{tikzpicture}
\caption{Complex for the Positive Crossing} \label{poscomplex}
\end{figure}

Looking at the vertical filtration, we see that 

\[ C_{F}(D_{+}) \cong \mathit{CFK}^{-}(D_{+}) \]

\noindent
The horizontal filtration gives us the cube of resolutions. The complex in the higher filtration level is the knot Floer complex of the oriented smoothing

\[ X' \xrightarrow{\hspace{3mm}\Phi_{A^{+}}\hspace{3mm}} Y' \]

\noindent
which is isomorphic to $C_{F}(D_{s})$. The complex in the higher lower level is 

\[ X' \xrightarrow{\hspace{1mm}U_{1}+U_{2}+U_{3}+U_{4}\hspace{1mm}} X' \]

\noindent
where $X'$, like $X$ in the positive crossing case, is the complex coming from the diagram in Figure \ref{singdiagram}. Thus, the complex in the higher filtration level is precisely $C_{F}(D_{x})$. We therefore have the decomposition

\[ C_{F}(D_{+}) = C_{F}(D_{x}) \xrightarrow{\hspace{10mm}} C_{F}(D_{s}) \]

For a diagram $D$, we define the cube of resolutions complex $C_{F}(D)$ to be the complex obtained by modifying the chain complex as defined above for all of the positive and negative crossings in the diagram. For a positive crossing, we call the singularization the 0-resolution of the crossing, and the smoothing the 1-resolution of the crossing. For a negative crossing, the smoothing is the 0-resolution and the singularization is the 1-resolution.

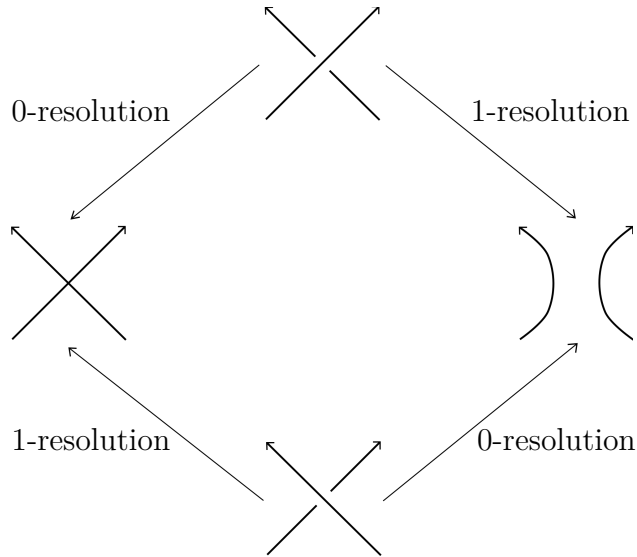
\begin{figure}[!h]
\centering
   \begin{overpic}[width=.5\textwidth]{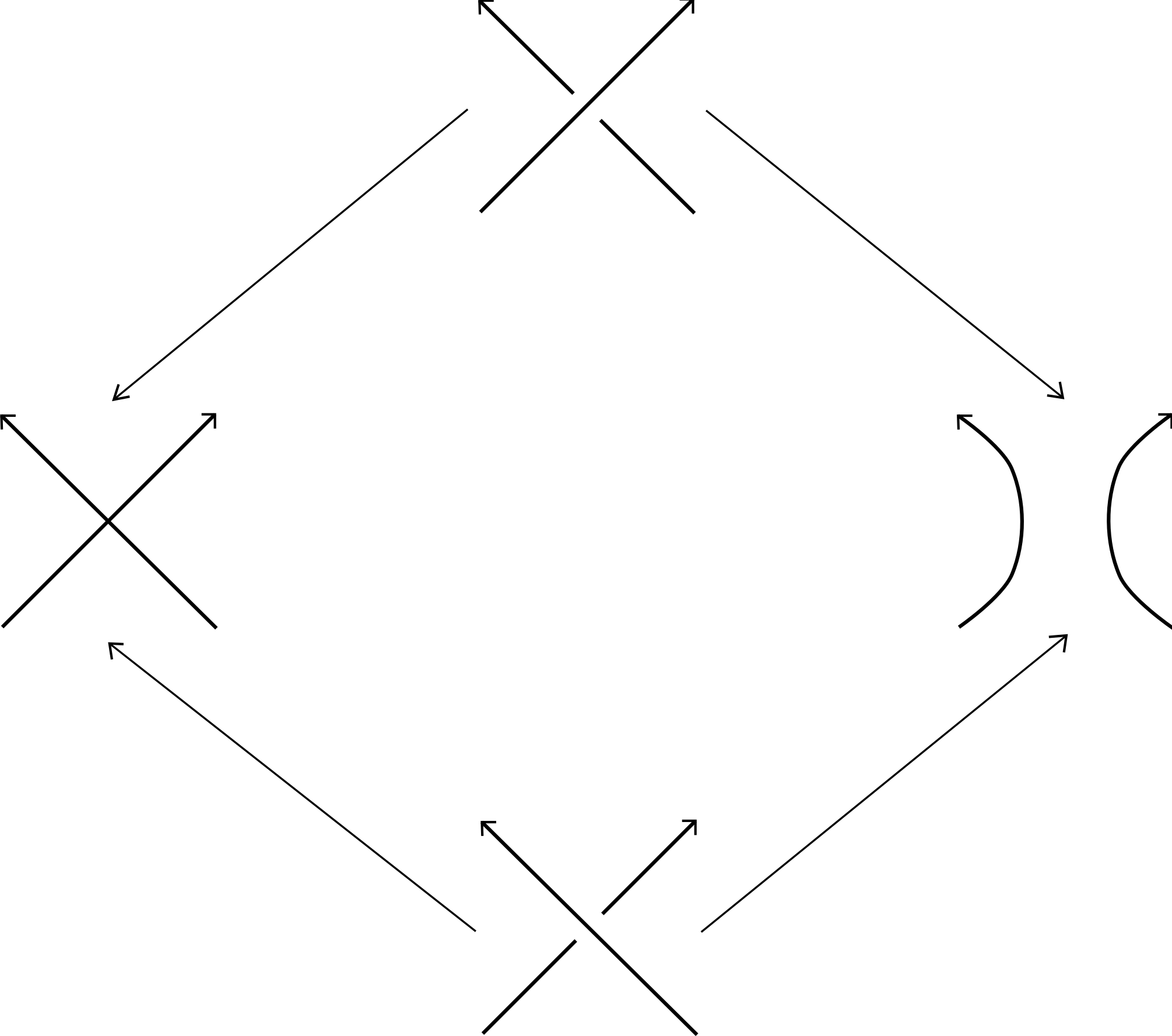} 
   \put(0,17){$1$-resolution}
   \put(75,17){$0$-resolution}
   \put(0,70){$0$-resolution}
   \put(74,70){$1$-resolution}
   \end{overpic}
   \caption{0- and 1-resolutions of positive and negative crossings}\label{resolutions}
\end{figure}

We turn the cube filtration into a grading by using the height in the cube of resolutions. In particular, let $c_{1},...c_{k}$ be the crossings in $D$. A complete resolution is a singular diagram $S_{i_{1}...i_{k}}$ for $i_{j} \in \{0,1\}$, where the $j$-th crossing in $D$ has received the $i_{j}$-resolution. This vertex in the cube of resolutions has cube grading 

\[ gr_{cube}(S_{i_{1}...i_{k}}) = \sum_{j=1}^{k} i_{j} \]

\noindent
Using this grading, the differential can be decomposed as 

\[ d = d_{0}+d_{1}+...+d_{k} \]

\noindent
where $d_{j}$ increases the cube grading by $j$.

\subsection{The Three Gradings} The knot Floer complex comes equipped with a bigrading - the Maslov grading, and the Alexander grading. We will refer to these as $M$ and $A$, respectively. With respect to $(M,A)$, the differential has bigrading $(-1,0)$ and multiplication by $U_{i}$ changes the bigrading by $(-2,-1)$. This bigrading extends to the cube of resolutions complex, but the cube complex also has the cube grading, making it triply graded. 

HOMFLY-PT homology is also triply graded, with gradings $gr_{q}$, $gr_{h}$, and $gr_{v}$ (using the conventions of \cite{Rasmussen}). These gradings are called the quantum grading, the horizontal grading, and the vertical grading, respectively. With respect to this triple grading, the vertex maps $d_{0}$ are homogeneous of degree $(2,2,0)$, the edge maps $d_{1}$ are homogeneous of degree $(0,0,2)$, and $U_{i}$ multiplication is homogeneous of degree $(2,0,0)$.

We can define analogs for these three gradings on $C_{F}(D)$ (which will also be denoted by $gr_{q}$, $gr_{h}$, and $gr_{v}$) as follows:

\[gr_{v} = 2 gr_{cube} -c(D) -b(D)  \]
\[ gr_{q} = -2M+2A - gr_{v}\] \[gr_{h} = -2M+4A -gr_{v} \] 

\noindent
where $c(D)$ is the number of crossings in $D$ and $b(D)$ is the number of strands in the braid $D$. With respect to this triple grading, the vertex maps, edge maps, and $U_{i}$ multiplication behave in the same way as in the HOMFLY-PT complex.

\subsection{The Decorated Edge and Insertions}

If we view our diagram as a 4-valent graph with additional information at the crossings, we see that the $O_{i}$ basepoints are in bijection with edges of the graph. It is sometimes convenient to perform (0,3) stabilization at such an $O_{i}$, as in Figure \ref{HDMarkedEdge}. To maintain the bijection between the $O_{i}$ basepoints and edges in the diagram $D$, we add in a bivalent vertex in the diagram $D$ whenever perform one of these stabilizations. Stabilizations of this type are called \emph{insertions}.

\begin{figure}[h!]

 \centering
   \begin{overpic}[width=.6\textwidth]{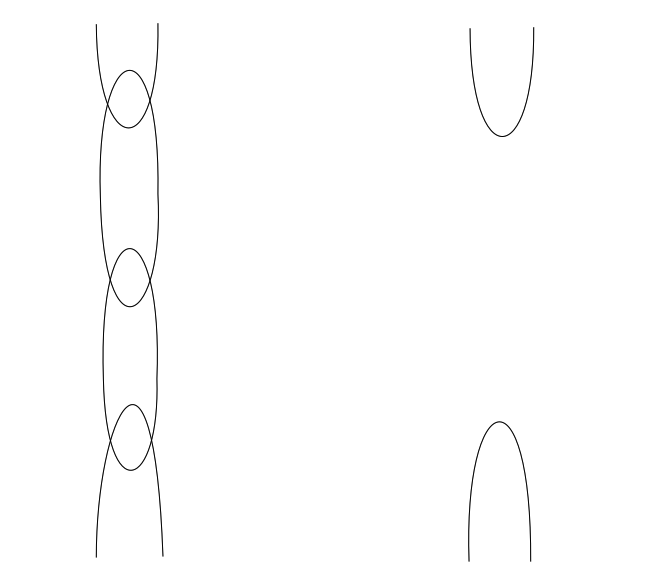}\small
   \put(17.8,72.5){$O_{1}$}
   \put(75,71){$O_{1}$}
   \put(18.1,44.7){$X$}
   \put(75,44.7){$X$}
   \put(17.9,20){$O_{2}$}
   \put(74.6,18.5){$O_{2}$}
   \end{overpic}
\caption{The diagram at a typical insertion (left) and the marked insertion (right)}\label{HDMarkedEdge}
\end{figure}
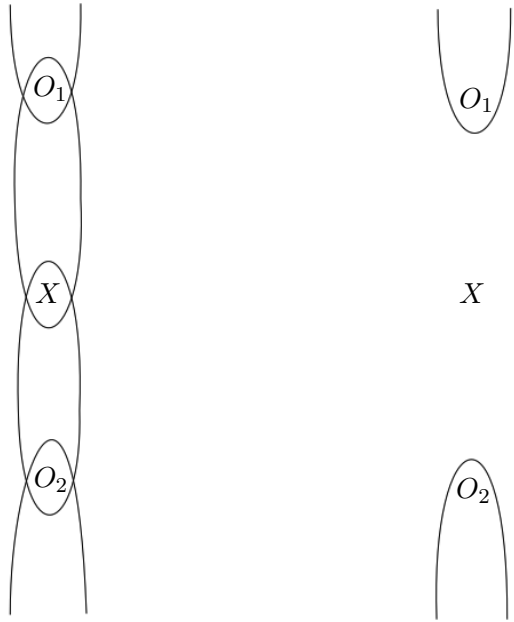

One aspect of the Heegaard diagram that we have neglected to mention thus far is its dependence on the marking in the decorated braid diagram. If we were to draw our diagram ignoring this data, we would end up with one too many $\alpha$ and $\beta$ circles for a balanced Heegaard diagram. Therefore, our marked edge on the diagram $D$ specifies an edge at which we add an insertion, but leave out the $\alpha$ and $\beta$ circles, giving a well-defined Heegaard diagram. 

The dependence on this marked edge is one of the more confusing aspects of the theory. While its location does not impact the chain homotopy type of the complex, the chain homotopy equivalence maps are difficult to compute even when the complexes are quite simple.

In \cite{Manolescu} and \cite{Szabo}, the decorated edge was always assumed to be leftmost in the braid. When we are using this convention, we will write the complex as $C_{F}(D)$. However, we can also further simplify things by adding an extra unknotted component to the left of the braid and placing the decorated edge on this component. This is equivalent to just placing an $X$ and $O_{i}$ basepoint to the left of the diagram. This construction removes any dependence on the marked edge from our computations. We define $\widetilde{C}_{F}(D)$ to be the complex obtained in this way with the new $U_{i}$ variable set to $0$. 

In analogy with the three versions of HOMFLY-PT homology, we will call $\widetilde{C}_{F}(D)$ the \emph{unreduced} complex and $C_{F}(D)$ the \emph{middle} complex. We define the \emph{reduced} complex $\overline{C}_{F}(D)$ by setting one of the $U_{i}$ variables in $C_{F}(D)$ equal to zero.

\subsection{Generators and Cycles}

We define a subset of edges $Z$ in our diagram $D$ to be a \emph{cycle} if they form a homological cycle in the underlying oriented graph for $D$. In other words, at each vertex in $D$ (both bivalent and 4-valent) the number of incoming edges in $Z$ and the number of outgoing edges in $Z$ are the same.

There is a way of assigning a cycle $Z$ to each generator in our complex $\mathit{CFK}^{-}(D)$. Whenever a generator contains an intersection point on one of the $\alpha$ or $\beta$ arcs in Figure \ref{labeled diagram} ($\alpha_{3}, \alpha_{4}, \beta_{3}$, or $\beta_{4}$), we assign the edge corresponding to the $O_{i}$ contained in that arc. Defining $Z_{i_{1}...i_{k}}$ to be the local cycle in a tangle $T$ containing edges $e_{i_{1}},...e_{i_{k}}$, a full description of which generators correspond to which local cycles in Figure \ref{table}. This association of cycles to generators gives a very powerful framework in which to make computations, as we will see in later sections.

\begin{figure}[h!]

 \centering
   \begin{overpic}[width=.75\textwidth]{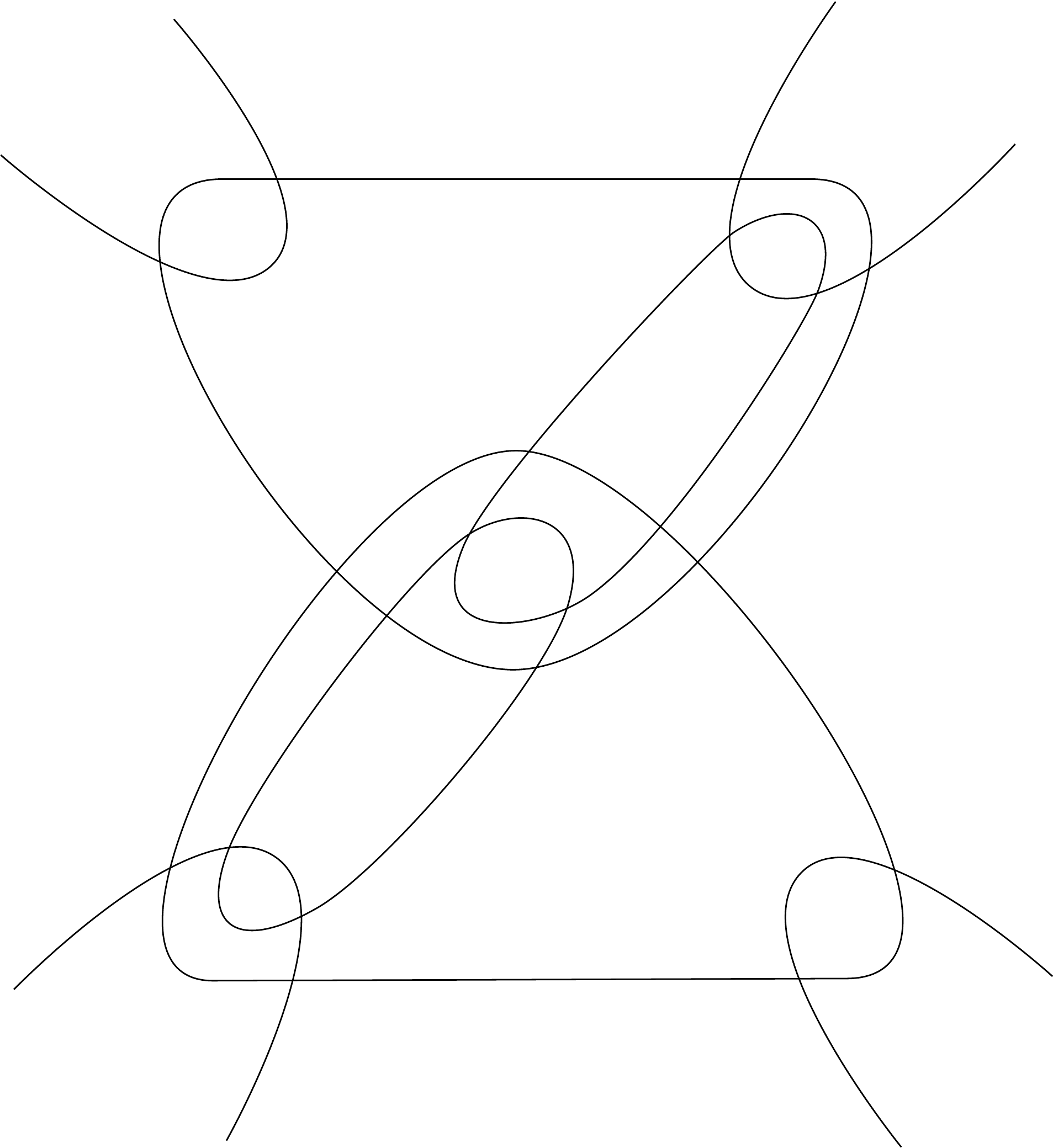}
   \put(0, 82){$\beta_{3}$}
   \put(86.5,83){$\beta_{4}$}
   \put(0, 17){$\alpha_{3}$}
   \put(88,19){$\alpha_{4}$}
   \put(38.6,54.6){$e_{1}$}
   \put(40.4,52.9){$\bullet$}
   \put(49.3,45){$e_{2}$}
   \put(48.5,46.3){$\bullet$}
   \put(17.6,79){$O_{1}$}
   \put(66,77){$O_{2}$}
   \put(48,55.7){$B$}
   \put(43,49.7){$A_{0}$}
   \put(38.2,44){$B$}
   \put(34,52){$A_{-}$}
   \put(53,48){$A_{+}$}
   \put(21,21.6){$O_{3}$}
   \put(71,19){$O_{4}$}
   \put(43,86){$\alpha_{1}$}
   \put(50,70){$\alpha_{2}$}
   \put(45,12){$\beta_{1}$}
   \put(40,30.4){$\beta_{2}$}
   \put(13.2,76.4){$\bullet$}
   \put(23.25,83.7){$\bullet$}
   \put(24,85.5){$c_{2}$}
   \put(11.5,75.2){$c_{1}$}
   \put(63.8,83.7){$\bullet$}
   \put(75,76){$\bullet$}
   \put(61.5,85.5){$a_{1}$}
   \put(76.5,75.5){$a_{2}$}
   \put(60.4,79.9){$b_{1}$}
   \put(71.2,71.8){$b_{2}$}
   \put(62.7,78.7){$\bullet$}
   \put(70.2,73.8){$\bullet$}
   \put(44,62.1){$d_{1}$}
   \put(58.5,53.6){$d_{2}$}
   \put(45.4,59.95){$\bullet$}
   \put(56.8,53.45){$\bullet$}
   \put(28.55,49.55){$\bullet$}
   \put(25.5,49.55){$f_{1}$}
   \put(60,50.3){$\bullet$}
   \put(62,50.3){$f_{2}$}
   \put(33,45.5){$\bullet$}
   \put(30,45.65){$g_{1}$}
   \put(46,41.2){$\bullet$}
   \put(47.2,40){$g_{2}$}
   \put(19.25,25.4){$\bullet$}
   \put(17,27.2){$h_{1}$}
   \put(25.4,19.55){$\bullet$}
   \put(26.5,18){$h_{2}$}
   \put(14.1,23.5){$\bullet$}
   \put(12.35,25.1){$i_{1}$}
   \put(24.8,14){$\bullet$}
   \put(25.9,12.34){$i_{2}$}
   \put(68.8,14.1){$\bullet$}
   \put(67,12.6){$j_{1}$}
   \put(77.15,23.3){$\bullet$}
   \put(78.5,25){$j_{2}$}
   \end{overpic}
   \caption{The Labeled Diagram}
   \label{labeled diagram}
\end{figure}

\begin{figure}
\begin{center}
  \begin{tabular}{ l | c }

    Cycles & Generators  \\ \hline
    $Z_{\phi}$ & $(d,g) \text{ and } (e,f)$ \\ \hline
    $Z_{13}$ & $(c,d,h) \text{ and } (c,e,i)$ \\ \hline
    $Z_{24}$ & $(a,e,j)\text{ and } (b,g,j)$  \\ \hline
    $Z_{23}$ & $(a,e,i) \text{, } (a,d,h)\text{, } (b,f,h) \text{, and } (b,g,i)$  \\ \hline
    $Z_{14}$ & $(c,e,j)$  \\ \hline
    $Z_{1234}$ & $(b,c,h,j)$  \\ 
    \hline
  \end{tabular}
\end{center}
\caption{Generators corresponding to each local cycle} \label{table}
\end{figure}

\subsection{The Basepoint Filtration on $\widetilde{C}_{F}(D)$}

Our choice of Heegaard diagram allows us to place a basepoint filtration on the complex $\mathit{CFK}^{-}(D)$. We add basepoints $p_{i}$ to the Heegaard diagram as in Figure \ref{markings} - note that these basepoints are in bijection with regions in the complement of $D$ in $S^{2}$.

\begin{figure}[h!]
 \centering
   \begin{overpic}[width=.75\textwidth]{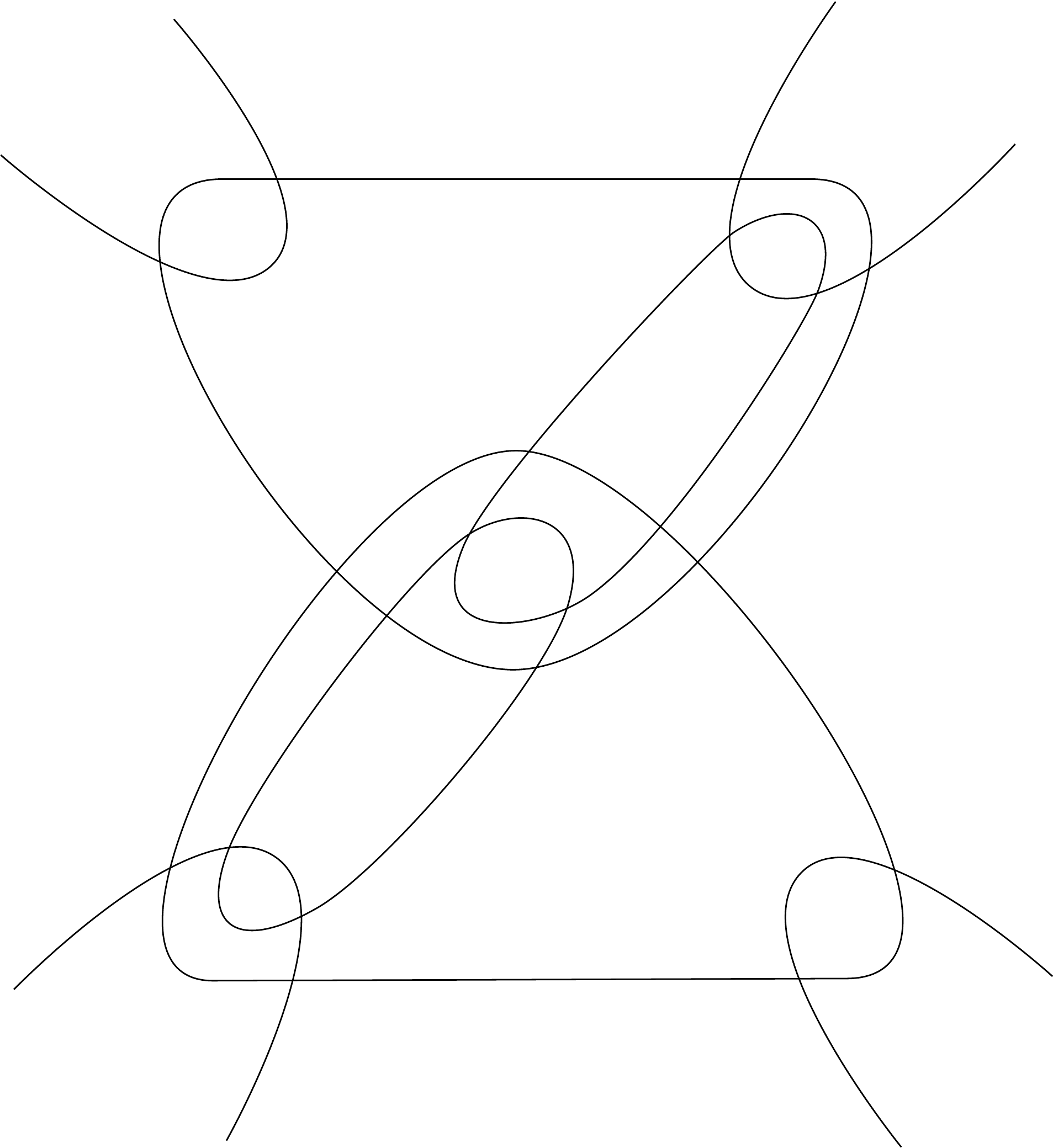}
   \put(39,54){$x$}
   \put(40.4,53){$\bullet$}
   \put(49.3,45){$x'$}
   \put(48.5,46.3){$\bullet$}
   \put(17.6,79){$O_{1}$}
   \put(66,77){$O_{2}$}
   \put(48,55.7){$B$}
   \put(43,49.7){$A_{0}$}
   \put(38.2,44){$B$}
   \put(34,52){$A_{-}$}
   \put(53,48){$A_{+}$}
   \put(21,21.6){$O_{3}$}
   \put(71,19){$O_{4}$}
   \put(43,86){$\alpha_{1}$}
   \put(50,70){$\alpha_{2}$}
   \put(45,12){$\beta_{1}$}
   \put(40,30.4){$\beta_{2}$}
   \put(13,48){$\bullet$}
   \put(80,50){$\bullet$}
   \put(47,5){$\bullet$}
   \put(44,95){$\bullet$}
   \put(10,49){$p_{2}$}
   \put(77,51){$p_{3}$}
   \put(44,6){$p_{4}$}   
   \put(41,96){$p_{1}$}
   \end{overpic}
   \caption{Additional basepoints} \label{markings}
\end{figure}

We will list a few lemmas regarding these basepoints, proved in \cite{Me}, Section 3.

\begin{lem}

These markings define a filtration on the complex $CFK^{-}(D)$, where the change in filtration level of a differential is given by the sum of the multiplicities of the corresponding holomorphic disc at these basepoints. This filtration does not depend on the location of the X's in the interior regions, so it extends to a filtration on $\widetilde{C}_{F}(D)$.

\end{lem}

Let $\widetilde{C}_{F}(D,Z)$ denote the complex generated by the elements corresponding to the cycle $Z$.

\begin{lem}

Let $d^{f}$ denote the component of the differential which preserves the basepoint filtration (i.e. discs which do not pass through the $p_{i}$). The differential $d^{f}$ preserves $\widetilde{C}_{F}(D,Z)$, i.e. it does not change the underlying cycle of a generator.

\end{lem}

\subsection{The Homology of a Fully Singular Diagram}

Since we are interested in studying the spectral sequence on $\widetilde{C}_{F}(D)$ induced by the cube filtration, the first thing we'll need to understand is the $E_{1}$ page. If $S_{v}$ is the fully singular braid at the vertex $v$ in the cube of resolutions, then 

\[ (\widetilde{C}_{F}(D), d_{0}) = \bigoplus_{v} \widetilde{C}_{F}(S_{v}) \]

We computed the homology $\widetilde{C}_{F}(S)$ for a singular braid $S$ in \cite{Me2}, and we will summarize the results here. Since the basepoint filtered differential splits over the cycles, we have that 

\[ H_{*}(\widetilde{C}_{F}(S), d^{f}) = \bigoplus_{Z} H_{*}(\widetilde{C}_{F}(S,Z)) \]

Let $V_{2}(S)$ and $V_{4}(S)$ denote the set of 2-valent and 4-valent vertices in $S$. The set of cycles $Z$ which have generators in $\widetilde{C}_{F}(S)$ are precisely those cycles which do not have include all four edges adjacent to any vertex in $V_{4}$, and which do not pass through the decorated bivalent vertex.

Given any such cycle, let $W_{2}(Z)$ denote the set of bivalent vertices which are not the endpoint of an edge in $Z$, and $W_{4}(Z)$ the set of 4-valent vertices which are not the endpoint of an edge in $Z$. 

Given a vertex $v$, define $L(v)$ to be the sum of all edges adjacent to $v$, and for $v \in V_{4}(S)$, let $Q(v)$ denote the product of the two incoming edges plus the product of the two outgoing edges. 

The total complex for $\widetilde{C}_{F}(S,Z)$ is given by

\[ \widetilde{C}_{F}(S,Z) = \Big[ \bigotimes_{e_{i} \in Z} R \xrightarrow{\hspace{0mm}U_{i}\hspace{0mm}} R \Big] \otimes \Big[ \bigotimes_{v \in W_{2}(Z)} R \xrightarrow{\hspace{-.25mm}L(v)\hspace{-.25mm}} R \Big] \otimes \Big[ \bigotimes_{v \in W_{4}(Z)} R \xrightarrow{\hspace{-.25mm}Q(v)\hspace{-.25mm}} R \Big] \otimes \Big[ \bigotimes_{v \in V_{4}(S)} R \xrightarrow{\hspace{-.25mm}L(v)\hspace{-.25mm}} R \Big]  \]

\noindent
This can be seen by counting bigons in Figures \ref{singdiagram} and \ref{HDMarkedEdge}.

\begin{lem}\label{bigthm}

Let $\widetilde{H}_{H}(D)$ denote the unreduced HOMFLY-PT homology of $D$. The filtered homology $H_{*}(\widetilde{C}_{F}(S,Z),d_{0})$ is isomorphic to $\widetilde{H}_{H}(S-Z)$.

\end{lem}

This tells us that the basepoint-filtered homology of $\widetilde{C}_{F}(S)$ decomposes as a direct sum

\[ H_{*}(\widetilde{C}_{F}(S), d^{f}) \cong \bigoplus_{Z}\widetilde{H}_{H}(S-Z) \]

\noindent
The HOMFLY-PT homology of a singular link is bigraded, with bigrading $(gr_{q}, gr_{h})$ (because the cube grading is fixed). Using the quantum and horizontal gradings on $\widetilde{C}_{F}(S)$, this can be made into an isomorphism of bigraded vector spaces:

\[ H_{*}(\widetilde{C}_{F}(S), d^{f}) \cong \bigoplus_{Z}\widetilde{H}_{H}(S-Z)\{q(Z), a(Z)\} \]

\noindent
where the grading shifts $q(Z)$ and $a(Z)$ depend on combinatorial data from $Z$.

The following two theorems are the main results from \cite{Me2}:

\begin{thm}\label{thm2.4} With coefficients in $\Q$, there is an isomorphism of bigraded vector spaces 

\[ \bigoplus_{Z}\widetilde{H}_{H}(S-Z)\{q(Z), a(Z)\}  \cong \widetilde{H}_{H}(S)\langle1\rangle \]

\noindent
where $\langle1\rangle$ denotes a grading shift $(gr_{q}, gr_{h}) \mapsto (gr_{q}+kgr_{h}, gr_{h})$.

\end{thm}

\begin{thm} \label{thm2.5}

With coefficients in $\Q$, the spectral sequence on $\widetilde{C}_{F}(S)$ induced by the basepoint filtration collapses at the $E_{1}$ page.

\[ H_{*}(\widetilde{C}_{F}(S), d) \cong H_{*}(\widetilde{C}_{F}(S), d^{f}) \]

\end{thm}

We need these theorems to hold over $\Z_{2}$ as well. When computing HOMFLY-PT homology, the only complexes that are contracted have coefficient $\pm1$ (reference), so the rank of the homology of  $\widetilde{H}_{H}(S)$ in each bigrading is independent of the ground field. This proves that Theorem \ref{thm2.4} holds over $\Z_{2}$. 

Since the only complexes that were contracted when calculating $H_{*}(\widetilde{C}_{F}(S), d^{f})$ over $\Q$ had coefficient $\pm1$, the fact that all higher differentials are $0$ over $\Q$ implies that all higher differentials are also $0$ over $\Z_{2}$, proving that Theorem \ref{thm2.5} also holds over $\Z_{2}$.

These theorems give us some powerful information that will come in handy when calculating $E_{2}(\widetilde{C}_{F}(D))$. In particular, we see that the induced edge map $d_{1}^{*}$ can be filtered by the basepoint filtration, so long as we compute $E_{1}(\widetilde{C}_{F}(D))$ by canceling filtered differentials first. They also allow us to compute the graded Euler characteristic, which will be done in the next section. 

\subsection{The Graded Euler Characteristic} \label{Euler} We define the HOMFLY-PT polynomial via the skein relation

\[ aP_{H}(a,q,D_{+}) -a^{-1}P_{H}(a,q,D_{-})= (q-q^{-1})P_{H}(a,q,D_{s}) \]

\noindent
normalized to $P_{H}(unknot)=\frac{a-a^{-1}}{q-q^{-1}}$. 

Using $(gr_{q}, gr_{h}, gr_{v})$ as the triple grading on $\widetilde{H}_{H}(K)$, Rasmussen showed that

\[
 \sum_{i,j,k} (-1)^{(k-j)/2}q^{i}a^{j} dim(\widetilde{H}_{H}^{i,j,k}(D)) = P_{H}(a,q,m(K))
\]

\noindent
where $m(D)$ denotes the mirror of the knot $D$. 

\begin{rem}
Rasmussen doesn't present this in terms of the mirror of the knot $K$, but rather defines a different version of the HOMFLY-PT polynomial, which we will denote $P_{H}^{m}$, via the skein relation 

\[ aP^{m}_{H}(a,q,D_{-}) -a^{-1}P^{m}_{H}(a,q,D_{+})= (q-q^{-1})P^{m}_{H}(a,q,D_{s}) \]

\end{rem}

Using the triple grading $(gr_{q}, gr_{h}, gr_{v})$ on $E_{2}(\widetilde{C}_{F}(D))$, we see that the Poincare polynomial of $\widetilde{H}_{F}(S)$ is the same as that of $\widetilde{H}_{H}(S)$ with $aq$ substituted for $a$. In both the HOMFLY-PT complex and the knot Floer cube of resolutions, the edge maps have triple grading $(0,0,2)$, so the sum is alternating in the cube grading. It follows that

\[
 \sum_{i,j,k} (-1)^{(k-j)/2}q^{i}a^{j} dim(E_{2}(\widetilde{C}_{F}(D))^{i,j,k}(D)) = P_{H}(aq,q, m(D))
\]

\noindent
For a more detailed exposition on the graded Euler characteristic, see \cite{Me}, Section 4.

Although this proves that the graded Euler characteristic gives a version of the HOMFLY-PT polynomial, it is preferable to have a triple grading on $E_{2}(\widetilde{C}_{F}(D))$ that gives the standard HOMFLY-PT polynomial. The HOMFLY-PT polynomial for a knot and it's mirror are related by the following formula:

\[ P_{H}(a,q, m(D)) = P_{H}(a^{-1}, q^{-1}, D) \]

\noindent
Additionally, the modification $(gr_{q}, gr_{h}) \mapsto (gr_{q}-gr_{h}, gr_{h})$ maps $aq \mapsto a$ on the level of polynomials. Thus, if we redefine the triple grading to be $(-gr_{q}+gr_{h}, -gr_{h}, gr_{v})$, then  

\[
 \sum_{i,j,k} (-1)^{(k-j)/2}q^{i}a^{j} dim(E_{2}^{i,j,k}(\widetilde{C}_{F}(D))) = P_{H}(a,q,D)
\]

\subsection{The Filtered Edge Maps} \label{filterededgemaps}

Some of the local cycles only appear in one of the two resolutions of a crossing, so they won't have any edge maps. Using the notation from Figure \ref{table}, the cycles $Z_{14}$ and $Z_{23}$ will only have non-trivial homology at the singularization, and $Z_{1234}$ will only have non-trivial homology at the smoothing. However, the cycles $Z_{\emptyset}$, $Z_{13}$, and $Z_{24}$ appear non-trivially in both the singularization and the smoothing. The filtered edge maps for these cycles are computed in detail in \cite{Me}, and we will summarize the results here.

Let $S_{1}$ and $S_{2}$ be two resolutions of a diagram $D$ which only differ in how a single positive crossing is resolved, with $S_{1}$ the singularization and $S_{2}$ the smoothing. For the empty cycle $Z_{\emptyset}$, we want to understand the map 

\[ d_{1}^{f}: \widetilde{H}_{H}(S_{1}-Z_{\emptyset}) \to \widetilde{H}_{H}(S_{2}-Z_{\emptyset}) \]

Since $S_{1}-Z_{\emptyset}$ and $S_{2}-Z_{\emptyset}$ differ at a single crossing where the former is singularized and the latter is smoothed, there is an obvious candidate for this map: the standard HOMFLY-PT edge map for a positive crossing, which we will denote $\Delta_{+}$. In \cite{Me}, we compute $d_{1}^{f}$ directly and prove that this is in fact the edge map for the empty cycle.

For the cycle $Z_{13}$, we see by inspection that $S_{1}-Z_{13}$ and $S_{2}-Z_{13}$ are actually the same singular diagram, so there is a canonical isomorphism between $\widetilde{H}_{H}(S_{1}-Z_{13})$ and $\widetilde{H}_{H}(S_{2}-Z_{13})$. The map 

\[ d_{1}^{f}: \widetilde{H}_{H}(S_{1}-Z_{13}) \to \widetilde{H}_{H}(S_{2}-Z_{13}) \]

\noindent
is this isomorphism.

The cycle $Z_{24}$ is similar, in that $S_{1}-Z_{24}$ and $S_{2}-Z_{24}$ are the same singular diagram. However, the edge map in this case is not an isomorphism, but instead it is multiplication by $U_{1}$.

When the crossing is negative instead of positive, the edge maps go from $S_{2}$ to $S_{1}$ instead. For the empty cycle $Z_{\emptyset}$, the map 

\[ d_{1}^{f}: \widetilde{H}_{H}(S_{2}-Z_{\emptyset}) \to \widetilde{H}_{H}(S_{1}-Z_{\emptyset}) \]

\noindent
is the HOMFLY-PT edge map for a negative crossing, $\Delta_{-}$. For the cycles $Z_{13}$ and $Z_{24}$, the maps are opposite those for the positive crossing. In particular, the map from $ \widetilde{H}_{H}(S_{2}-Z_{13})$ to $ \widetilde{H}_{H}(S_{1}-Z_{13})$ is multiplication by $U_{4}$, while the map from $ \widetilde{H}_{H}(S_{2}-Z_{24})$ to $ \widetilde{H}_{H}(S_{1}-Z_{24})$ is the isomorphism. These results are summarized in Figure \ref{table2}.

\begin{figure}
\begin{center}
  \begin{tabular}{ l | c | c }

    Cycles & (+) Edge Map & (-) Edge Map \\ \hline
    $Z_{\phi}$ & $\Delta_{+}$ & $\Delta_{-}$\\ \hline
    $Z_{13}$ & $1$ & $U_{4}$ \\ \hline
    $Z_{24}$ & $U_{1}$ & $1$ \\ \hline
    $Z_{23}$ & $0$ & $0$ \\ \hline
    $Z_{14}$ & $0$ & $0$ \\ \hline
    $Z_{1234}$ & $0$ & $0$ \\ 
    \hline
  \end{tabular}
\end{center}
\caption{Summary of the Filtered Edge Maps} \label{table2}
\end{figure}

We refer to the local cycle $Z_{13}$ as a `left turn' and the local cycle $Z_{24}$ as a `right turn.' It follows from these edge maps that any cycle which makes a left turn at a positive crossing or a right turn at a negative crossing will have no contribution to the homology of the $E_{2}$ page. For this reason, we say that a cycle is \emph{admissible} if it makes no left turns at positive crossings or right turns at negative crossings.

\begin{defn}

We say that a complex has been \emph{reduced} at an edge $e_{i}$ if it is tensored with the complex 

\[ R \xrightarrow{\hspace{2mm}U_{i}\hspace{2mm}} R \]

\noindent
where the map increases the cube grading by $1$ (and treats the other two gradings the usual way for edge maps).

\end{defn}

\begin{defn}

Let $D$ be a diagram for a link $L$, and let $n$ be greater than the number of components of $L$. We define $\widetilde{H}_{H}(D,n)$ to be the complex $\widetilde{H}_{H}(D)$ reduced at $n$ edges, with at least one edge on each component of $L$.

\end{defn}

To see that this is well-defined and does not depend on the choice of edges, apply the theorem from \cite{Rasmussen} that multiplication by edges on the same component of $D$ gives the same action on homology.

Using the above computations, we see that the homology for the empty cycle is just HOMFLY-PT homology of the the whole diagram, $\widetilde{H}_{H}(D)$. However, when $D$ is a diagram of a knot, then any non-empty cycle $Z$ must make at least one turn. In fact, if we look at each component of $D-Z$, $Z$ must make a turn adjacent to this component. Thus, if $Z$ is admissible, then each component of $D-Z$ gets reduced by an edge map coming from a turn in $Z$. It follows that 

\[H_{*}(H_{*}(\widetilde{C}_{F}(D), d_{0}^{f}), (d_{1}^{f})^{*}) \cong \widetilde{H}_{H}(D) \oplus \bigoplus_{\substack{Z \text{ nonempty} \\ Z \text{ admissible} }} \widetilde{H}_{H}(D-Z, n(Z)) \]

\noindent
where $n$ is the number of turns made by $Z$.

\section{A Sufficient Condition for Invariance} \label{section3}

Let $C(D)$ denote an abstract cube of resolutions complex. In this section we will give sufficient algebraic conditions on $C(D)$ for the $E_{2}$ and higher pages of the induced cube of resolutions to be invariants.

In order for a homology theory to be an invariant of braids, it needs to be invariant under the braid-like isotopies shown in Figure \ref{Reidemeister}. Instead of the typical 3-crossing Reidemeister III move, we use the move that relates the $6$-crossing diagram with the $0$-crossing diagram - the reason for this choice will become clear below. We will view these braid-like Reidemeister moves as replacing one tangle with another within the closed braid. 

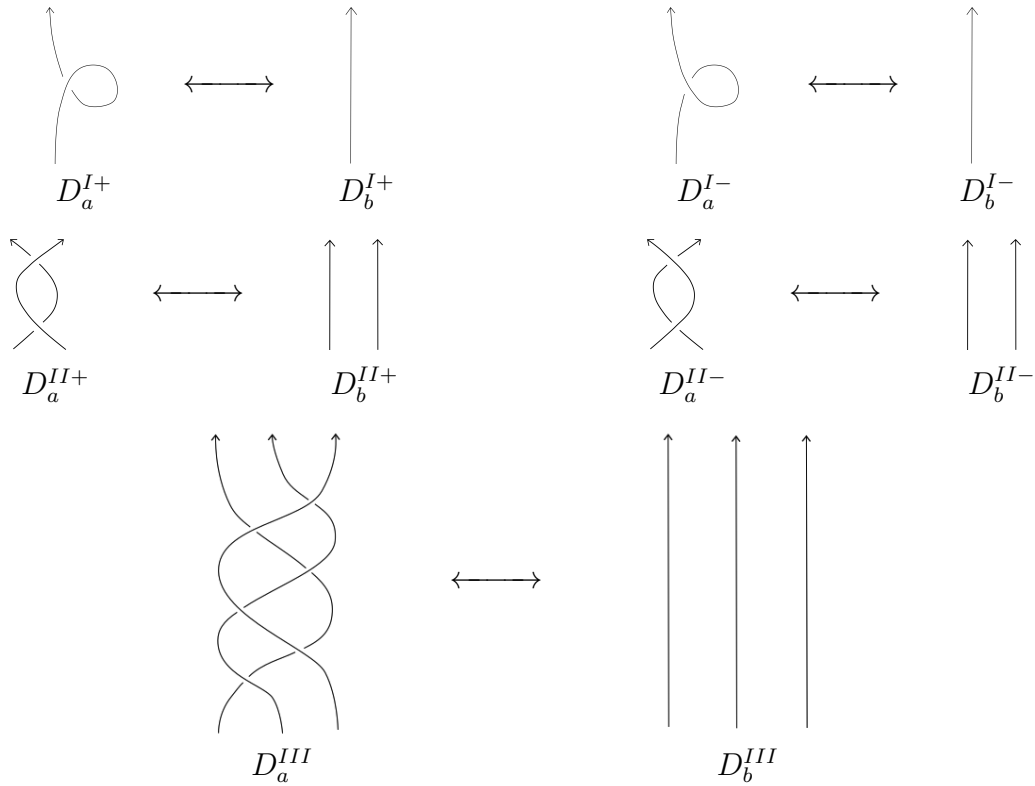
\begin{figure}[h!]
\begin{subfigure}{.5\textwidth}
 \centering
   \begin{overpic}[width=.5\textwidth]{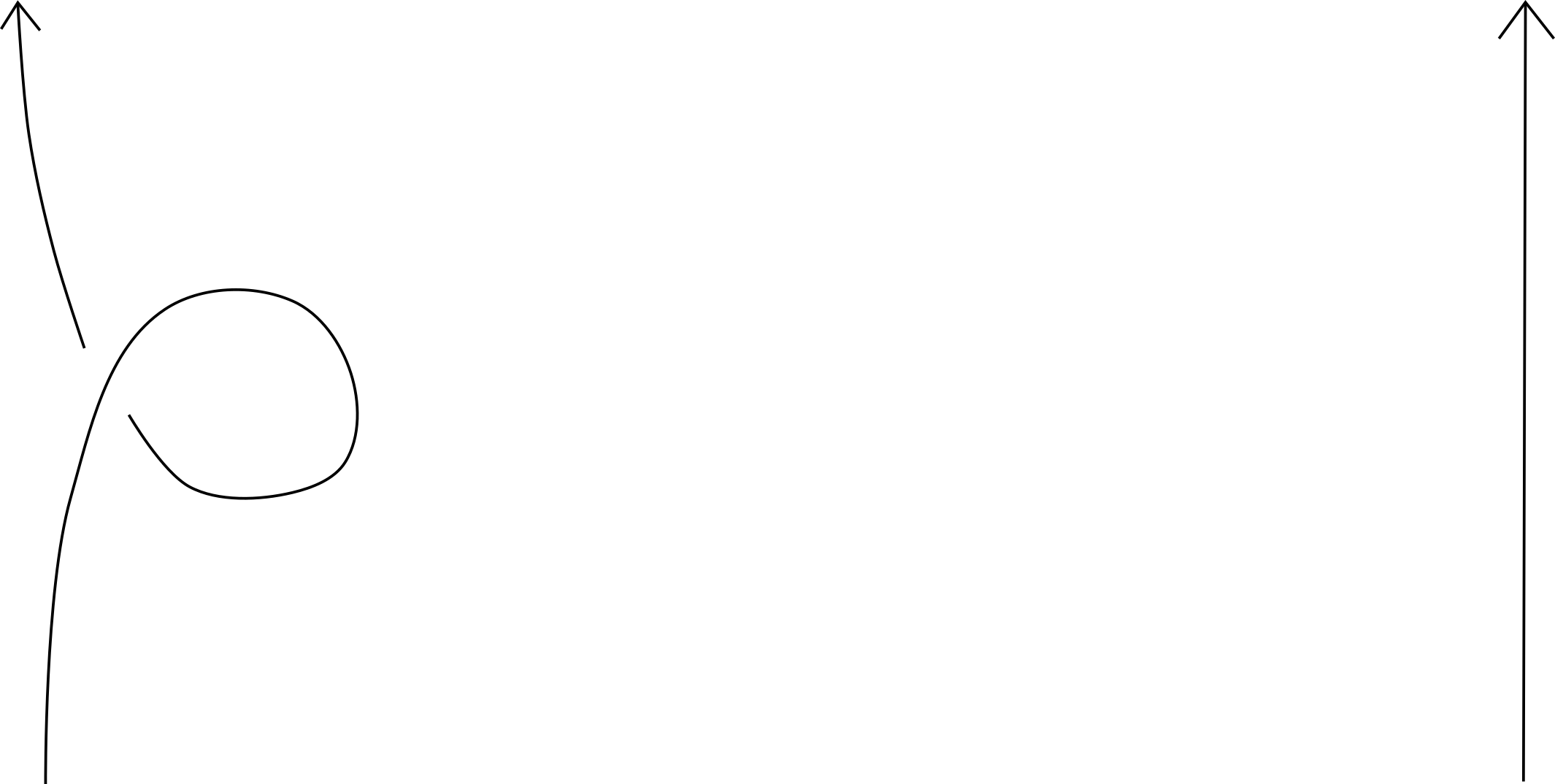}
   \put(44,23){$\xleftrightarrow{\hspace{10mm}}$}
   \put(3,-12){$D^{I+}_{a}$}
   \put(94,-12){$D^{I+}_{b}$}   
   \end{overpic}
\end{subfigure}%
\begin{subfigure}{.5\textwidth}
  \centering
   \begin{overpic}[width=.5\textwidth]{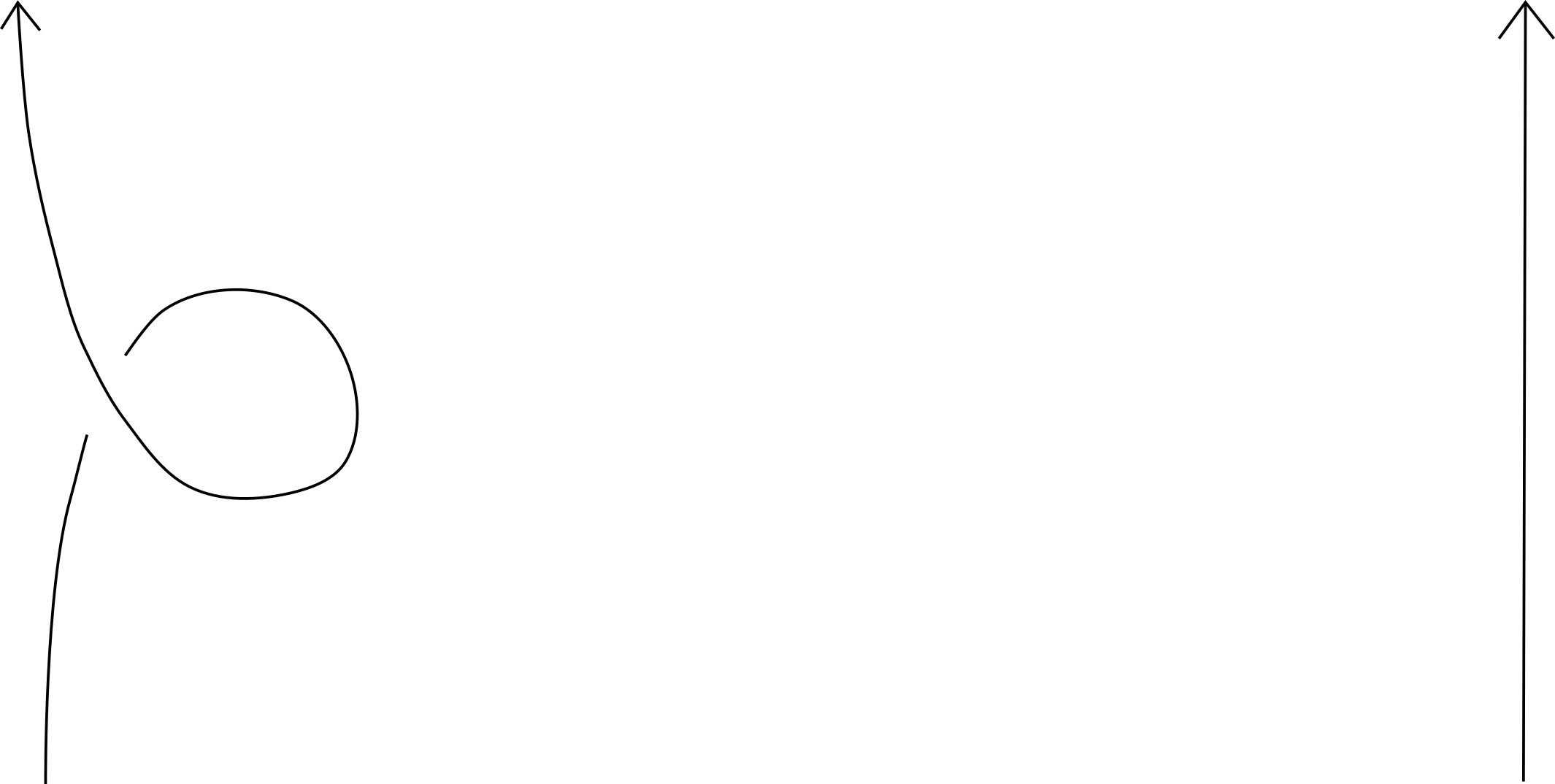}
   \put(45,23){$\xleftrightarrow{\hspace{10mm}}$}
   \put(3,-12){$D^{I-}_{a}$}
   \put(94,-12){$D^{I-}_{b}$}      
   \end{overpic}
\end{subfigure}

\vspace{10mm}

\begin{subfigure}{.5\textwidth}
 \centering
   \begin{overpic}[width=.6\textwidth]{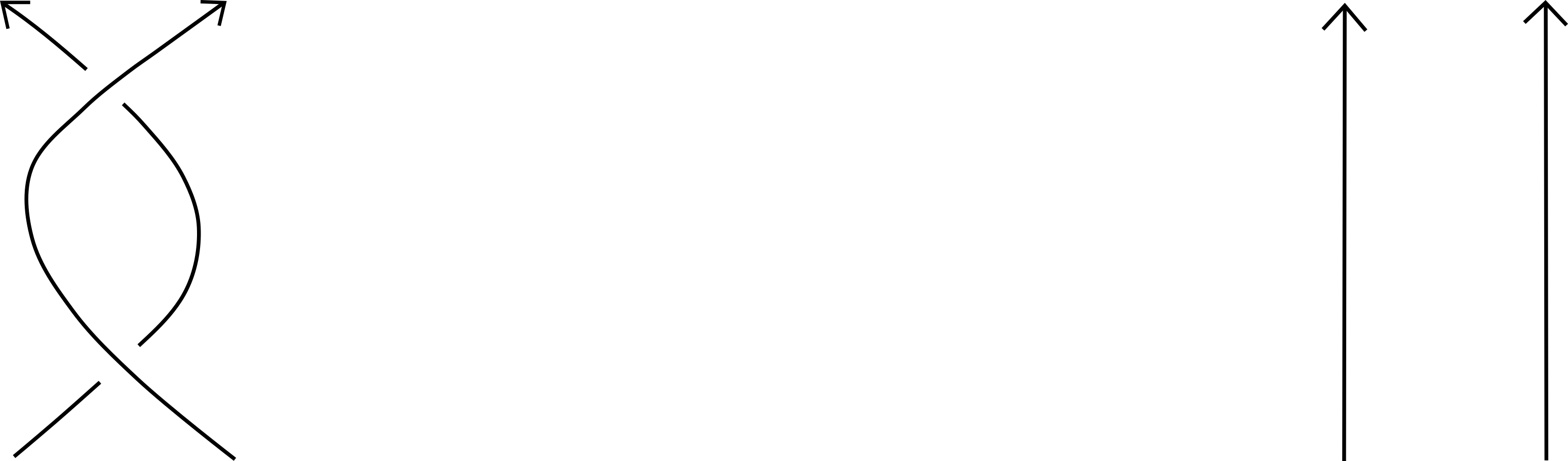}
   \put(38,13){$\xleftrightarrow{\hspace{10mm}}$}
   \put(3,-12){$D^{II+}_{a}$}
   \put(86,-12){$D^{II+}_{b}$}   
   \end{overpic}
  \vspace{10mm}
\end{subfigure}%
\begin{subfigure}{.5\textwidth}
  \centering
  \hspace{3mm}
   \begin{overpic}[width=.6\textwidth]{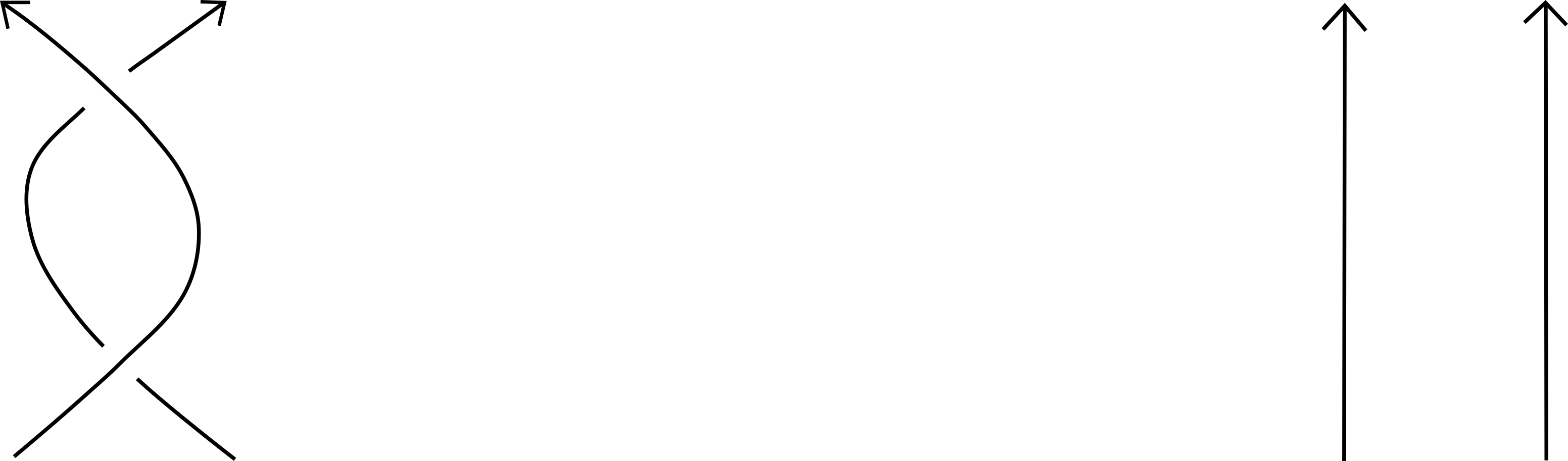}
   \put(38,13){$\xleftrightarrow{\hspace{10mm}}$}
   \put(3,-12){$D^{II-}_{a}$}
   \put(86,-12){$D^{II-}_{b}$}      
   \end{overpic}
  \vspace{10mm}
\end{subfigure}
\begin{subfigure}{1.0 \textwidth}
 \centering
   \begin{overpic}[width=.5\textwidth]{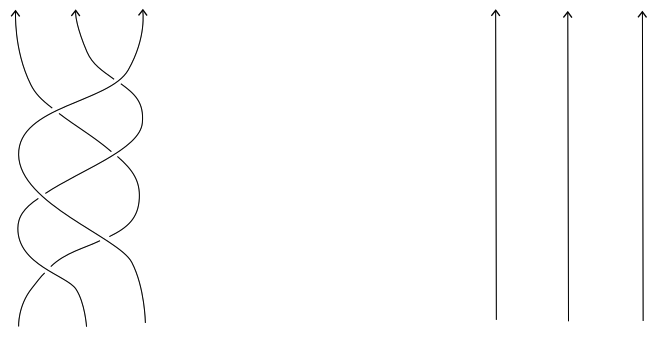}
   \put(40,25){$\xleftrightarrow{\hspace{10mm}}$}
   \put(8,-5){$D^{III}_{a}$}
   \put(83,-5){$D^{III}_{b}$}   
   \end{overpic}
\end{subfigure}
\vspace{5mm}
 \caption{Braidlike Reidemeister Moves} \label{Reidemeister}
\end{figure}

We will refer to the tangles shown in Figure \ref{Reidemeister} as \emph{inner} tangles, and the the complement of these tangles in $D$ as \emph{outer} tangles. For the Type I moves, the inner and outer tangles are (1,1) tangles, for the Type II moves the inner and outer tangles are (2,2) tangles, and for the Type III move the inner and outer tangles are (3,3) tangles.

Let $D_{a}$ and $D_{b}$ be two braid diagrams related by one of the braid-like Reidemeister moves such that $D_{a}$ has the type (a) inner tangle and $D_{b}$ has the type (b) inner tangle. We will denote these inner tangles by $T_{a}$ and $T_{b}$, respectively, with the outer tangle denoted $T_{out}$.

Since the cube grading is given by the sum over all crossings of the grading from that crossing, we can decompose the cube grading into the inner cube grading (the sum over all crossings in the inner diagram) and the outer cube grading. We will denote these by $gr_{in}$ and $gr_{out}$, respectively. Note that $gr_{in}+gr_{out} = gr_{cube}$. Let $\mathcal{F}_{in}$, $\mathcal{F}_{out}$, and $\mathcal{F}_{cube}$ be the filtrations induced by these gradings.

Let $d_{ij}$ denote the differentials on $D$ which increase the outer cube grading by $i$ and the inner cube grading by $j$

\begin{thm} \label{thm3.1}

Suppose $C(D)$ satisfies the following properties:

\begin{enumerate}[label=(\alph*)] 
\item With respect to the filtration $\mathcal{F}_{out}$, the complexes $C(D_{a})$ and $C(D_{b})$ are filtered chain homotopy equivalent.
\item The homology $H_{*}(H_{*}(C(D_{a}), d_{00}), d_{01}^{*})$ lies in a single inner grading $gr_{in}=c$.
\end{enumerate}

\noindent
Then the $E_{k}(C(D_{a}))\{-c\} \cong E_{k}(C(D_{b}))$ for all $k \ge 2$, where the isomorphism is graded with grading $gr_{cube}$ and $\{n\}$ refers to a shift by $n$ in the cube grading.

\end{thm}

\begin{proof}

Since the tangle $T_{b}$ does not contain any crossings, we have that 

\[ (C(D_{b}), \mathcal{F}_{cube}) \cong (C(D_{b}), \mathcal{F}_{out}) \]

\noindent
as filtered complexes. Together with (a), this shows that

\[ (C(D_{a}), \mathcal{F}_{out}) \cong (C(D_{b}), \mathcal{F}_{cube}) \]

\noindent
Thus, it suffices to show that 

\[ E_{k}(C(D_{a}), \mathcal{F}_{cube})\{-c\} \cong E_{k}(C(D_{a}), \mathcal{F}_{out}) \]

\noindent
for $k \ge 2$, where the spectral sequences are induced by the respective filtrations. We are therefore interested in comparing these two filtrations on the complex $C(D_{a})$.

Let $d_{k}^{out}$ denote the differentials which increase $gr_{out}$ by $k$, and $d_{k}^{cube}$ the differentials which increase $gr_{cube}$ by $k$. Then

\[ d_{k}^{out} = \sum_{j} d_{kj}, \hspace{20mm} d_{k}^{cube} = \sum_{i+j=k} d_{ij} \]

The $E_{1}$ page relative to the $\mathcal{F}_{out}$ filtration is given by

\[ E_{1}(C(D_{a}), \mathcal{F}_{out}) = H_{*}(C(D_{a}), d^{out}_{0}) = H_{*}(C(D_{a}), d_{00}+d_{01}+...+d_{0n}) \]

\noindent
where $n$ is the number of crossings in $T_{a}$. If we filter the $d^{out}_{0})$ differential by the inner filtration $\mathcal{F}_{in}$, then this homology will be the $E_{\infty}$ page of the induced spectral sequence. However, using assumption (b), this spectral sequence collapses at the $E_{2}$ page because it is in a single inner cube grading.

\[ E_{1}(C(D_{a}), \mathcal{F}_{out}) \cong H_{*}(H_{*}(C(D_{a}), d_{00}), d_{01}^{*}) \]

\noindent
The induced $d_{1}^{out}$ map is given by 

\[ (d_{1}^{out})^{*} = d_{10}^{*}+d_{11}^{*}+...+d_{1n}^{*}  \]

\noindent
but since the homology is in a single inner cube grading, $d_{ij}^{*}=0$ for $j \ne 0$. Thus, $(d_{1}^{out})^{*} = d_{10}^{*} $ and

\[ E_{2}(C(D_{a}), \mathcal{F}_{out}) \cong H_{*}(H_{*}(H_{*}(C(D_{a}), d_{00}), d_{01}^{*}), d_{10}^{*})\]

The $E_{1}$ page relative to the $\mathcal{F}_{cube}$ filtration is given by

\[ E_{1}(C(D_{a}), \mathcal{F}_{out}) = H_{*}(C(D_{a}), d^{cube}_{0}) = H_{*}(C(D_{a}), d_{00}) \]

\noindent
The induced $d_{1}^{cube}$ map is given by 

\[ (d_{1}^{cube})^{*} = d_{01}^{*}+d_{10}^{*}  \]

\noindent
By filtering this differential with the outer cube grading, we can first take homology with respect to $d_{01}^{*}$, giving

\[H_{*}(H_{*}(C(D_{a}), d_{00}), d_{01}^{*})\]

\noindent
This homology lies in a single inner cube grading, so the spectral sequence induced by the the outer cube filtration on $(d_{1}^{cube})^{*}$ collapses at the next page:

\[ H_{*}(H_{*}(C(D_{a}, d_{00}), d^{*}_{01}+d^{*}_{10}) \cong H_{*}(H_{*}(H_{*}(C(D_{a}, d_{00}), d^{*}_{01}), d^{*}_{10}) \]

\noindent
proving that

\[ E_{2}(C(D_{a}), \mathcal{F}_{cube}) \cong H_{*}(H_{*}(H_{*}(C(D_{a}), d_{00}), d_{01}^{*}), d_{10}^{*})  \cong E_{2}(C(D_{a}), \mathcal{F}_{out})  \]

This proves the isomorphism of the $E_{2}$ pages - to make the isomorphism graded, we observe that $E_{2}(C(D_{a}), \mathcal{F}_{cube})$ lies in inner cube grading $c$, so the isomorphism is homogeneous of grading $-c$. Thus, adding in the grading shift, we see that 

\[ E_{2}(C(D_{a}), \mathcal{F}_{cube})\{-c\} \cong E_{2}(C(D_{a}), \mathcal{F}_{out}) \]

To see that the higher pages of the spectral sequence are isomorphic as well, we note that the $E_{2}$ pages of the two complexes were computed by taking homology with respect to the same differentials. It follows that the induced higher differentials are the same, and since the local cube grading has collapsed, the filtrations $\mathcal{F}_{cube}$ and $\mathcal{F}_{out}$ are now the same. Since the filtered complexes are isomorphic, all higher pages in the spectral sequence are isomorphic as well.

\end{proof}

One can show that the HOMFLY-PT complex $C_{H}(D)$ satisfies the conditions of these theorems. This gives an alternative proof that HOMFLY-PT homology, the $E_{2}$ page of the spectral sequence induced by the cube filtration, is an invariant of braids. The higher pages of the spectral sequence are Rasmussen's $E_{k}(-1)$ spectral sequence.

\section{The Proof of Invariance}

Let $\widetilde{C}_{F}(D)$ denote the triply graded complex with triple grading $(M,A,gr_{cube})$. We define $\widetilde{C}_{F}(D)\{k\}$ to be the complex in which the cube grading has an overall shift of $k$. In this section we will prove the main theorem:

\begin{thm}

Let $E_{k}(\widetilde{C}_{F}(D))\{\frac{1}{2}(-c(D)-b(D))\}$ denote the spectral sequence induced by the cube filtration. Then $E_{2}(\widetilde{C}_{F}(D))\{\frac{1}{2}(-c(D)-b(D))\}$ is a triply graded link invariant which categorifies the HOMFLY-PT polynomial, and the higher pages are link invariants.

\end{thm}

We will prove this by showing that $\widetilde{C}_{F}(D)$ satisfies the conditions of Theorem \ref{thm3.1}. Condition (a) will follow an argument extending the triangle maps corresponding to isotopies and handleslides to the cube complex, while (b) will involve direct computation. We will complete the proof by showing that $\frac{1}{2}(-c(D)-b(D))$ is the proper grading shift to make the cube grading into an absolute grading (this is the same grading shift needed for HOMFLY-PT homology).

We need a lemma from homological algebra:

\begin{lem}
\label{homalglemma}

Suppose $F$ is a filtered chain map between filtered complexes $C$ and $C'$. If $F$ induces an isomorphism on the homologies of the associated graded objects of $C$ and $C'$, then $F$ is a filtered chain homotopy equivalence.

\end{lem}

\begin{lem} \label{heegaardmoves}

If $\mathcal{H}_{1}$ and $\mathcal{H}_{2}$ are two Heegaard diagrams containing the crossing diagram for the oriented cube of resolutions that differ by a handleslide or isotopy away the $\alpha$ and $\beta$ circles involved in the crossing, then there is a filtered chain homotopy equivalence from the cube of resolutions corresponding to $\mathcal{H}_{1}$ to the cube of resolutions corresponding to $\mathcal{H}_{2}$. 
\end{lem}

By `away from the $\alpha$ and $\beta$ curves' we mean that the the curves $\alpha_{1}$, $\alpha_{2}$, $\beta_{1}$, and $\beta_{2}$ in Figure \ref{HDCrossing} are rigid. This does allow handleslides over $\alpha_{2}$ and $\beta_{2}$.

\begin{proof}

We will assume the crossing is positive, but an analogous argument works for the negative crossing. From Heegaard Floer theory, we know that an isotopy or handleslide induces a chain homotopy equivalence on $\mathit{CFK}^{-}$ obtained by counting holomorphic triangles. If we place $X$'s at the $A$'s, we get the map 

\begin{figure}[!h]
\centering
\begin{tikzpicture}
  \matrix (m) [matrix of math nodes,row sep=5em,column sep=6em,minimum width=2em] {
     X_{1} & Y_{1} \\
     X_{2} & Y_{2} \\};
  \path[-stealth]
    (m-1-1) edge node [left] {$f_{0}$} (m-2-1)
            edge node [above] {$\Phi_{B}$} (m-1-2)
            edge node [right]{$f_{B}$} (m-2-2)
    (m-2-1.east|-m-2-2) edge node [below] {$\Phi_{B}$} (m-2-2)
    (m-1-2) edge node [right] {$f_{0}$} (m-2-2);
\end{tikzpicture}
\end{figure}

\noindent
where $f_{0}$ counts triangles with multiplicity zero at $A$ and $B$, and $f_{B}$ counts triangles with multiplicity 1 at one of the $B$'s and 0 at the other $B$ and both $A$'s.

Similarly, if we place $X$'s at the $B$'s, we get the map between smoothings

\begin{figure}[!h]
\centering
\begin{tikzpicture}
  \matrix (m) [matrix of math nodes,row sep=5em,column sep=6em,minimum width=2em] {
     Y_{1} & X_{1} \\
     Y_{2} & X_{2} \\};
  \path[-stealth]
    (m-1-1) edge node [left] {$f_{0}$} (m-2-1)
            edge node [above] {$\Phi_{A}$} (m-1-2)
            edge node [right]{$f_{A}$} (m-2-2)
    (m-2-1.east|-m-2-2) edge node [below] {$\Phi_{A}$} (m-2-2)
    (m-1-2) edge node [right] {$f_{0}$} (m-2-2);
\end{tikzpicture}
\end{figure}

Using these maps, we can construct a map $f$ between the cubes as in Figure \ref{f}, where $L$ is the linear relation $U_{1}+U_{2}+U_{3}+U_{4}$ and $f_{A^{+}B}$ counts triangles with multiplicity one at one of the $A$'s and one of the $B$'s, and multiplicity 0 at the other $A$ and $B$. The fact that $f$ is a chain map almost follows from the previous two examples - the only remaining maps that we have to show commute correspond to the maps from the upper left $X_{1}$ to the lower right $X_{2}$. That is,

\begin{equation} \label{Adiscs}
 f_{A^{+}B} \circ d_{X_{1}} + d_{X_{2}} \circ f_{A^{+}B} + f_{A^{+}} \circ \Phi_{B} + \Phi_{A^{+}} \circ f_{B} + \Phi_{A^{+}B} \circ f_{0}+f_{0} \circ \Phi_{A^{+}B} =0 
 \end{equation}

\begin{figure}

\centering
\begin{tikzpicture}
  \matrix (m) [matrix of math nodes,row sep=5em,column sep=6em,minimum width=2em] {
     X_{1} & Y_{1} \\
     X_{1} & X_{1} \\
                &            \\
     X_{2} & Y_{2} \\
     X_{2} & X_{2} \\};
  \path[-stealth]
    (m-1-1) edge node [left] {$L$} (m-2-1)
            edge node [above] {$\Phi_{B}$} (m-1-2)
            edge node [right]{$\Phi_{A^{+}B}$} (m-2-2)
            edge [bend right = 45] node [left] {$f_{0}$} (m-4-1)
            edge [bend right = 55] node [right] {$f_{B}$} (m-4-2)
    (m-2-1) edge node [below] {1} (m-2-2)
            edge [bend right=45] node [left] {$f_{0}$} (m-5-1)
    (m-1-2) edge node [right] {$\Phi_{A^{+}}$} (m-2-2)
            edge [bend left=45] node [right] {$f_{0}$} (m-4-2)
            edge [bend left=55] node [right] {$f_{A}$} (m-5-2)
    (m-2-2) edge [bend left=45] node [right] {$f_{0}$} (m-5-2)    
    (m-4-1) edge node [left] {$L$} (m-5-1)
            edge node [above] {$\Phi_{B}$} (m-4-2)
            edge node [right]{$\Phi_{A^{+}B}$} (m-5-2)
    (m-5-1.east|-m-5-2) edge node [below] {1} (m-5-2)
    (m-4-2) edge node [right] {$\Phi_{A^{+}}$} (m-5-2);
    
 \draw [->] (m-1-1) 
           to [out=180,in=135] (current bounding box.south west) 
           to [out=-45,in=-130] node[below,midway]{$f_{AB}$} (m-5-2); 
\end{tikzpicture}
\caption{The Chain Map}\label{f}
\end{figure}

To see this, we will examine the standard knot Floer complexes from the diagrams $\mathcal{H}_{1}$ and $\mathcal{H}_{2}$, with the modification that we allow discs to pass through both $A$ and $A^{+}$. These complexes are now curved complexes with \[d^{2}=(U_{1}+U_{2}+U_{3}+U_{4})I\] and counting holomorphic triangles gives a morphism of curved complexes that still commutes with the two differentials (see Figure \ref{curved}).

\begin{figure}[!h]

\centering
\begin{tikzpicture}
  \matrix (m) [matrix of math nodes,row sep=5em,column sep=6em,minimum width=2em] {
     X_{1} & Y_{1} \\
     X_{2} & Y_{2} \\};
  \path[-stealth]
    (m-1-1) edge [bend left] node [left] {$f_{0}$} (m-2-1)
            edge [bend left] node [above] {$\Phi_{B}$} (m-1-2)
            edge node [right]{$f_{B}$} (m-2-2)
            edge [bend right] node [left] {$f_{A^{+}B}$} (m-2-1)
    (m-2-1) edge [bend left] node [below] {$\Phi_{B}$} (m-2-2)
    (m-1-2) edge [bend right] node [right] {$f_{0}$} (m-2-2)
            edge [bend left] node [right] {$f_{A^{+}B}$} (m-2-2)
            edge [bend left] node [above] {$\Phi_{A^{+}}$} (m-1-1)
    (m-2-2) edge [bend left] node [below right]  {$\Phi_{A^{+}}$} (m-2-1);
            
\draw [->] (m-1-2) 
           to [out=45,in=45] (current bounding box.south east) 
           to [out=-135,in=-60] (m-2-1);
           
\node at  (2.7,-2.5)  {$f_{A^{+}}$};
           
\end{tikzpicture}
\caption{A Morphism of Curved Complexes}\label{curved}
\end{figure}
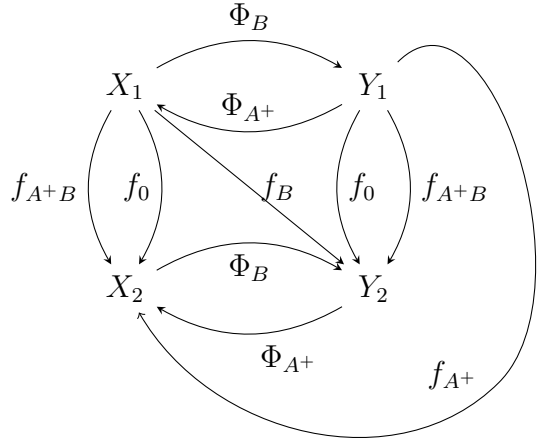

This complex is filtered by the number of times a disc passes through one of the $A$'s (the Alexander filtration) and the $df+fd$ terms from $X_{1}$ to $X_{2}$ which decrease the filtration by $1$ are precisely those from (\ref{Adiscs}). Since the maps from holomorphic triangles commute with the differentials on these two complexes, they must sum to zero. It follows that our choice of $f$ from the cube corresponding to $\mathcal{H}_{1}$ to the cube corresponding to $\mathcal{H}_{2}$ is a chain map. 

Since $f$ does not decrease the cube grading, $f$ is filtered. To see that this map induces an isomorphism on the associated graded objects, we simply observe that the maps restricted to each vertex in the cube of resolutions are precisely the standard maps induced by isotopies or handleslides for those diagrams, which are known to be chain homotopy equivalences. Applying Lemma \ref{homalglemma} completes the proof.

\end{proof}

This lemma allows us to do isotopies and handleslides away from crossings without impacting the filtered chain homotopy type. We will also need to perform $(0,3)$ stabilizations.

\begin{lem} \label{insertionlemma}

Let $D$ be a braid diagram, and let $D'$ be the diagram obtained by adding a bivalent vertex (an insertion) to $D$. Then the filtered chain homotopy types of $D$ and $D'$ are the same with respect to the cube filtration.

\end{lem}

\begin{proof}

The local Heegaard diagrams for $D$ and $D'$ are shown in Figure \ref{insert}. We can perform two handleslides in the Heegaard diagram for $D'$ as shown in Figure . By Lemma \ref{heegaardmoves}, since these handleslides happen away from all crossings, they do not change the chain homotopy type of the cube complex.

\begin{figure}[h!]
\begin{subfigure}{.5\textwidth}
 \centering
   \begin{overpic}[width=.5\textwidth]{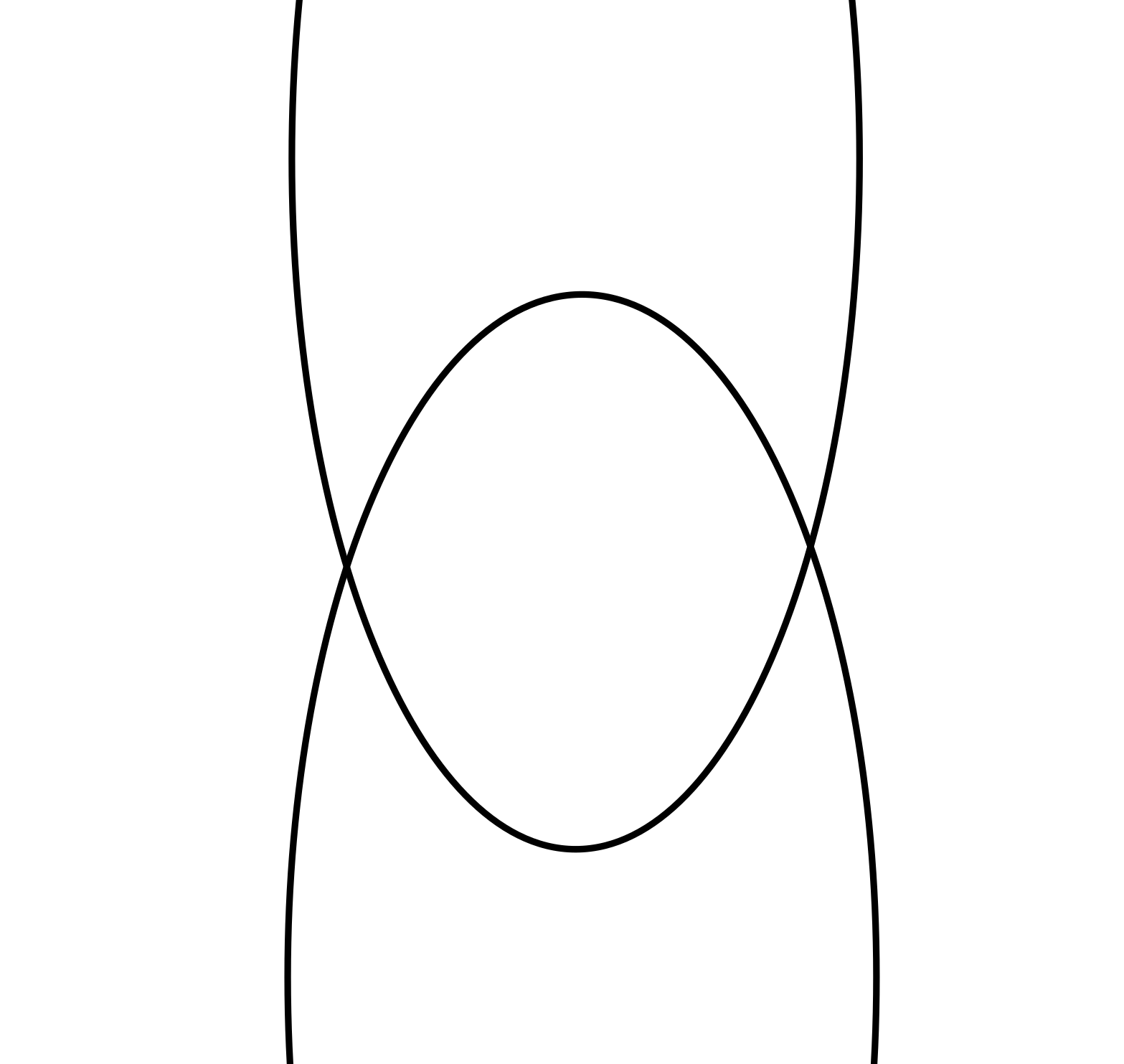}
   \put(46,41){$O_{1}$}
   \end{overpic}
  \caption{The Heegaard diagram for $D$}
\end{subfigure}%
\begin{subfigure}{.5\textwidth}
  \centering
   \begin{overpic}[width=.5\textwidth]{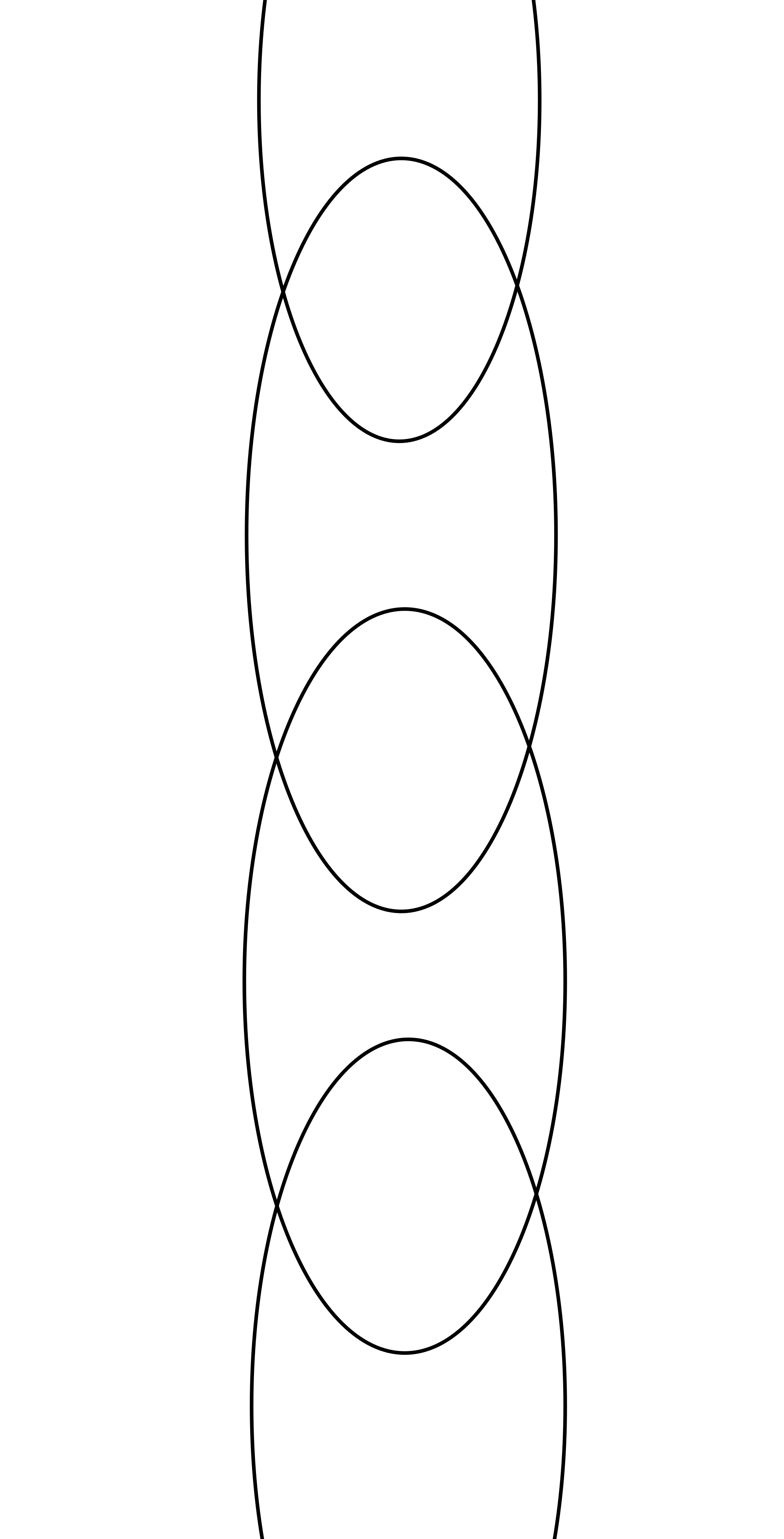}
   \put(23,79){$O_{1}$}
   \put(23.5,49){$X$}
   \put(23,21){$O_{2}$}
   \end{overpic}
   \caption{The Heegaard diagram for $D'$}
\end{subfigure}
\caption{The Heegaard diagrams before and after adding the insertion}
\label{insert}
\end{figure}

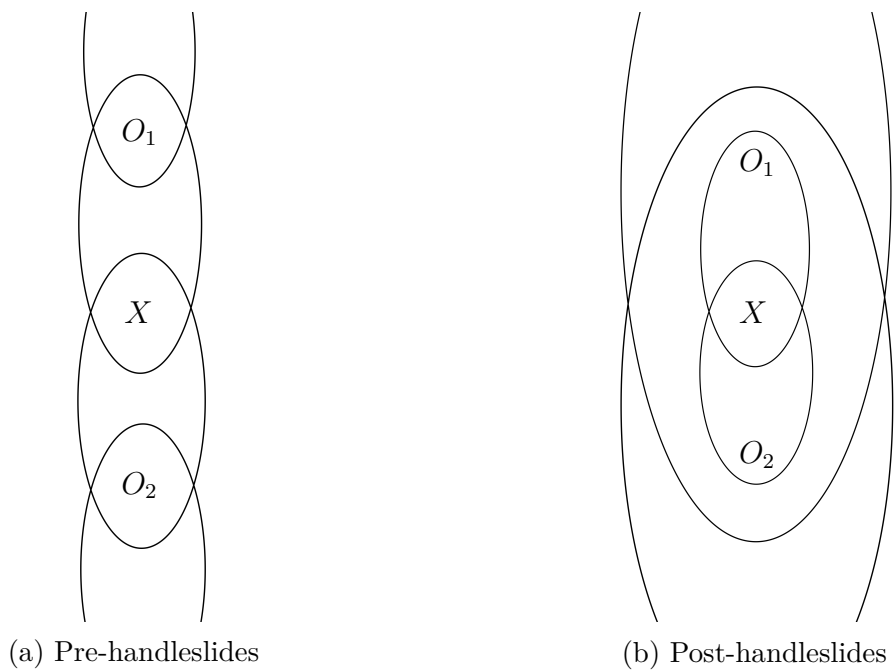
\begin{figure}[h!]
\begin{subfigure}{.5\textwidth}
 \centering
   \begin{overpic}[width=.5\textwidth]{Insertion.png}
   \put(23,79){$O_{1}$}
   \put(23.5,49){$X$}
   \put(23,21){$O_{2}$}
   \end{overpic}
  \caption{Pre-handleslides}
\end{subfigure}%
\begin{subfigure}{.5\textwidth}
  \centering
   \begin{overpic}[width=.6\textwidth]{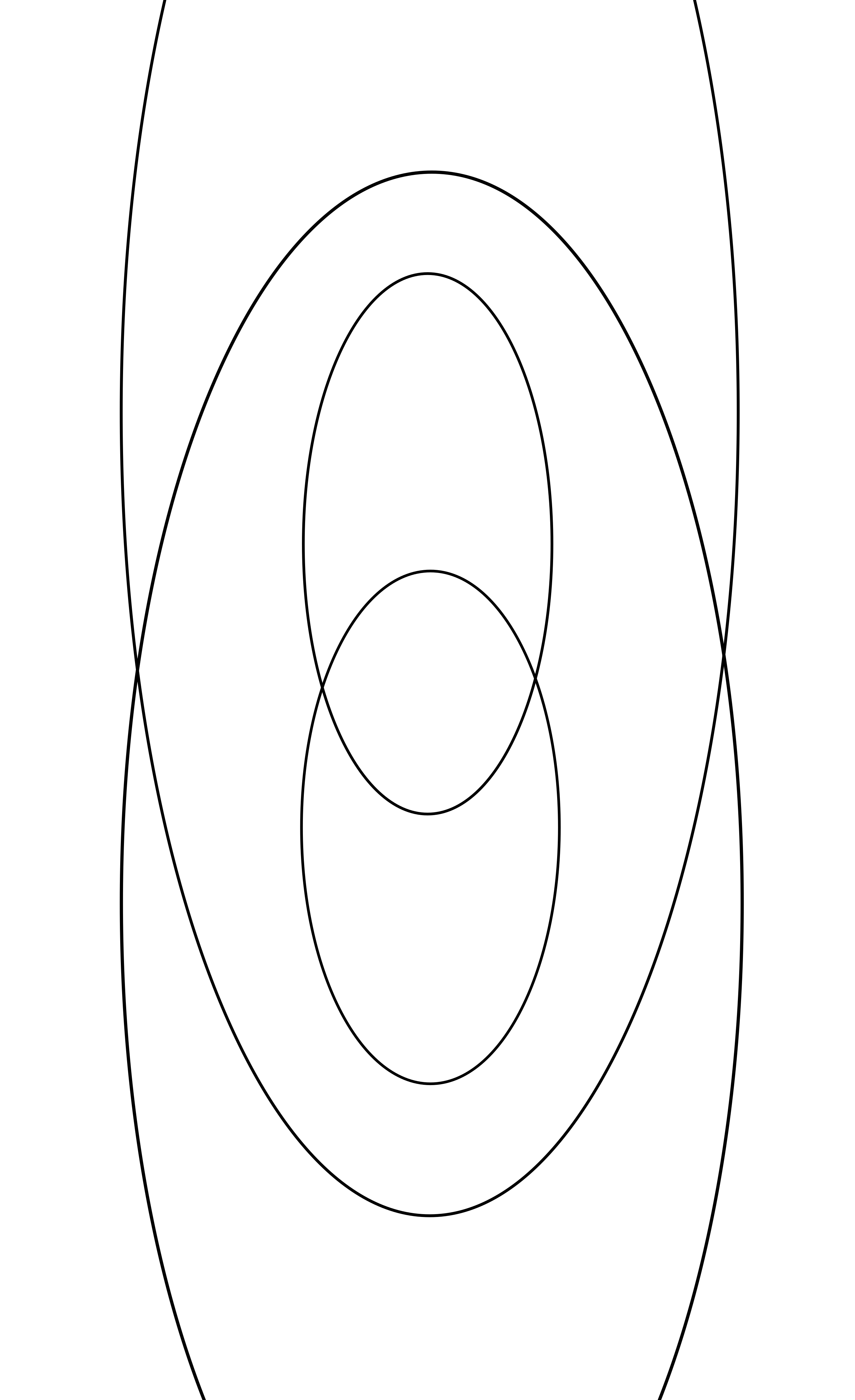}
      \put(28,49){$X$}
      \put(28,74){$O_{1}$}
      \put(28,26){$O_{2}$} 
   \end{overpic}
   \caption{Post-handleslides}
\end{subfigure}
\caption{The Heegaard diagram for $D'$ before and after the two handleslides.}
\label{insert}
\end{figure}

By the standard argument regarding $(0,3)$ stabilizations, the intersection points $x$ and $y$ give us a decomposition of $\widetilde{C}(D')$

\[ \widetilde{C}(D') \cong \widetilde{C}(D)[U_{2}] \xrightarrow{U_{1}+U_{2}} \widetilde{C}(D)[U] \]

\noindent
Canceling the mapping cone gives 

\[\widetilde{C}(D') \cong \widetilde{C}(D)[U_{2}] /U_{1}=U_{2} \cong \widetilde{C}(D) \]
\end{proof}

Let $D_{a}$ and $D_{b}$ be two braid diagrams which differ by one of the moves in Figure. Let $T_{a}$ and $T_{b}$ be the associated tangles, and suppose $B$ is a ball in $S^{2}$ with $B \cap D_{a}=T_{a}$, $B \cap D_{b} = T_{b}$. The ball $B$ determines a ball $B'$ in the Heegaard diagrams for $D_{a}$ and $D_{b}$, unique up to isotopy relative to the $\alpha$, $\beta$ curves and the $X$, $O$ markings. Let $\mathcal{H}_{a}$ and $\mathcal{H}_{b}$ denote the corresponding Heegaard diagrams.

\begin{lem} \label{lem4.5}

The Heegaard diagrams for $D_{a}$ and $D_{b}$ can be connected by a sequence of isotopies and handleslides that take place within $B'$, together with addition or removal of insertions.

\end{lem}

\begin{proof}

The moves to connect diagrams $D_{a}^{I+}$ and $D_{b}^{I+}$ are shown in Figure \ref{Moves} - the other diagrams can be connected similarly.

\begin{figure}[h!]
\begin{subfigure}{.33\textwidth}
 \centering
   \begin{overpic}[height=.65\textwidth]{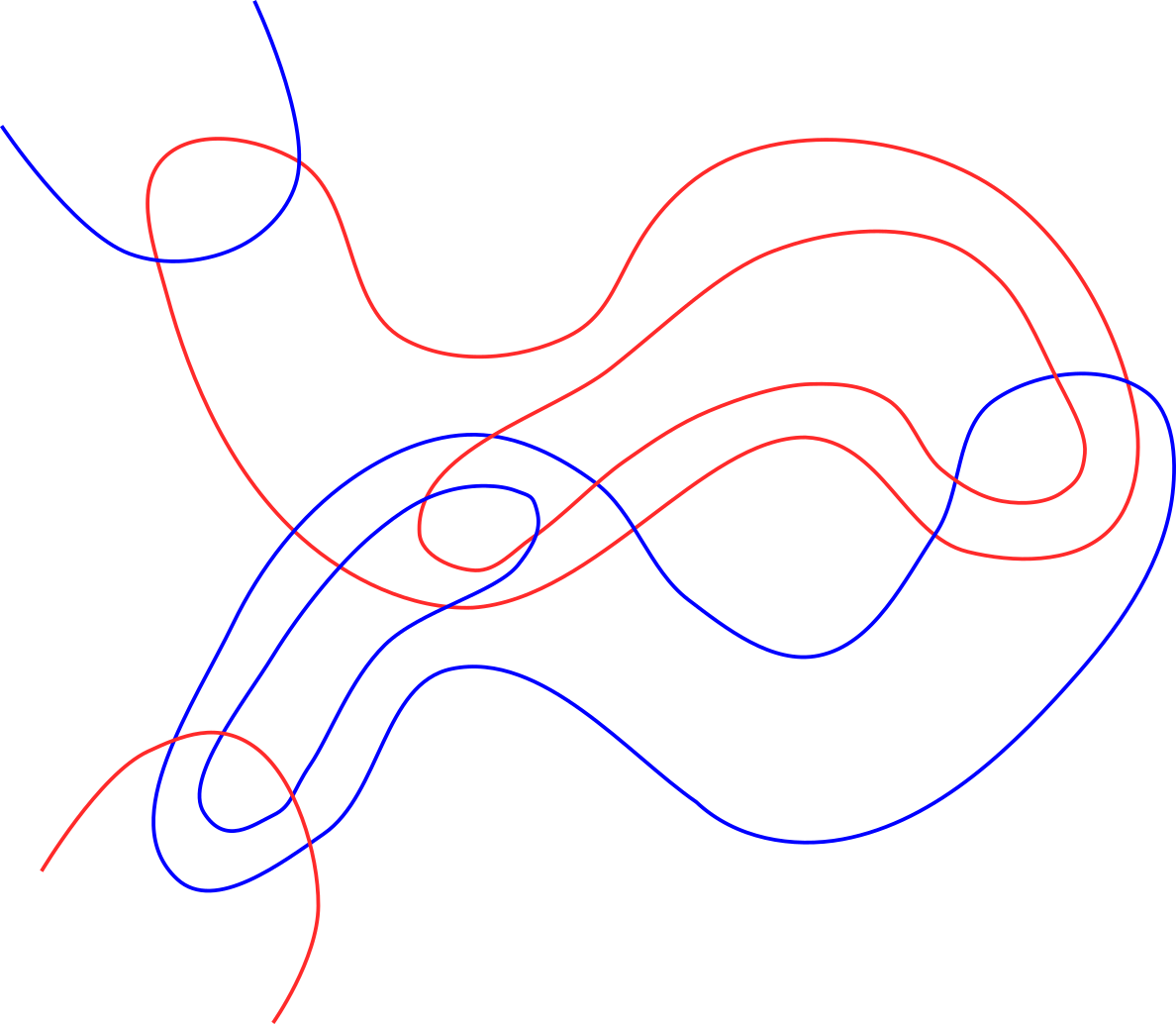} \fontsize{4}{4}\selectfont
   \put(37.5,41){$X$}
   \put(45.5, 39){$X$}
   \put(18.5,19.5){$O_{3}$}
   \put(84,48){$O_{2}$}
   \put(15,70){$O_{1}$}
   \normalsize
   \put(115,40){$\Longrightarrow$}
   \end{overpic}
\end{subfigure}%
\begin{subfigure}{.33\textwidth}
  \centering
   \begin{overpic}[height=.65\textwidth]{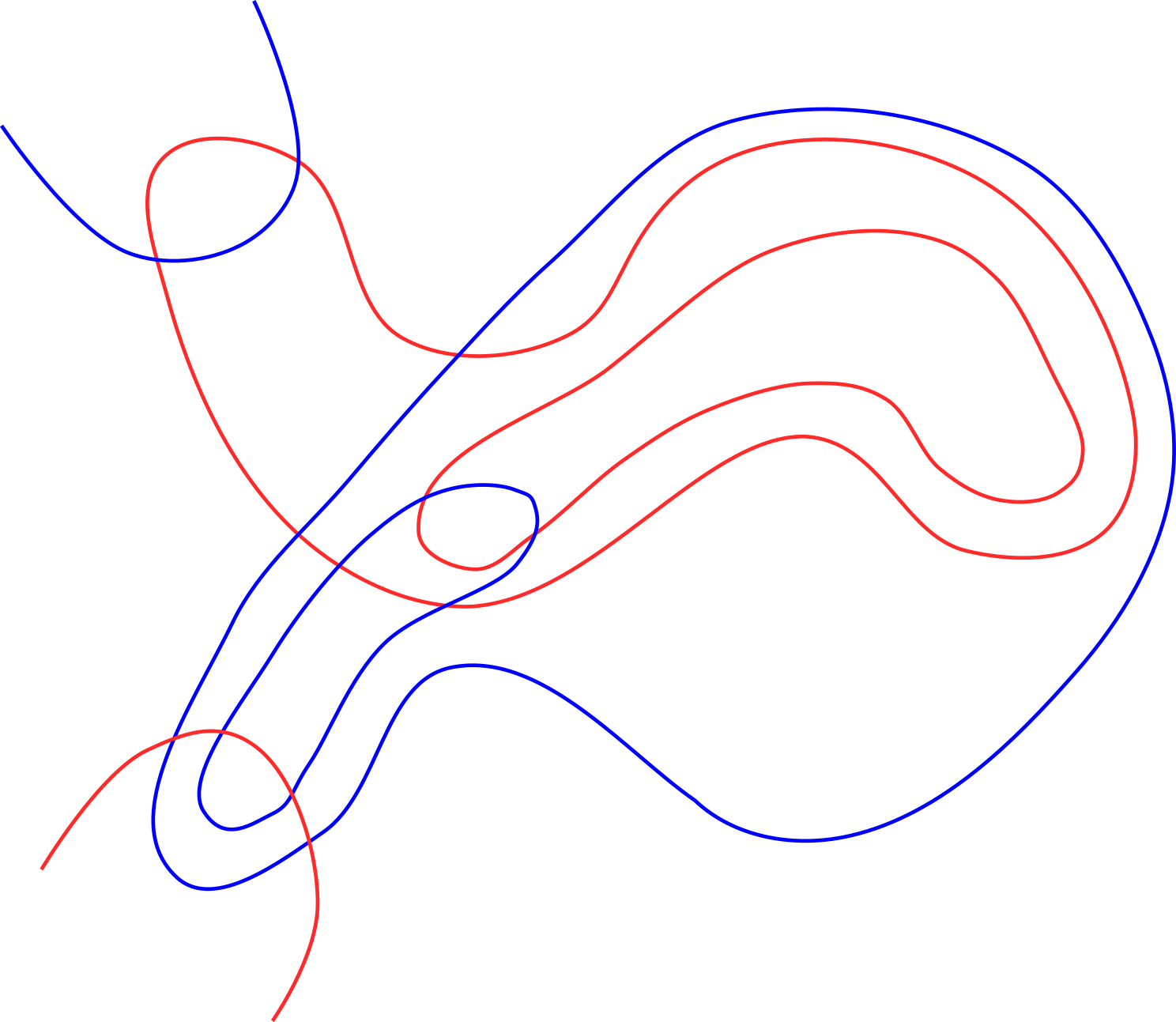}  \fontsize{4}{4}\selectfont
   \put(37.5,41){$X$}
   \put(45.5, 39){$X$}
   \put(18.5,19.5){$O_{3}$}
   \put(84,48){$O_{2}$}
   \put(15,70){$O_{1}$} 
   \normalsize
   \put(117,40){$\Longrightarrow$}
   \end{overpic}
\end{subfigure}
\begin{subfigure}{.33\textwidth}
 \centering
   \begin{overpic}[height=.65\textwidth]{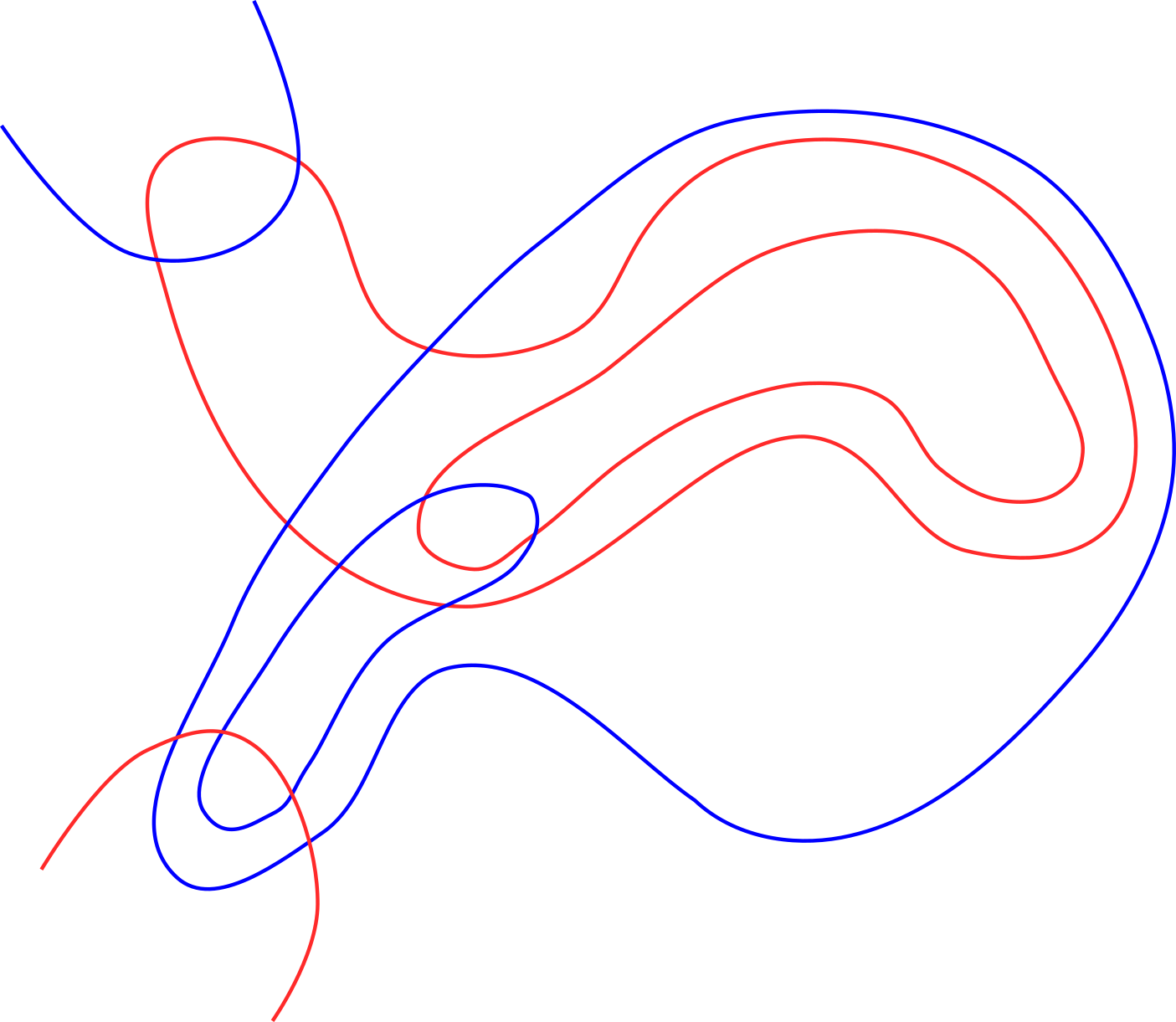} \fontsize{4}{4}\selectfont
   \put(37.5,41){$X$}
   \put(37.5, 51){$X$}
   \put(18.5,19.5){$O_{3}$}
   \put(84,48){$O_{2}$}
   \put(15,70){$O_{1}$} 
   \normalsize
   \put(50,-20){$\big\Downarrow$}
   \end{overpic}
\end{subfigure}%
\newline
\vspace{10mm}
\newline
\begin{subfigure}{.2\textwidth}
 \centering
   \begin{overpic}[height=.9\textwidth]{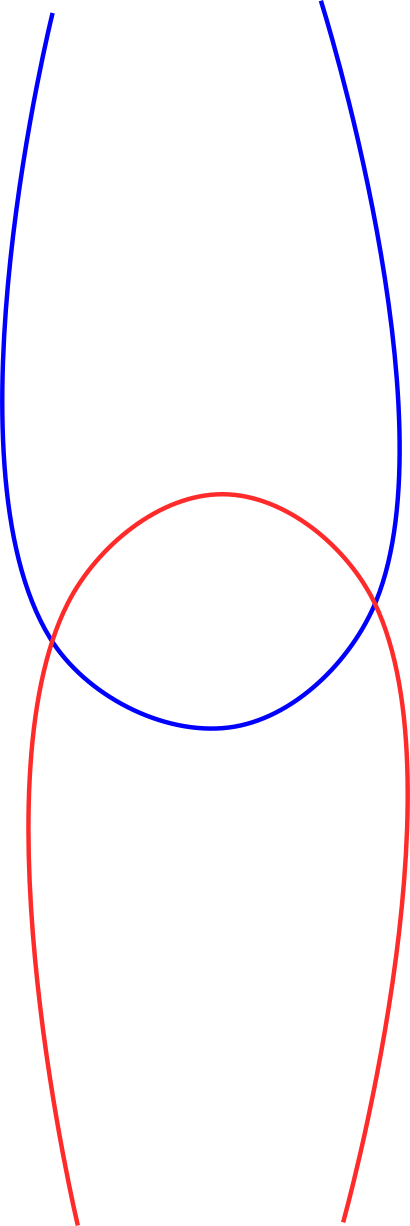} \fontsize{5}{12}\selectfont
   \put(12,48){$O_{1}$}
   \normalsize
   \put(65,45){$\Longleftarrow$}
   \end{overpic}
\end{subfigure}
\begin{subfigure}{.2\textwidth}
 \centering
   \begin{overpic}[height=.9\textwidth]{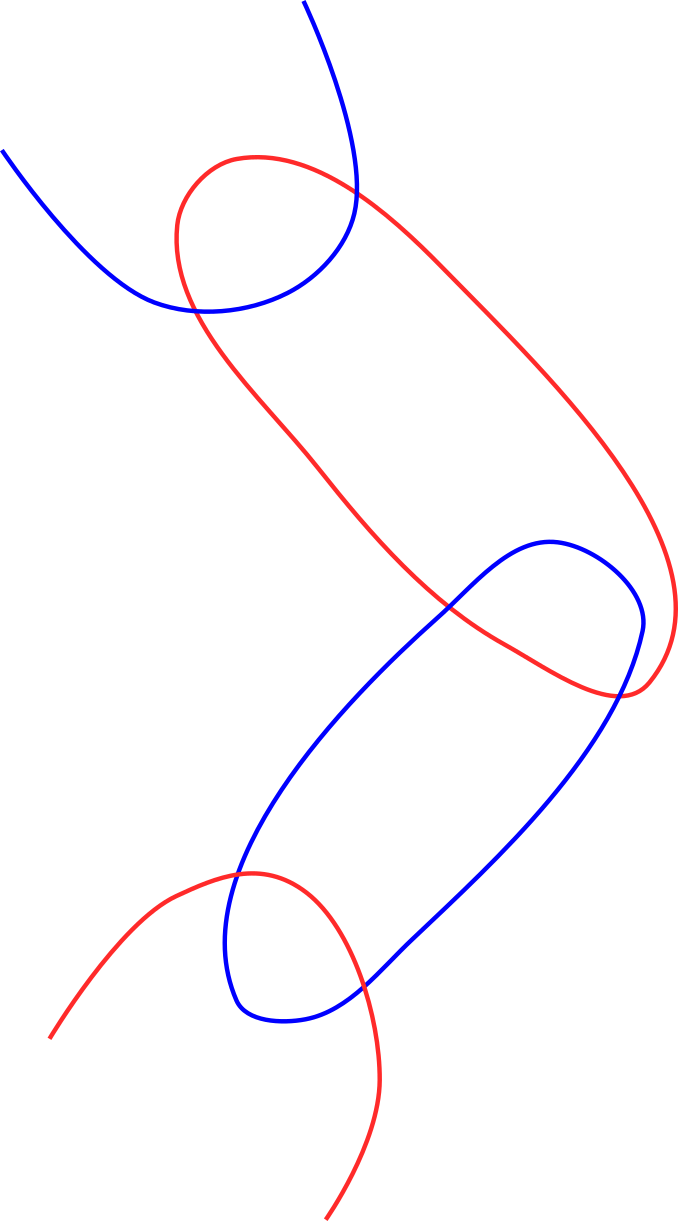} \fontsize{4}{12}\selectfont
   \put(42, 49){$X$}
   \put(19,20){$O_{3}$}
   \put(17,79){$O_{1}$} 
   \normalsize
   \put(80,45){$\Longleftarrow$}
   \put(225,45){$\Longleftarrow$}
   \end{overpic}
\end{subfigure}
\begin{subfigure}{.25\textwidth}
 \centering
   \begin{overpic}[height=.6\textwidth]{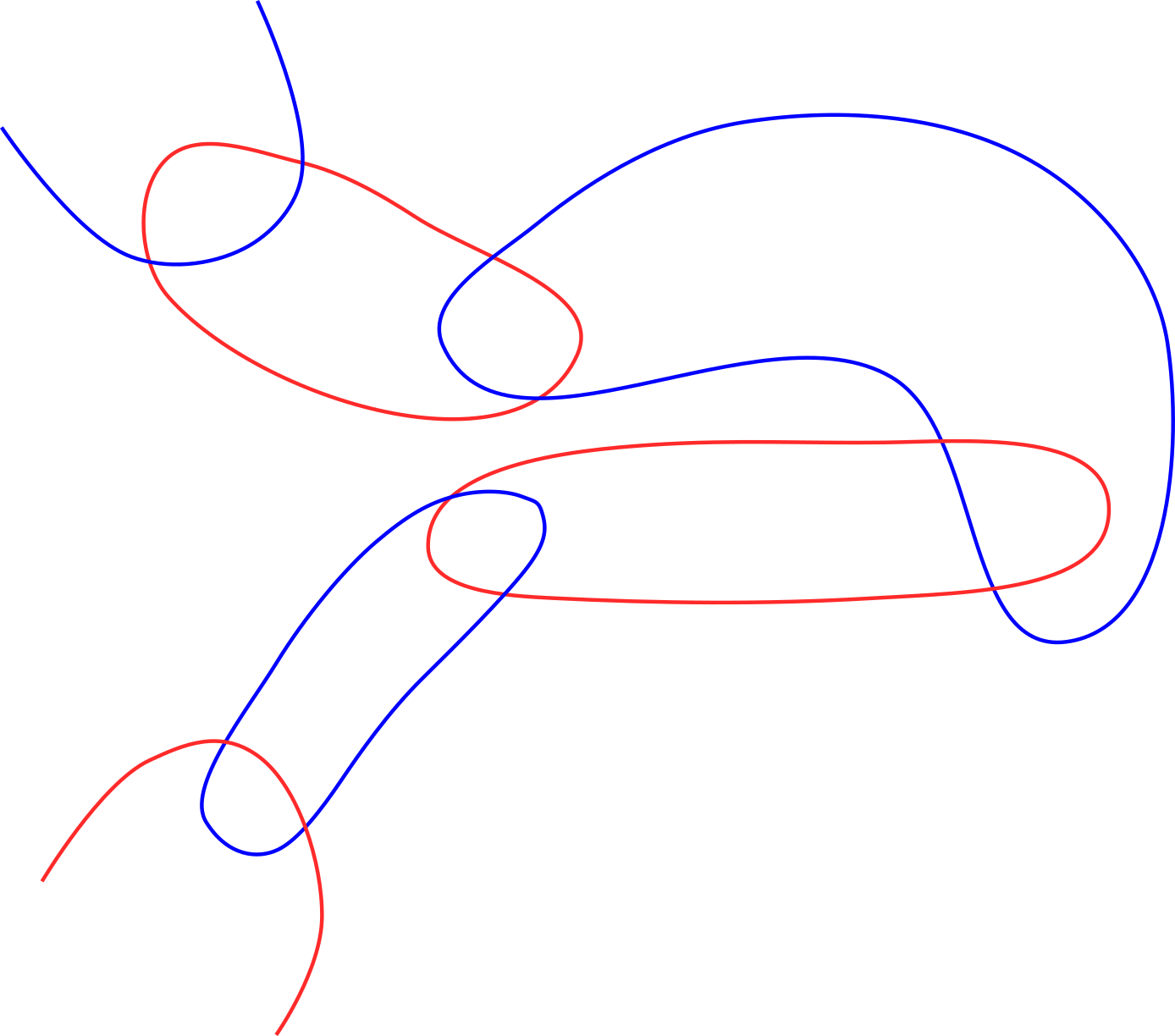} \fontsize{4}{4}\selectfont
   \put(37.5,41){$X$}
   \put(40, 58){$X$}
   \put(17.25,19.5){$O_{3}$}
   \put(84,43){$O_{2}$}
   \put(15,70){$O_{1}$} 
   \end{overpic}
\end{subfigure}
\begin{subfigure}{.25\textwidth}
  \centering
   \begin{overpic}[height=.6\textwidth]{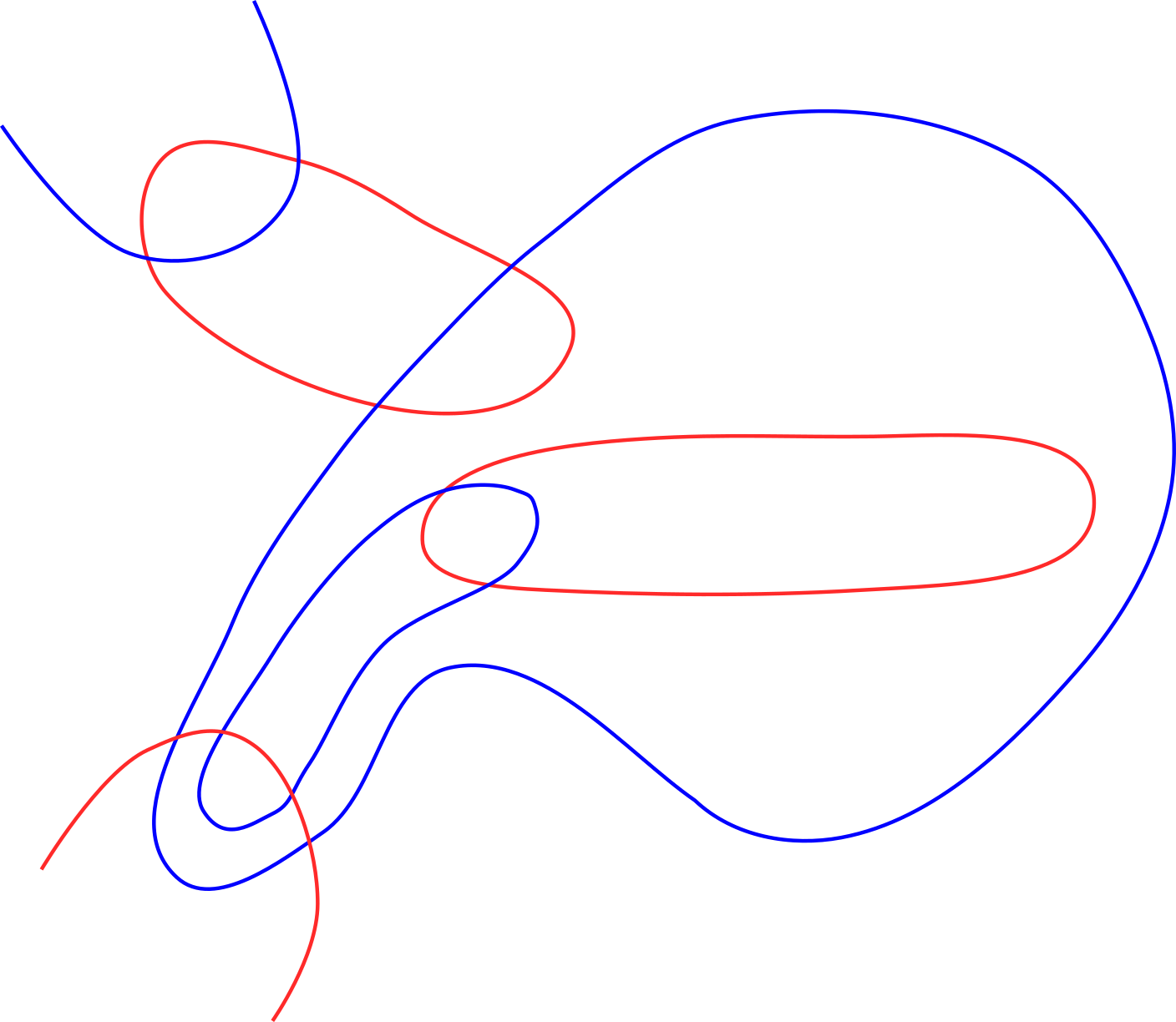} \fontsize{4}{4}\selectfont
   \put(37.5,41){$X$}
   \put(37.5, 55){$X$}
   \put(17.25,19.5){$O_{3}$}
   \put(84,43){$O_{2}$}
   \put(15,70){$O_{1}$} 
   \end{overpic}
\end{subfigure}

 \caption{Isotopies, handleslides, and destabilizations connecting $D_{a}^{I+}$ to $D_{b}^{I+}$ } \label{Moves}
\end{figure}
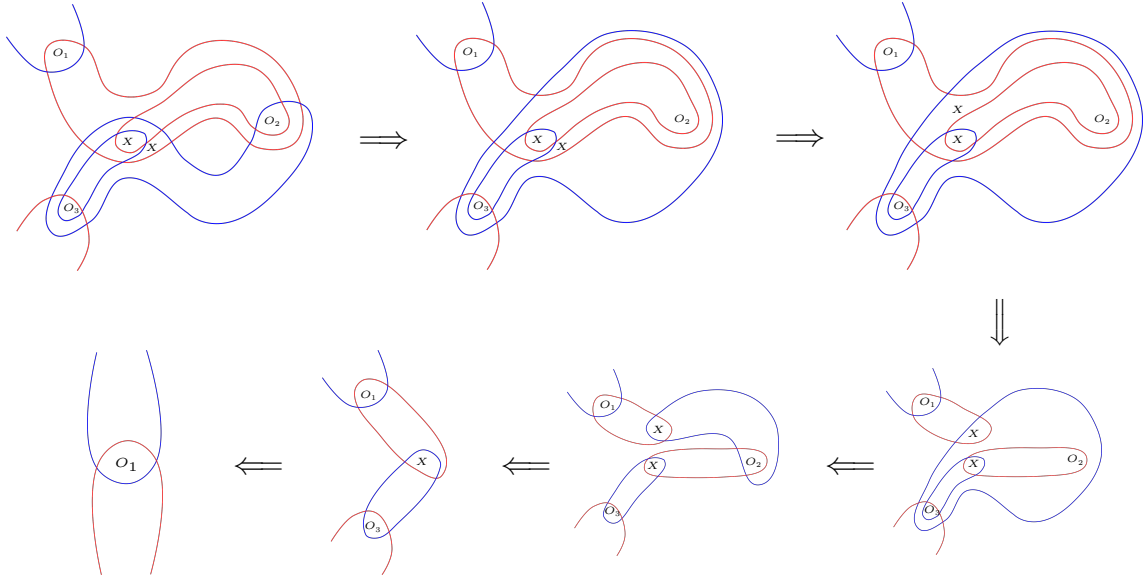

\end{proof}

Combining Lemmas \ref{heegaardmoves} and \ref{lem4.5} proves that $\widetilde{C}_{F}(D)$ satisfies condition (a) of Theorem \ref{thm3.1}. We now want to show that it satisfies (b) as well, i.e. that for each braidlike Reidemeister move, the homology $H_{*}(H_{*}(\widetilde{C}_{F}(D), d_{00}), d_{01}^{*})$ lies in a single local cube grading.

\subsection{Reidemeister I} We will add basepoints to the diagram $D_{a}^{I+}$ in every region in the complement of  $D_{a}^{I+}$ except the center region involved in the Reidemeister move  (see Figure \ref{R1basepoints}). Let $d_{ij}^{f}$ denote the component of $d_{ij}$ which preserves this basepoint filtration. Since $H_{*}(\widetilde{C}_{F}(D), d_{00}) \cong H_{*}(\widetilde{C}_{F}(D), d^{f}_{00})$, it suffices to show that $H_{*}(H_{*}(\widetilde{C}_{F}(D_{a}^{I+}), d^{f}_{00}), (d^{f}_{01})^{*})$ lies in a single local cube grading. 

\begin{figure}
 \centering
   \begin{overpic}[width=.25\textwidth]{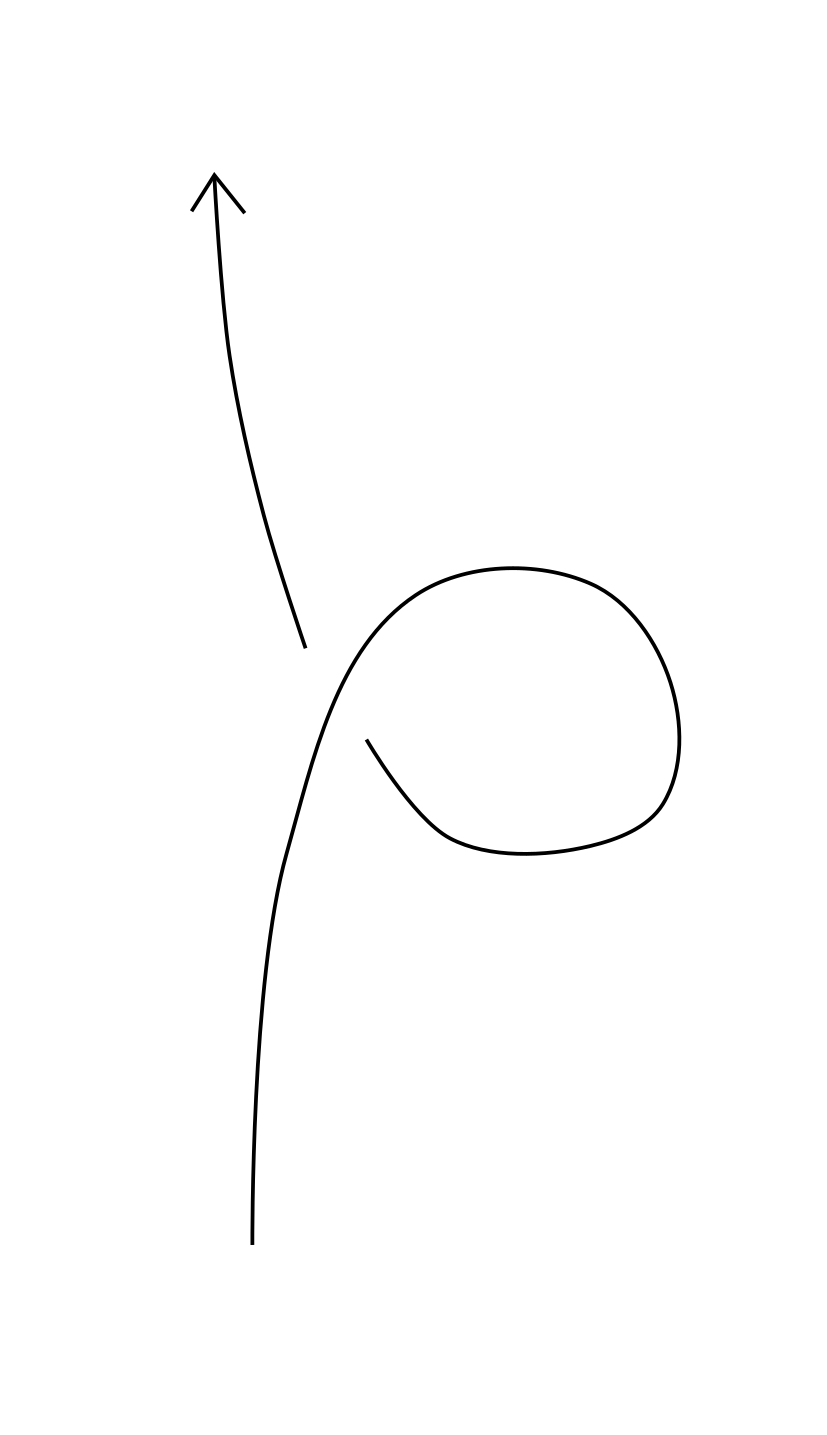}
   \put(12,18){$e_{3}$}
   \put(52,50){$e_{2}$}
   \put(10,80){$e_{1}$}
   \end{overpic}
  \caption{The diagram $D_{a}^{I+}$ with additional basepoints} \label{R1basepoints}
\end{figure}%

\begin{defn}

Let $T$ be a tangle in a diagram $D$ as in the previous definition. If $Z$ is a cycle in $D$, we define the \emph{inner cycle} of $Z$ to be the set of edges of $Z$ that lie in $T$, and the \emph{outer cycle} to be the set of edges in $Z$ in the complement of $T$.

\end{defn}

Given an arbitrary inner cycle and outer cycle, they may not necessarily join up to make a cycle in the whole diagram $D$ - whether or no they do is determined by whether the outgoing edges in the outer cycle are the same as the incoming edges of the inner cycle, and vice-versa. 

\begin{defn}

For a braid-like tangle $T$ with $k$ incoming and $k$ outgoing strands, we define a \emph{connectivity} to be a pair $A_{in},A_{out}$ of subsets of $\{1,...,k\}$. If $Z$ is an inner cycle, then $i \in A_{in}(Z)$ if the $i$th incoming edge is contained in $Z$. Similarly, $j \in A_{out}(Z)$ if the $j$th outgoing edge is contained in $Z$. 

\end{defn}

Since the differentials $ d^{f}_{00}$ and  $d^{f}_{01}$ preserve vertices in the outer cube, it suffices to consider partial resolutions $S_{a}^{I+}$ of $D_{a}^{I+}$ such that all outer crossings are resolved, but the inner crossing is not.

Note that an outer cycle uniquely determines the connectivity of a compatible inner cycle. With the filtration coming from the additional basepoints, two cycles lie in the same filtration level iff they have the same outer cycle. We will prove that for each outer cycle, the homology $H_{*}(H_{*}(\widetilde{C}_{F}(S^{I+}_{a}), d^{f}_{00}), (d^{f}_{01})^{*})$ lies in the same local cube grading.

Let $Z$ be an outer cycle. For the Reidemeister I move, there are exactly two possible connectivities - $A_{in}(Z)=A_{out}(Z)=\emptyset$ and $A_{in}(Z)=A_{out}(Z)=\{1\}$.

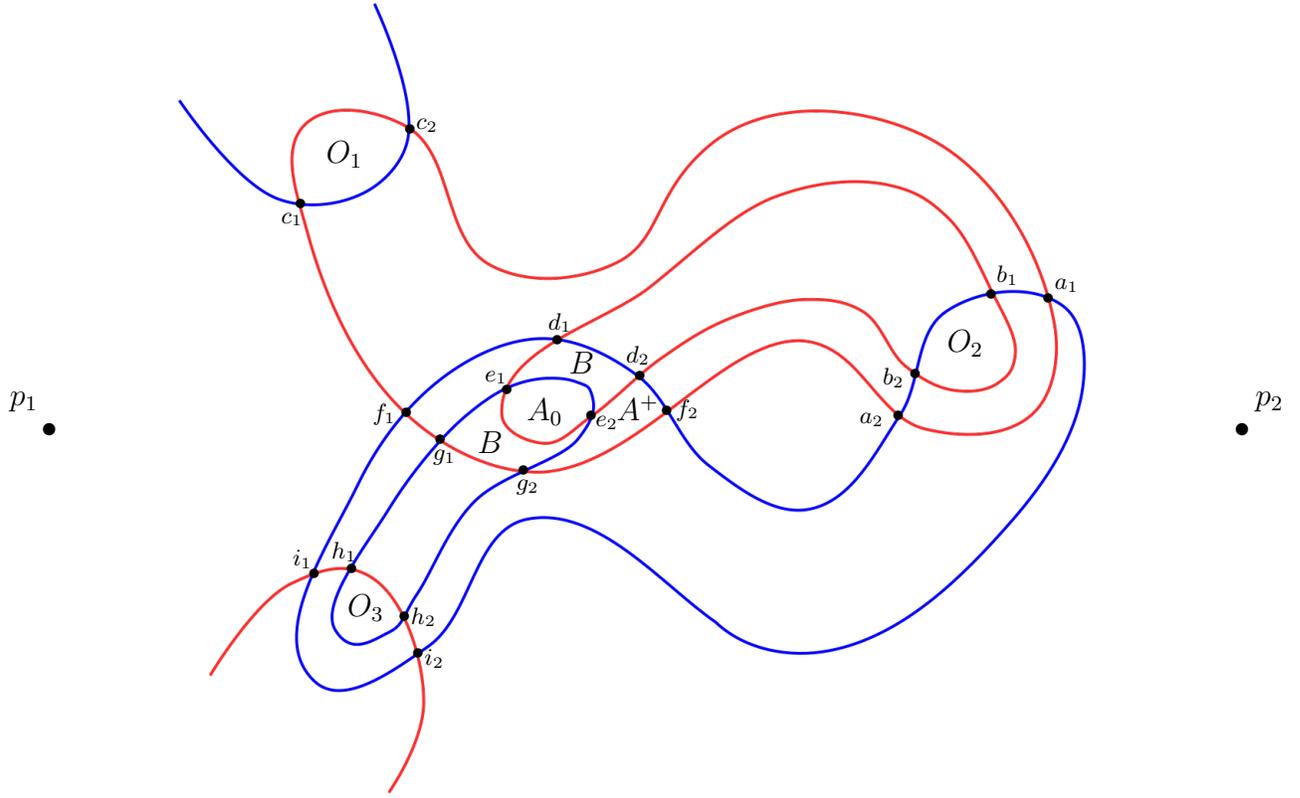
\begin{figure}[h!]
 \centering
   \begin{overpic}[width=.85\textwidth]{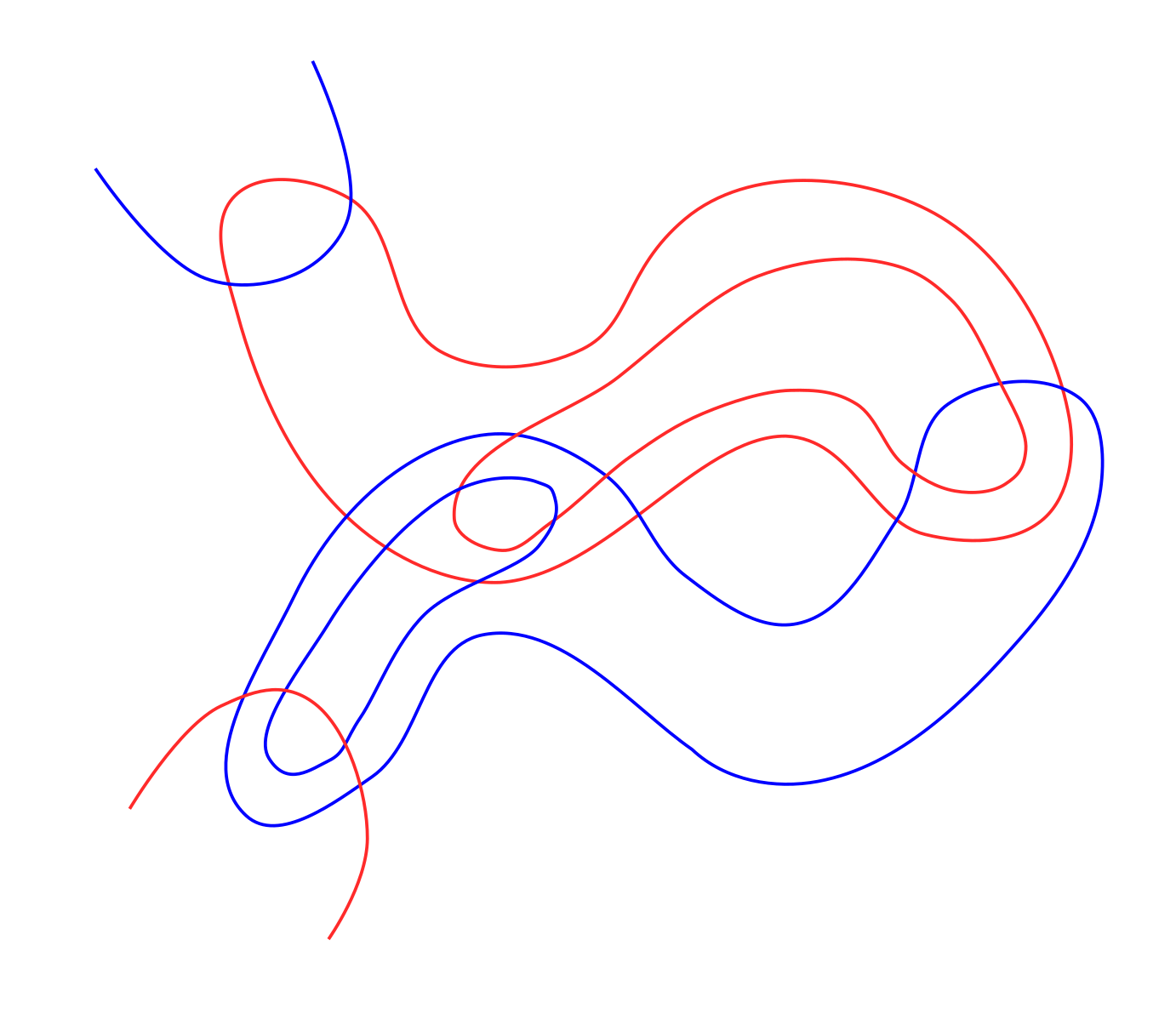}
   \scriptsize
   \put(89.8,52.6){$\bullet$}
   \put(91,54){$a_{1}$}
   \put(84.4,53){$\bullet$}
   \put(85.5, 54.7){$b_{1}$}
   \put(77.2,45.5){$\bullet$}
   \put(74.7, 45){$b_{2}$}
   \put(75.6,41.5){$\bullet$}
   \put(72.5,41.3){$a_{2}$}
   \put(46.5,41.5){$\bullet$}
   \put(47.5, 41){$e_{2}$}
   \put(38.5,44){$\bullet$}
   \put(37,45.3){$e_{1}$}
   \put(43.3,48.7){$\bullet$}
   \put(43, 50.2){$d_{1}$}
   \put(51.1,45.3){$\bullet$}
   \put(50.4,47.1){$d_{2}$}
   \put(32.2,39.2){$\bullet$}
   \put(32.1,37.8){$g_{1}$}
   \put(40.1,36.3){$\bullet$}
   \put(40,35){$g_{2}$}
   \put(29,41.75){$\bullet$}
   \put(26.4,41.6){$f_{1}$}
   \put(53.65,41.95){$\bullet$}
   \put(55.1,42){$f_{2}$}
   \put(18.96, 61.63){$\bullet$}
   \put(17.7,60.3){$c_{1}$}
   \put(29.35, 68.65){$\bullet$}
   \put(30.5, 69.2){$c_{2}$}
   \put(23.8, 27){$\bullet$}
   \put(22.5,28.5){$h_{1}$}
   \put(20.25, 26.5){$\bullet$}
   \put(18.8, 27.8){$i_{1}$}
   \put(28.8, 22.5){$\bullet$}
   \put(30,22.3){$h_{2}$}
   \put(30.1, 19.0){$\bullet$}
   \put(31.3,18.4){$i_{2}$}
     
   \normalsize
   \put(-8,43){$p_{1}$}
   \put(-5,40){$\bullet$}
   \put(110,43){$p_{2}$}
   \put(108,40){$\bullet$}   
   \put(22, 66){$O_{1}$}
   \put(80.8, 48.1){$O_{2}$}
   \put(24, 23){$O_{3}$}
   \put(41, 41.6){$A_{0}$}
   \put(49.5, 41.6){$A^{+}$}
   \put(45,46){$B$}
   \put(36.3, 38.5){$B$}
   \end{overpic}
   \caption{Reidemeister I Heegaard Diagram} \label{HDR1basepoints}
\end{figure}

\noindent
\textbf{Case 1:} $A_{in}(Z)=A_{out}(Z)= \emptyset$.

In $S_{a}^{I+}$, there are two inner cycles with this connectivity: $Z_{\emptyset}$ and $Z_{2}$. We will fix the outer cycle with this connectivity throughout. Let $Z'_{\emptyset}$ denote the unique cycle in $S_{b}^{I+}$ with the same outer cycle.

With the Heegaard diagram for $S_{a}^{I+}$ as labeled in Figure \ref{HDR1basepoints}, the generators with this connectivity are given by $(a,e), (f,e), (b,g)$, and $(d,g)$. After canceling the isomorphisms in $Y$ from $e_{1}$ to $e_{2}^{\gamma}$ corresponding to the bigon containing $A_{0}$, the complex is given by Figure \ref{d0R1}. The diagram only includes the $d^{f}_{00}$ differentials because it gets far too messy with the $d^{f}_{01}$ differentials included as well. Instead, we list the $d^{f}_{01}$ differentials separately here (the computation is straightforward).

\[
\begin{split}
 & d_{01}^{f}(a_{2}e_{2}^{\alpha}) = U_{1}b_{2}g_{2} \\
 & d_{01}^{f}(a_{1}e_{2}^{\alpha}) = U_{1}b_{1}g_{2} + d_{2}g_{2} \\
 & d_{01}^{f}(a_{2}e_{2}^{\beta}) = U_{1}b_{2}g_{1} \\
 & d_{01}^{f}(a_{1}e_{2}^{\beta}) = U_{1}b_{1}g_{1} + d_{1}g_{2} \\
 & d_{01}^{f}(f_{2}e_{2}^{\alpha})= (U_{1}+U_{2})d_{2}g_{2} \\
 & d_{01}^{f}(f_{1}e_{2}^{\alpha})= d_{2}g_{1}+d_{1}g_{2} \\
 & d_{01}^{f}(f_{2}e_{2}^{\beta})= U_{1}d_{2}g_{1}+U_{2}d_{1}g_{2} \\
 & d_{01}^{f}(f_{1}e_{2}^{\beta})= d_{1}g_{1} \\ 
\end{split}
\]

\begin{figure}
\centering
\begin{tikzpicture}
  \matrix (m) [matrix of math nodes,row sep=4em,column sep=5em,minimum width=2em] {
     a_{2}e_{2}^{\alpha} & a_{1}e_{2}^{\alpha} & & b_{2}g_{2} & b_{1}g_{2} \\
     a_{2}e_{2}^{\beta} & a_{1}e_{2}^{\beta} & &b_{2}g_{1} & b_{1}g_{1} \\
                &            \\
     f_{2}e_{2}^{\alpha} & f_{1}e_{2}^{\alpha} & & d_{2}g_{2} & d_{1}g_{2}\\
     f_{2}e_{2}^{\beta} & f_{1}e_{2}^{\beta} & & d_{2}g_{1} & d_{1}g_{1} \\};
     
 \scriptsize    
 \path[-stealth]
    (m-1-1) edge node [right] {$U_{1}+U_{3}$} (m-2-1)
            edge node [above] {$U_{2}$} (m-1-2)
            edge [bend right = 30] node [left] {$1$} (m-4-1)
    (m-2-1) edge node [above] {$U_{2}$} (m-2-2)
            edge [bend right=30] node [left] {$1$} (m-5-1)
    (m-1-2) edge node [left] {$U_{1}+U_{3}$} (m-2-2)
            edge [bend left=30] node [right] {$U_{1}+U_{3}$} (m-4-2)
    (m-2-2) edge [bend left=30] node [right] {$U_{1}+U_{3}$} (m-5-2)    
    (m-4-1) edge node [right] {$U_{1}+U_{3}$} (m-5-1)
            edge node [above] {$U_{2}(U_{1}+U_{3})$} (m-4-2)
    (m-5-1.east|-m-5-2) edge node [above] {$U_{2}(U_{1}+U_{3})$} (m-5-2)
    (m-4-2) edge node [left] {$U_{1}+U_{3}$} (m-5-2)
    (m-1-4) edge node [right] {$U_{1}+U_{3}$} (m-2-4)
            edge node [above] {$U_{2}$} (m-1-5)
            edge [bend right = 30] node [left] {$1$} (m-4-4)
    (m-2-4) edge node [above] {$U_{2}$} (m-2-5)
            edge [bend right=30] node [left] {$1$} (m-5-4)
    (m-1-5) edge node [left] {$U_{1}+U_{3}$} (m-2-5)
    (m-4-4) edge node [right] {$U_{1}+U_{3}$} (m-5-4)
    (m-4-5) edge node [left] {$U_{1}+U_{3}$} (m-5-5);

\end{tikzpicture}

\caption{The complex for Case 1 with $d^{f}_{00}$ differentials shown} \label{d0R1}
\end{figure}
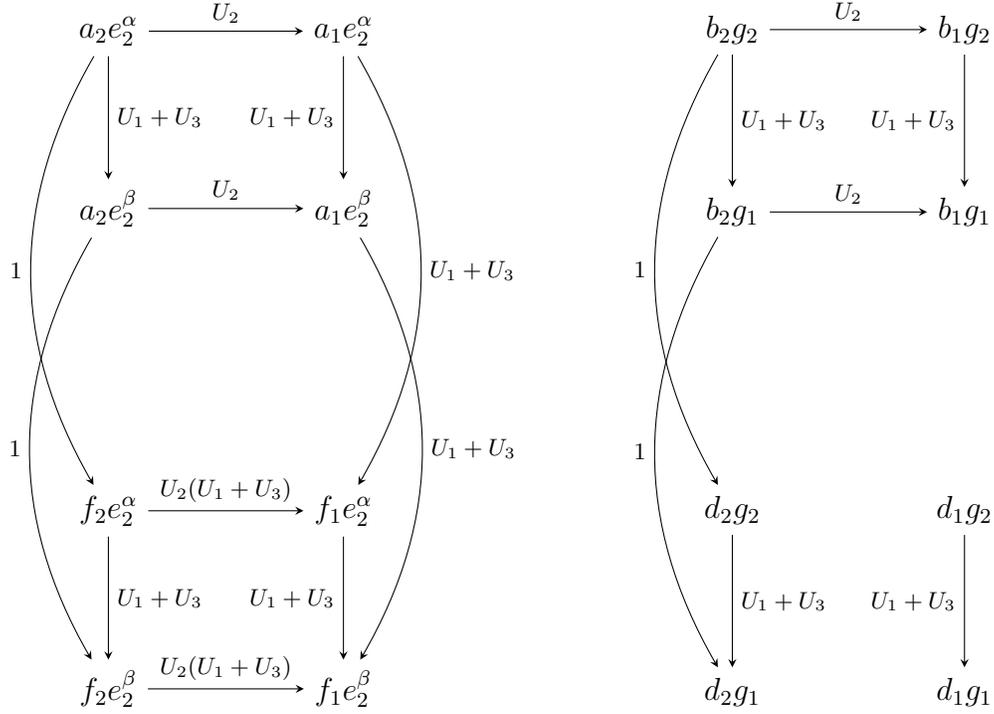

After canceling the $d^{f}_{00}$ differentials labeled with $1$ together with the change of basis $f_{1}e_{2}^{\alpha} \mapsto  f_{1}e_{2}^{\alpha}+a_{1}e_{2}^{\beta}$, the complex is given by Figure \ref{f15} (with $d^{f}_{01}$ differentials now included in the diagram).

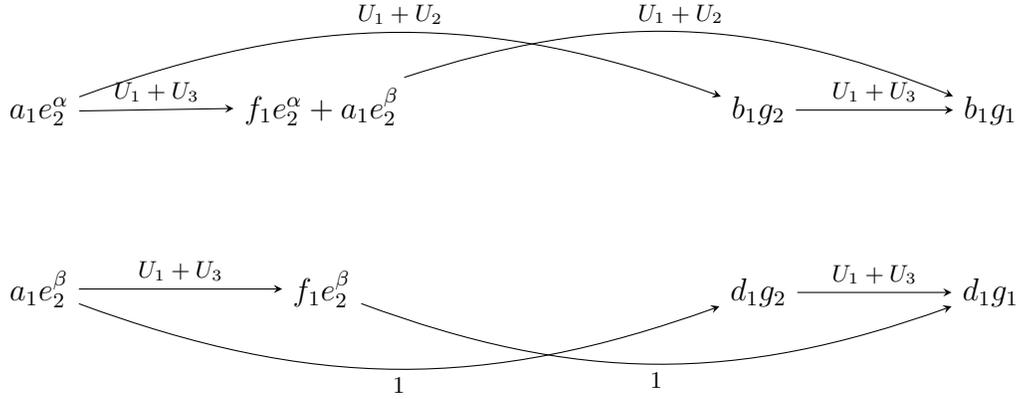
\begin{figure}
\centering

\begin{tikzpicture}
  \matrix (m) [matrix of math nodes,row sep=4em,column sep=5em,minimum width=2em] {
     a_{1}e_{2}^{\alpha} & f_{1}e_{2}^{\alpha}+a_{1}e_{2}^{\beta} & & b_{1}g_{2} & b_{1}g_{1} \\
     a_{1}e_{2}^{\beta}& f_{1}e_{2}^{\beta} & &d_{1}g_{2} & d_{1}g_{1} \\};
     
 \scriptsize    
 \path[-stealth]
    (m-1-1) edge node [above] {$U_{1}+U_{3}$} (m-1-2)
    		edge [bend left = 20] node [above] {$U_{1}+U_{2}$} (m-1-4)
    (m-1-2) edge [bend left = 20] node [above] {$U_{1}+U_{2}$} (m-1-5)		
    (m-2-1) edge node [above] {$U_{1}+U_{3}$} (m-2-2)
    		edge [bend right = 20] node [below] {$1$} (m-2-4)
    (m-2-2) edge [bend right = 20] node [below] {$1$} (m-2-5)
    (m-1-4) edge node [above] {$U_{1}+U_{3}$} (m-1-5)
    (m-2-4) edge node [above] {$U_{1}+U_{3}$} (m-2-5);

\end{tikzpicture}

\caption{The complex for Case 1 after canceling $d_{0}$ isomorphisms} \label{f15}

\end{figure}

For a pair of generators in $\widetilde{C}_{F}(S_{a}^{I+})$ connected by a differential with coefficient $U_{1}+U_{3}$, the homology can be written

\[ \widetilde{H}_{H}(S_{b}^{I+} - Z'_{\emptyset})[U_{2}] \]

\noindent
Thus, we can view the total homology with respect to $d^{f}_{00}$ as shown in Figure \ref{f16}. The differentials shown are the induced $d^{f}_{10}$ differentials.

\begin{figure}
\centering
\vspace{10mm}
\begin{tikzpicture}
  \matrix (m) [matrix of math nodes,row sep=4em,column sep=5em,minimum width=2em] {
     \widetilde{H}_{H}(S_{b}^{I+} - Z'_{\emptyset})[U_{2}] &  &\widetilde{H}_{H}(S_{b}^{I+} - Z'_{\emptyset})[U_{2}] \\
     \widetilde{H}_{H}(S_{b}^{I+} - Z'_{\emptyset})[U_{2}] &  & \widetilde{H}_{H}(S_{b}^{I+} - Z'_{\emptyset})[U_{2}]  \\};
     
 \scriptsize    
 \path[-stealth]
    (m-1-1) edge node [above] {$U_{1}+U_{2}$} (m-1-3)
	
    (m-2-1) edge node [above] {$1$} (m-2-3);

\end{tikzpicture}

\caption{The complex $H_{*}(\widetilde{C}_{F}(S^{I+}_{a}), d^{f}_{00})$ with differential $(d^{f}_{10})^{*}$} \label{f16}

\end{figure}

Since the induced $d^{f}_{01}$ maps are injective, the homology is concentrated in local cube grading $1$ (at the smoothing). This proves Case 1.

\begin{rem}

We can compute the homology for this connectivity explicitly to be 

\[ \widetilde{H}_{H}(S_{b}^{I+} - Z'_{\emptyset})[U_{2}]/U_{1}=U_{2} \]

\noindent
which is precisely the homology

\[ H_{*}(H_{*}(\widetilde{C}_{F}(S_{b}^{I+}), d_{00}^{f}),(d_{01}^{f})^{*}) \]

\noindent
as expected. In fact, a standard invariance proof would have been to prove isomorphisms of this type. However, this requires proving that the higher differentials on the complexes agree, which is possible on algebraically defined complexes like HOMFLY-PT or $sl_{n}$ homology, but very difficult in terms of counting holomorphic discs. This is why we resort to the algebraic argument from Section \ref{section3}.

\end{rem}

\noindent
\textbf{Case 2:} $A_{in}(Z)=A_{out}(Z)= \{1\}$. This case is much easier than the previous case. $S_{1}$ has two cycles with this connectivity, $Z_{13}$ and $Z_{123}$. However, $Z_{13}$ is not admissible, so the contribution to homology is trivial. The cycle $Z_{123}$ only appears in the smoothing, so the homology is concentrated in local cube grading $1$.

The negative Reidemeister I case is similar, with the homology also being concentrated in local cube grading $1$. We leave this as an exercise for the reader.

\subsection{Reidemeister II} As with the Reidemeister I move, we will add basepoints to the diagram $D_{a}^{II+}$ in every region in the complement of  $D_{a}^{II+}$ except the center region involved in the Reidemeister II move  (see Figure \ref{R2basepoints}). Let $d_{ij}^{f}$ denote the component of $d_{ij}$ which preserves this basepoint filtration. Once again, it suffices to show that $H_{*}(H_{*}(\widetilde{C}_{F}(D_{a}^{II+}), d^{f}_{00}), (d^{f}_{01})^{*})$ lies in a single local cube grading. 

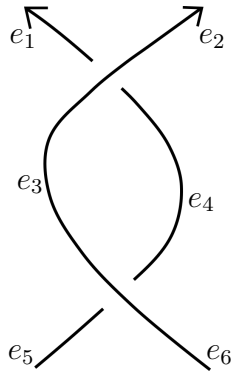
\begin{figure}[h!]

\vspace{5mm}
 \centering
   \begin{overpic}[width=.15\textwidth]{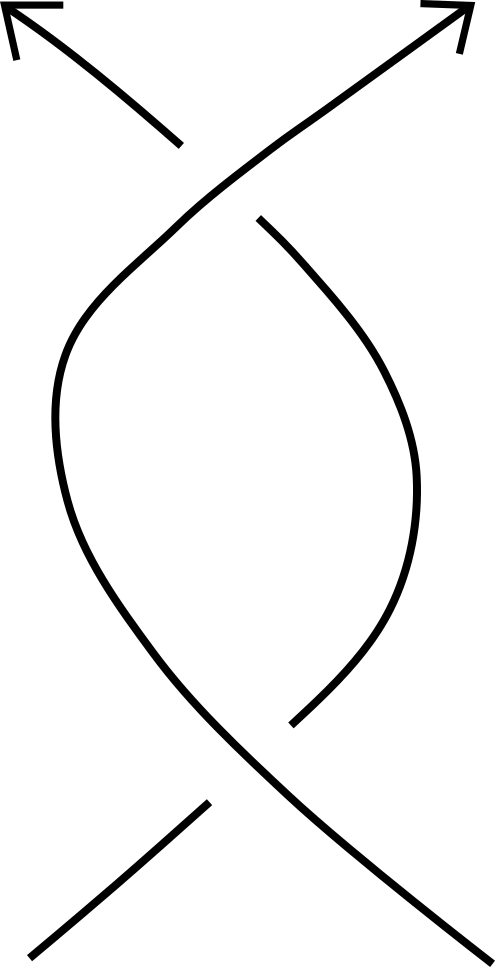}
      \put(-4,90){$e_{1}$}
   \put(48,90){$e_{2}$}
   \put(-2, 50){$e_{3}$}
   \put(45, 45){$e_{4}$}
   \put(-4.5,3){$e_{5}$}
   \put(50,3){$e_{6}$}
   \end{overpic}
   \caption{The diagram $D_{a}^{II+}$ with additional basepoints}\label{R2basepoints}
\end{figure}

Let $S^{II+}_{a}$ denote a partial resolution of $D_{a}^{II+}$ such that all outer crossings are resolved, but the inner crossings are not. We will prove that for each outer cycle, the homology $H_{*}(H_{*}(\widetilde{C}_{F}(S^{II+}_{a}), d^{f}_{00}), (d^{f}_{01})^{*})$ lies in the same local cube grading. There are $6$ connectivities for cycles $Z$ in the Reidemeister II tangle, which we will consider case by case.

\noindent
\textbf{Case 1:} $A_{in}(Z)=A_{out}(Z)= \emptyset$. The only local cycle with this connectivity is the empty cycle, $Z_{\emptyset}$. Using the edge map computations in Figure \ref{table2}, we see that this is locally just the HOMFLY-PT complex with the HOMFLY-PT edge maps. Thus,

\[ H_{*}(H_{*}(\widetilde{C}_{F}(S^{II+}_{a}, Z_{\emptyset}), d^{f}_{00}), (d^{f}_{01})^{*}) \cong H_{H}(S^{II+}_{a} - Z_{\emptyset}) \]

\noindent
It was shown in the invariance proof of HOMFLY-PT homology in \cite{KR2} that this homology lies in local cube grading $1$.

\noindent
\textbf{Case 2:} $A_{in}(Z)=A_{out}(Z)= \{1\}$. This connectivity has two cycles - $Z_{135}$ and $Z_{145}$. The cycle $Z_{135}$ is not admissible, so it does not contribute. The cycle $Z_{145}$ only has non-trivial contribution when both inner crossings are singular, which is the $0$ resolution for the positive crossing and the $1$ resolution for the negative crossing. The homology therefore only lies in local cube grading $1$.

\noindent
\textbf{Case 3:} $A_{in}(Z)= \{2\}$, $A_{out}(Z)=\{1\}$. This connectivity has two cycles in $S^{II+}_{a}$ - $Z_{136}$ and $Z_{146}$. Neither of these cycles are admissible, so the homology corresponding to this connectivity is trivial.

\noindent
\textbf{Case 4:} $A_{in}(Z)= \{1\}$, $A_{out}(Z)=\{2\}$. This connectivity has two cycles in $S^{II+}_{a}$ - $Z_{235}$ and $Z_{245}$. For $i,j \in \{0,1\}$, let $(S^{II+}_{a})^{ij}$ be the complete resolution of $S^{II+}_{a}$ such that the top crossing has the $i$-resolution and the bottom crossing has the $j$-resolution.

We will use the following shorthand notation: let $\widetilde{C}_{F}(Z_{klm}^{ij})$ denote the component of the complex $\widetilde{C}_{F}((S^{II+}_{a})^{ij})$ with underlying cycle $Z_{klm}$ (Using the notation developed so far, this would be written $\widetilde{C}_{F}((S^{II+}_{a})^{ij}, Z_{klm})$ which is quite cumbersome.)

Both of the cycles $Z_{235}$ and $Z_{245}$ are admissible and have one turn, so the total complex has four contributions: $\widetilde{C}_{F}(Z_{235}^{00})$, $\widetilde{C}_{F}(Z_{235}^{01})$, $\widetilde{C}_{F}(Z_{245}^{01})$, and $\widetilde{C}_{F}(Z_{235}^{11})$. The first superscript is 0 when the top crossing is singularized and 1 when it is smoothed, and the second superscript is 0 when the bottom crossing is smoothed and 1 when it is singularized. The complexes are as follows:

\[
\begin{split}
 & \widetilde{C}_{F}(Z^{00}_{235}) \cong \widetilde{C}_{F}^{out}(\{1\}, \{2\}) \otimes Kosz_{R}(U_{2}, U_{3}, U_{5}, U_{1}+U_{2}+U_{3}+U_{4}, U_{4}+U_{6}) \\
 & \widetilde{C}_{F}(Z^{01}_{235}) \cong \widetilde{C}_{F}^{out}(\{1\}, \{2\}) \otimes Kosz_{R}(U_{2}, U_{3}, U_{5}, U_{1}+U_{2}+U_{3}+U_{4}, U_{3}+U_{4}+U_{5}+U_{6}) \\
 & \widetilde{C}_{F}(Z^{01}_{245}) \cong \widetilde{C}_{F}^{out}(\{1\}, \{2\}) \otimes Kosz_{R}(U_{2}, U_{4}, U_{5}, U_{1}+U_{2}+U_{3}+U_{4}, U_{3}+U_{4}+U_{5}+U_{6}) \\
 & \widetilde{C}_{F}(Z^{11}_{245}) \cong \widetilde{C}_{F}^{out}(\{1\}, \{2\}) \otimes Kosz_{R}(U_{2}, U_{4}, U_{5}, U_{1}+U_{3}, U_{3}+U_{4}+U_{5}+U_{6}) \\
\end{split}
\]

\noindent
where $C_{F}^{out}(\{1\}, \{2\})$ denotes the portion of the complex coming from the outer cycle with connectivity $A_{in}(Z)= \{1\}$, $A_{out}(Z)=\{2\}$.

Since $U_{2}$ and $U_{5}$ appear in all of these and the edges in a cycle form a regular sequence, we will cancel them.

\[
\begin{split}
 & \widetilde{C}_{F}(S_{a}, Z^{00}_{235}) \cong \widetilde{C}_{F}^{out}(\{1\}, \{2\}) \otimes Kosz_{R/(U_{2}=U_{5}=0)}(U_{3}, U_{1}+U_{3}+U_{4}, U_{4}+U_{6}) \\
 & \widetilde{C}_{F}(S_{a}, Z^{01}_{235}) \cong \widetilde{C}_{F}^{out}(\{1\}, \{2\}) \otimes Kosz_{R/(U_{2}=U_{5}=0)}(U_{3},  U_{1}+U_{3}+U_{4}, U_{3}+U_{4}+U_{6}) \\
 & \widetilde{C}_{F}(S_{a}, Z^{01}_{245}) \cong \widetilde{C}_{F}^{out}(\{1\}, \{2\}) \otimes Kosz_{R/(U_{2}=U_{5}=0)}(U_{4},  U_{1}+U_{3}+U_{4}, U_{3}+U_{4}+U_{6}) \\
 & \widetilde{C}_{F}(S_{a}, Z^{11}_{245}) \cong \widetilde{C}_{F}^{out}(\{1\}, \{2\}) \otimes Kosz_{R/(U_{2}=U_{5}=0)}(U_{4},  U_{1}+U_{3}, U_{3}+U_{4}+U_{6}) \\
\end{split}
\]

From here, we will add in the $d_{00}^{f}$ and $d_{10}^{f}$ differentials. Computing them directly from the diagram would not require any advanced techniques, but there are a lot of generators and discs to count. Fortunately, there is an easier way using the edge map computations in Section \ref{filterededgemaps} together with the fact that $d^{2}=0$. The only discs we will have to count are those within the singular diagram $S_{a}^{01}$ which map between $\widetilde{C}_{F}(Z^{01}_{235})$ and $\widetilde{C}_{F}(Z^{01}_{245})$. These are shown in Figure \ref{d10}, where  

\[ A = \widetilde{C}_{F}^{out}(\{1\}, \{2\}) \otimes Kosz_{R/(U_{2}=U_{5}=0)}(U_{1}+U_{3}+U_{4})\] 

\begin{figure}

\centering
\begin{tikzpicture}
  \matrix (m) [matrix of math nodes,row sep=4em,column sep=6em,minimum width=2em] {
     A & A \\
     A & A \\
                &            \\
     A & A \\
     A & A \\};
     \scriptsize
  \path[-stealth]
    (m-1-1) edge node [right] {$U_{3}+U_{4}+U_{6}$} (m-2-1)
            edge node [above] {$U_{3}$} (m-1-2)
            edge [bend right = 45] node [right] {$1$} (m-4-2)
    (m-2-1) edge node [below] {$U_{3}$} (m-2-2)
    	    edge [bend left = 45] node [right] {$1$} (m-5-2)
    (m-1-2) edge node [right] {$U_{3}+U_{4}+U_{6}$} (m-2-2)
    (m-4-1) edge node [left] {$U_{3}+U_{4}+U_{6}$} (m-5-1)
            edge node [above] {$U_{4}$} (m-4-2)
    (m-5-1.east|-m-5-2) edge node [below] {$U_{4}$} (m-5-2)
    (m-4-2) edge node [left] {$U_{3}+U_{4}+U_{6}$} (m-5-2);

\end{tikzpicture}
\caption{The chain complex $\widetilde{C}_{F}(S^{01})$}\label{d10}
\end{figure}
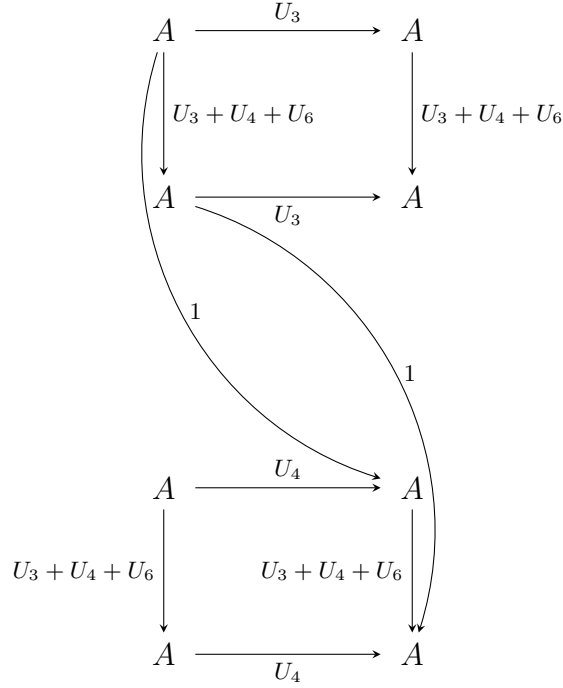

Recall that there is a canonical isomorphism between $\widetilde{C}_{F}(Z^{00}_{235})$ and $\widetilde{C}_{F}(Z^{01}_{235})$. In terms of the Koszul complexes above, this isomorphism corresponds to adding the $U_{3}$ to the $U_{4}+U_{6}$ in $\widetilde{C}_{F}(Z^{00}_{235})$. The component of $d_{10}^{f}$ which maps from $\widetilde{C}_{F}(Z^{00}_{235})$ to $\widetilde{C}_{F}(Z^{01}_{235})$ is this isomorphism composed with multiplication by $U_{6}$. Thus, the complex involving $\widetilde{C}_{F}(Z^{00}_{235})$, $\widetilde{C}_{F}(Z^{01}_{235})$, and $\widetilde{C}_{F}(Z^{01}_{245})$ is given by Figure \ref{d14}, where the only differentials not included are those mapping from $\widetilde{C}_{F}(Z^{00}_{235})$ to $\widetilde{C}_{F}(Z^{01}_{245})$. In this diagram, we have performed a change of basis in $\widetilde{C}_{F}(Z^{00}_{235})$ from $Kosz(U_{3}, U_{4}+U_{6})$ to $Kosz(U_{3}, U_{3}+U_{4}+U_{6})$ to reduce the number of arrows.

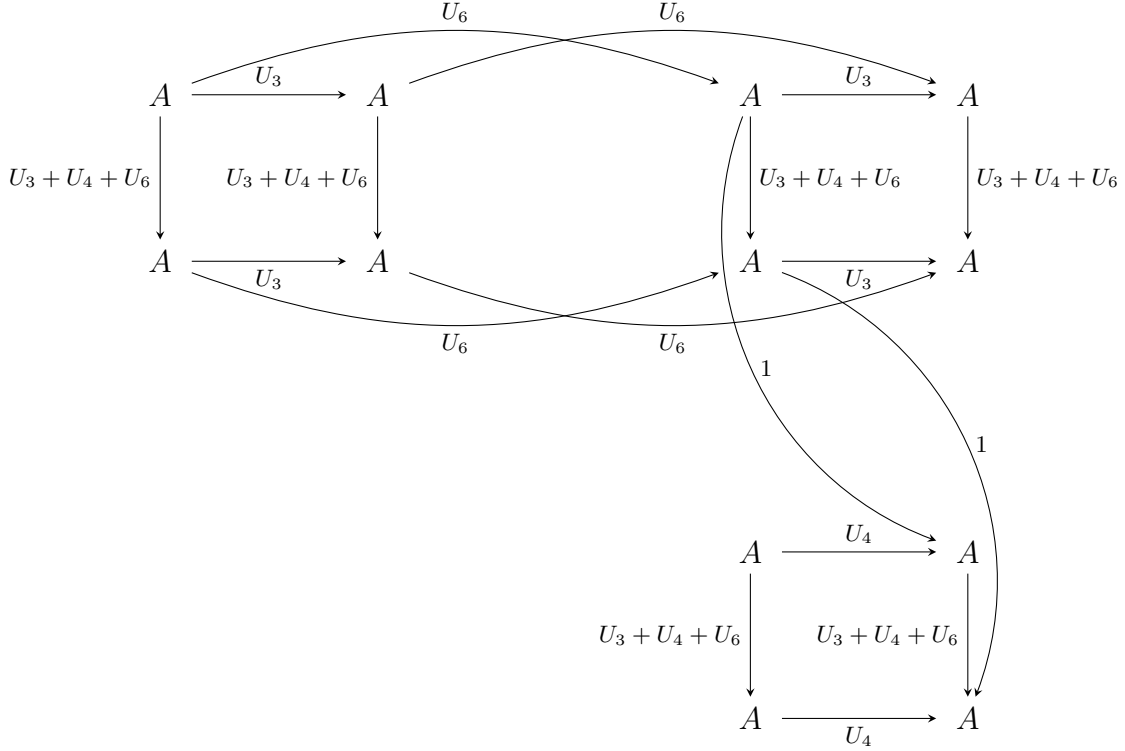
\begin{figure}

\centering
\begin{tikzpicture}
  \matrix (m) [matrix of math nodes,row sep=4em,column sep=5em,minimum width=2em] {
    A & A& & A & A \\
    A & A & &A & A \\
                &            \\
     & & & A& A \\
     & & & A & A \\};
     \scriptsize
  \path[-stealth]  
    (m-1-1) edge node [left] {$U_{3}+U_{4}+U_{6}$} (m-2-1)
            edge node [above] {$U_{3}$} (m-1-2)
            edge [bend left = 20] node [above] {$U_{6}$} (m-1-4)
    (m-2-1) edge node [below] {$U_{3}$} (m-2-2)
    	    edge [bend right = 20] node [below] {$U_{6}$} (m-2-4)
    (m-1-2) edge node [left] {$U_{3}+U_{4}+U_{6}$} (m-2-2)
    	    edge [bend left = 20] node [above] {$U_{6}$} (m-1-5)
    (m-2-2) edge [bend right = 20] node [below] {$U_{6}$} (m-2-5)
    (m-1-4) edge node [right] {$U_{3}+U_{4}+U_{6}$} (m-2-4)
            edge node [above] {$U_{3}$} (m-1-5)
            edge [bend right = 45] node [right] {$1$} (m-4-5)
    (m-2-4) edge node [below] {$U_{3}$} (m-2-5)
    	    edge [bend left = 45] node [right] {$1$} (m-5-5)
    (m-1-5) edge node [right] {$U_{3}+U_{4}+U_{6}$} (m-2-5)
    (m-4-4) edge node [left] {$U_{3}+U_{4}+U_{6}$} (m-5-4)
            edge node [above] {$U_{4}$} (m-4-5)
    (m-5-4.east|-m-5-2) edge node [below] {$U_{4}$} (m-5-5)
    (m-4-5) edge node [left] {$U_{3}+U_{4}+U_{6}$} (m-5-5);

\end{tikzpicture}
\caption{Chain complex involving $\widetilde{C}_{F}(Z^{00}_{235})$, $\widetilde{C}_{F}(Z^{01}_{235})$, and $\widetilde{C}_{F}(Z^{01}_{245})$, counting all differentials except those from $\widetilde{C}_{F}(Z^{00}_{235})$ to $\widetilde{C}_{F}(Z^{01}_{245})$ }\label{d14}
\end{figure}

We see that the differentials included so far uniquely determine the differentials from $\widetilde{C}_{F}(Z^{00}_{235})$ to $\widetilde{C}_{F}(Z^{01}_{245})$ subject to $d^{2}=0$. These differentials are included in Figure \ref{d14}. After taking homology with respect to the $d_{00}^{f}$ differentials which preserve the underlying cycle, we see that the $d_{10}^{f}$ differentials induce an isomorphism between $\widetilde{H}_{F}(Z^{00}_{235})$ and $\widetilde{H}_{F}(Z^{01}_{245})$.

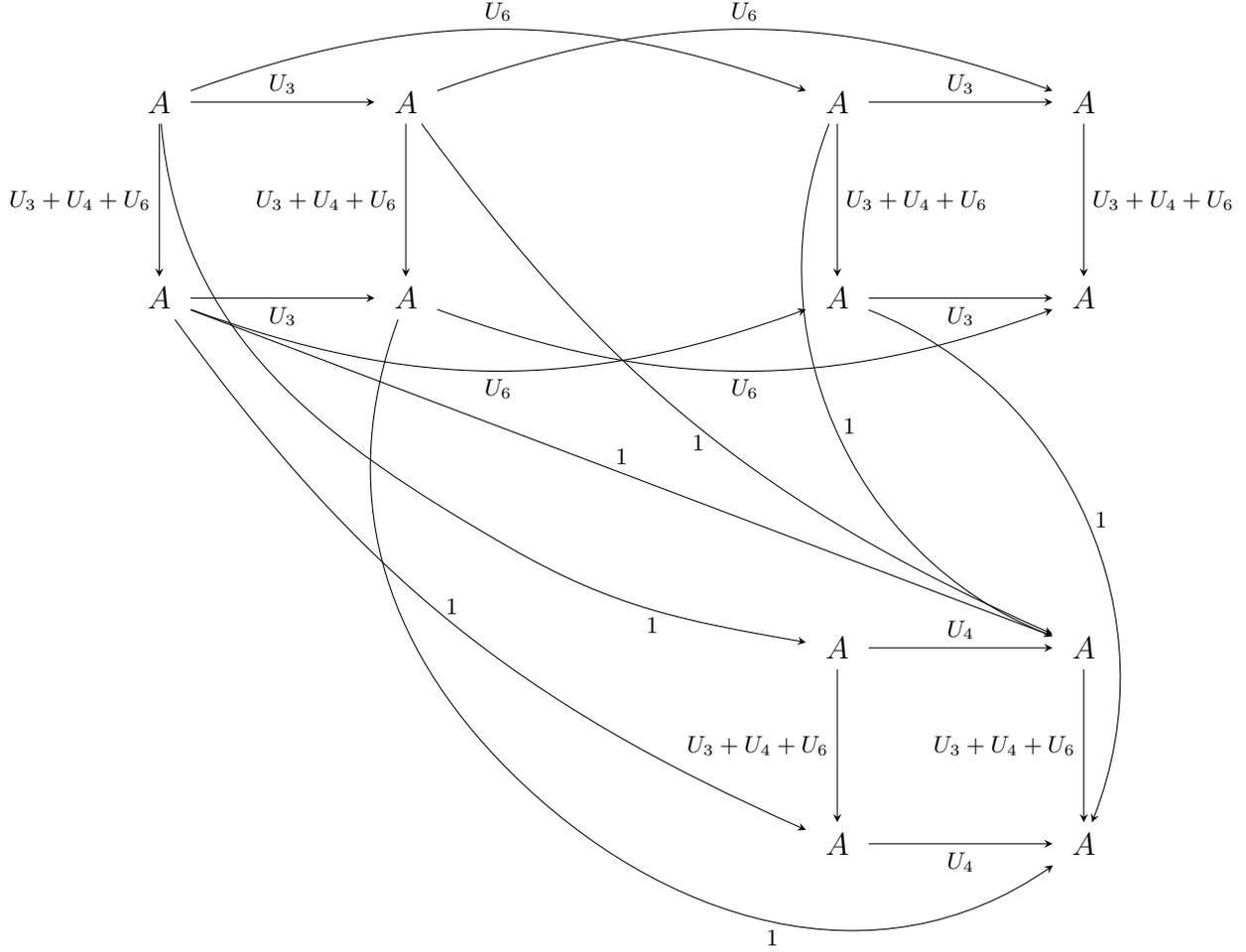
\begin{figure}

\centering
\begin{tikzpicture}
  \matrix (m) [matrix of math nodes,row sep=5em,column sep=6em,minimum width=2em] {
    A & A& & A & A \\
    A & A & &A & A \\
                &            \\
     & & & A& A \\
     & & & A & A \\};
     \scriptsize
  \path[-stealth]  
    (m-1-1) edge node [left] {$U_{3}+U_{4}+U_{6}$} (m-2-1)
            edge node [above] {$U_{3}$} (m-1-2)
            edge [bend left = 20] node [above] {$U_{6}$} (m-1-4)
    (m-2-1) edge node [below] {$U_{3}$} (m-2-2)
    	    edge [bend right = 20] node [below] {$U_{6}$} (m-2-4)
	    edge node [above] {$1$} (m-4-5)
	    edge [bend right = 15] node [above] {$1$} (m-5-4)
    (m-1-2) edge node [left] {$U_{3}+U_{4}+U_{6}$} (m-2-2)
    	    edge [bend left = 20] node [above] {$U_{6}$} (m-1-5)
	    edge [bend right = 15] node [below] {$1$} (m-4-5)
    (m-2-2) edge [bend right = 20] node [below] {$U_{6}$} (m-2-5)
    (m-1-4) edge node [right] {$U_{3}+U_{4}+U_{6}$} (m-2-4)
            edge node [above] {$U_{3}$} (m-1-5)
            edge [bend right = 45] node [right] {$1$} (m-4-5)
    (m-2-4) edge node [below] {$U_{3}$} (m-2-5)
    	    edge [bend left = 45] node [right] {$1$} (m-5-5)
    (m-1-5) edge node [right] {$U_{3}+U_{4}+U_{6}$} (m-2-5)
    (m-4-4) edge node [left] {$U_{3}+U_{4}+U_{6}$} (m-5-4)
            edge node [above] {$U_{4}$} (m-4-5)
    (m-5-4) edge node [below] {$U_{4}$} (m-5-5)
    (m-4-5) edge node [left] {$U_{3}+U_{4}+U_{6}$} (m-5-5);
    
     \draw [-stealth] (m-1-1) 
           to [out = 275, in = 150] (-1.5,-1) 
           to [out = 330, in = 170] node [below,midway]{$1$} (m-4-4); 
           
     \draw [-stealth] (m-2-2) 
           to [out = 250, in = 135] (-1.5,-4) 
           to [out = -45, in = 215] node[below,midway]{$1$} (m-5-5);

\end{tikzpicture}
\caption{Chain complex involving $\widetilde{C}_{F}(Z^{00}_{235})$, $\widetilde{C}_{F}(Z^{01}_{235})$, and $\widetilde{C}_{F}(Z^{01}_{245})$ with all differentials included }\label{d15}
\end{figure}

A similar argument counting the differentials between $\widetilde{C}_{F}(S_{1}^{01})$ and $\widetilde{C}_{F}(S_{1}^{11})$ shows that $d_{10}^{f}$ also induces an isomorphism from $\widetilde{H}_{F}(Z^{01}_{235})$ to $\widetilde{H}_{F}(Z^{11}_{245})$. This shows that the complex $(H_{*}(\widetilde{C}_{F}(S^{II+}_{a}), d_{00}^{f}), (d_{10}^{f})^{*})$ restricted to this connectivity is given in Figure \ref{d16}, which is clearly contractible. Thus, the homology with this connectivity is trivial.

\begin{figure}

\centering
\begin{tikzpicture}
  \matrix (m) [matrix of math nodes,row sep=4em,column sep=6em,minimum width=2em] {
     \widetilde{H}_{F}(Z^{00}_{235})& \widetilde{H}_{F}(Z^{01}_{235}) \\
     \widetilde{H}_{F}(Z^{01}_{245}) & \widetilde{H}_{F}(Z^{01}_{245}) \\};
     \scriptsize
  \path[-stealth]
    (m-1-1) edge node [right] {$\cong$} (m-2-1)
            edge node [above] {$U_{6}$} (m-1-2)
    (m-2-1) edge node [below] {$U_{1}$} (m-2-2)
    (m-1-2) edge node [right] {$\cong$} (m-2-2);
    
\end{tikzpicture}
\caption{The simplified chain complex for case 4 after taking $d_{00}^{f}$ differentials.}\label{d16}
\end{figure}

\noindent
\textbf{Case 5:} $A_{in}(Z)= \{2\}$, $A_{out}(Z)=\{2\}$. This case follows from the same argument as Case 2.

\noindent
\textbf{Case 6:} $A_{in}(Z)=A_{out}(Z)=\{1,2\}$. This connectivity has one cycle in $S^{II+}_{a}$ consisting of all six edges, which only appears in the complete resolution of $S^{II+}_{a}$ in which both inner crossings have been smoothed. This complete resolution lies in local cube grading $1$.

This completes the proof that $H_{*}(H_{*}(\widetilde{C}_{F}(D_{a}^{II+}), d^{f}_{00}), (d^{f}_{01})^{*})$ is concentrated in local cube grading $1$. Analogous case study arguments show that $H_{*}(H_{*}(\widetilde{C}_{F}(D_{a}^{II-}), d^{f}_{00}), (d^{f}_{01})^{*})$ also lies in local cube grading $1$, giving invariance under both braidlike Reidemeister II moves.

\subsection{Reidemeister III} As with the previous two moves, we will add basepoints to the diagram $D_{a}^{III}$ in every region in the complement of  $D_{a}^{III}$ except the regions involved in the Reidemeister III move  (see Figure \ref{R3basepoints}). Let $d_{ij}^{f}$ denote the component of $d_{ij}$ which preserves this basepoint filtration. Once again, it suffices to show that $H_{*}(H_{*}(\widetilde{C}_{F}(D_{a}^{III}), d^{f}_{00}), (d^{f}_{01})^{*})$ lies in a single local cube grading. 

\begin{figure}[h!]

\vspace{5mm}
 \centering
   \begin{overpic}[width=.3\textwidth]{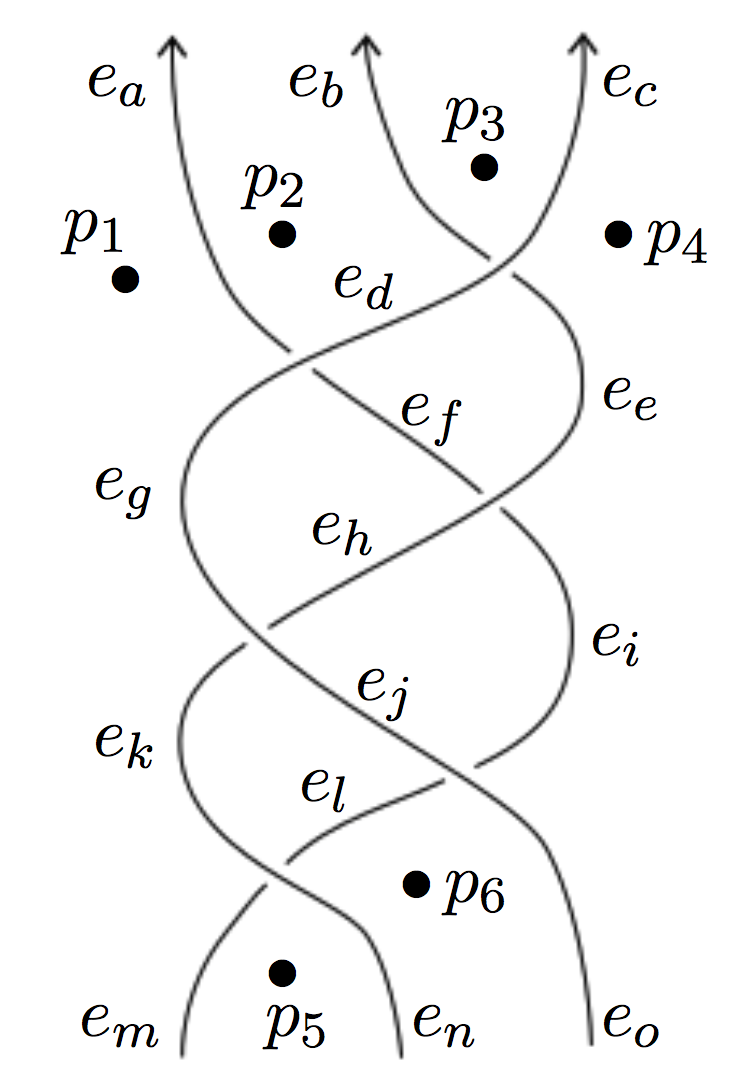}
   \end{overpic}
   \caption{The diagram $D_{a}^{III}$ with additional basepoints}\label{R3basepoints}
\end{figure}

As with the previous cases, let $S^{III}_{a}$ denote a partial resolution of $D_{a}^{III}$ such that all outer crossings are resolved, but the inner crossings are not. We will prove that for each outer cycle, the homology $H_{*}(H_{*}(\widetilde{C}_{F}(S^{III}_{a}), d^{f}_{00}), (d^{f}_{01})^{*})$ lies in the same local cube grading. For the Reidemeister III diagram, there are $20$ connectivities to consider, so it will be useful to have some algebraic lemmas to make computations go more quickly.

\begin{lem}[\hspace{1sp}\cite{Rasmussen}] \label{injective}
Given a (possibly singular) diagram $D$, for any edge $e_{i}$ in $D$, the multiplication map 

\[U_{i}: \widetilde{H}_{H}(D) \longrightarrow \widetilde{H}_{H}(D) \]

\noindent
on unreduced HOMFLY-PT homology is injective. This is true for the middle HOMFLY-PT homology as well. Additionally, as a triply-graded vector space, the homology of the mapping cone does not depend on the choice of edge $e_{i}$.
\end{lem}

\begin{lem}\label{easyhomlemma}

Suppose $(C, d_{1}+...+d_{k})$ is a finite-dimensional, filtered chain complex such that the total homology is trivial. Suppose also that the homology is concentrated in 3 adjacent gradings $C_{1}$, $C_{2}$, and $C_{3}$, with

\[dim(C_{2}) = dim(C_{1})+dim(C_{3}) \]

\noindent
Then $H_{*}(C, d_{1})=0$.

\end{lem}

\begin{proof}

The differential $d_{1}$ is the only one that can map into or out of $C_{2}$ non-trivially, so in order for the total homology to be trivial, $H_{*}(C, d_{1})$ must be trivial in grading $2$. But the equality 

\[dim(C_{2}) = dim(C_{1})+dim(C_{3}) \]

\noindent
is preserved when taking homology with respect to $d_{1}$, so $H_{*}(C, d_{1})$ must also be trivial in gradings $1$ and $3$.
\end{proof}

We can now begin our case studies.

\noindent
\textbf{Case 1:} $A_{in}(Z)=A_{out}(Z)= \emptyset$. In this case we get

\[H_{*}(H_{*}(\widetilde{C}(S^{III}_{a}, Z_{\emptyset}), d^{f}_{00}), d_{10}^{*})= \widetilde{H}_{H}(S^{III}_{a}-Z_{\emptyset}) \]
\noindent
We know from the invariance of HOMFLY-PT homology under braid-like Reidemeister moves that $\widetilde{H}_{H}(S^{III}_{a}-Z_{\emptyset})$ lies in a single local cube grading, because $S^{III}_{b}$ does. It lies in the grading where all the crossings are smoothed, since that is the sub-quotient complex where $S^{III}_{b}$ lies within the cube complex for $S^{III}_{a}$. Since there are 3 positive crossings and 3 negative crossings, the homology is concentrated in inner cube grading $3$.

\noindent
\textbf{Case 2:} $A_{in}(Z)=A_{out}(Z)= \{1\} $. In $S^{III}_{a}$, there is only one admissible cycle, given by $Z_{afilm}$. This cycle makes no turns, and it's complement consists of two strands interacting in a Reidemeister II move. If we label the crossings $c_{1}$ through $c_{6}$ from top to bottom, then the inner cube grading is given by 

\[ gr_{in} = \sum_{i=1}^{6} gr(c_{i}) \]

\noindent
where $gr(c_{i})$ is 0 for the 0-resolution and $1$ for the $1$-resolution. 

The two crossings corresponding to the Reidemeister II move are $c_{1}$ and $c_{4}$, so the homology is concentrated in $gr(c_{1})+gr(c_{4})=1$. For the remaining crossings, $Z_{afilm}$ is diagonal, so it only appears in the singularization. Hence, $gr(c_{2})=gr(c_{3})=0$ and $gr(c_{5})=gr(c_{6})=1$. Thus, the homology is concentrated in local cube grading $3$.

\noindent
\textbf{Case 3:} $A_{in}(Z)=\{2\}$, $A_{out}(Z)= \{1\} $. There are no admissible cycles with this connectivity.

\noindent
\textbf{Case 4:} $A_{in}(Z)=\{3\}$, $A_{out}(Z)= \{1\} $. There are no admissible cycles with this connectivity.

\noindent
\textbf{Case 5:} $A_{in}(Z)=\{1\}$, $A_{out}(Z)= \{2\} $.  There are two admissible cycles in $S^{III}_{a}$, given by $Z_{behkm}$ and $Z_{beilm}$, each of which makes one turn. Note that these two cycles agree for the first two edges, so the top two crossings aren't involved in the maps between them. The crossing $c_{1}$ is forced to be singularized, but $c_{2}$ can have either resolution. We will assume for now that it is singularized as well.

We now only look at the portion of the diagram containing the bottom four crossings, i.e. the diagram from $e_{g}$, $e_{f}$, and $e_{e}$ to $e_{m}$, $e_{n}$ and $e_{o}$. Note that this 4-crossing diagram can be transformed by isotopy to a 2-crossing diagram in which there are no cycles from $e_{e}$ to $e_{m}$. It follows that the \emph{total} homology of this complex is trivial, but we need to show that this is the case for the $E_{2}$ page.

Subject to our assumption that $gr(c_{2}) = 0$, the complex for $Z_{behkm}$ is nontrivial at $4$ vertices in the inner cube. The contributions and edge maps between them are shown in Figure \ref{d28}. The complex for $Z_{beilm}$ is also non-trivial at 4 vertices in the inner cube, shown in Figure \ref{d29}.

\begin{figure}

\centering
\begin{tikzpicture}
  \matrix (m) [matrix of math nodes,row sep=4em,column sep=6em,minimum width=2em] {
     \widetilde{H}_{F}(Z^{000100}_{behkm})& \widetilde{H}_{F}(Z^{000110}_{behkm}) \\
     \widetilde{H}_{F}(Z^{000101}_{behkm}) & \widetilde{H}_{F}(Z^{000111}_{behkm}) \\};
     \scriptsize
  \path[-stealth]
    (m-1-1) edge node [right] {$U_{n}$} (m-2-1)
            edge node [above] {$\Delta_{-}$} (m-1-2)
    (m-2-1) edge node [below] {$\Delta_{-}$} (m-2-2)
    (m-1-2) edge node [right] {$U_{n}$} (m-2-2);
    
\end{tikzpicture}
\caption{The chain complex for $Z_{behkm}$ with $gr(c_{2}) = 0$.}\label{d28}
\end{figure}

\begin{figure}

\centering
\begin{tikzpicture}
  \matrix (m) [matrix of math nodes,row sep=4em,column sep=6em,minimum width=2em] {
     \widetilde{H}_{F}(Z^{000011}_{beilm})& \widetilde{H}_{F}(Z^{000111}_{beilm}) \\
     \widetilde{H}_{F}(Z^{0001011}_{beilm}) & \widetilde{H}_{F}(Z^{001111}_{beilm}) \\};
     \scriptsize
  \path[-stealth]
    (m-1-1) edge node [right] {$U_{f}$} (m-2-1)
            edge node [above] {$\Delta_{-}$} (m-1-2)
    (m-2-1) edge node [below] {$\Delta_{-}$} (m-2-2)
    (m-1-2) edge node [right] {$U_{f}$} (m-2-2);
    
\end{tikzpicture}
\caption{The chain complex for $Z_{beilm}$ with $gr(c_{2}) = 0$.}\label{d29}
\end{figure}

Applying Lemma \ref{injective}, the total complex is given in Figure \ref{d30}, where $f_{1}$ and $f_{2}$ are edge maps that haven't been computed yet. Note that locally, both the $Z_{behkm}$ and $Z_{beilm}$ complexes are the once reduced HOMFLY-PT homology of a single negative crossing, and outside of the local picture they have the exact same diagram. Thus, the reduced complexes are isomorphic. 

\begin{figure}
\centering
\begin{tikzpicture}
  \matrix (m) [matrix of math nodes,row sep=4em,column sep=6em,minimum width=2em] {
     \widetilde{H}_{F}(Z^{000101}_{behkm})/U_{n}=0 & \widetilde{H}_{F}(Z^{000111}_{behkm})/U_{n}=0 \\
     \widetilde{H}_{F}(Z^{0001011}_{beilm})/U_{f}=0 & \widetilde{H}_{F}(Z^{001111}_{beilm})/U_{f}=0 \\};
     \scriptsize
  \path[-stealth]
    (m-1-1) edge node [right] {$f_{1}$} (m-2-1)
            edge node [above] {$\Delta_{-}$} (m-1-2)
    (m-2-1) edge node [below] {$\Delta_{-}$} (m-2-2)
    (m-1-2) edge node [right] {$f_{2}$} (m-2-2);
    
\end{tikzpicture}
\caption{The total complex with $gr(c_{2}) = 0$.}\label{d30}
\end{figure}

We therefore want to show that $f_{1}$ and $f_{2}$ induce this quasi-isomorphism. With respect to the triple grading $(M, A, gr_{cube})$, this quasi-isomorphism would be homogeneous of degree $(-1,0,1)$. If we examine each Alexander grading, the complex is finite-dimensional over $\Z_{2}$ - thus, it satisfies the conditions of Lemma \ref{easyhomlemma}. It follows that the homology is trivial.

The same arguments works for the complex with $gr(c_{2})=1$, proving that the $E_{2}$ page with this connectivity is trivial.

\noindent
\textbf{Case 6:} $A_{in}(Z)=\{2\}$, $A_{out}(Z)= \{2\} $. This case follows from the same argument as Case 2.

\noindent
\textbf{Case 7:} $A_{in}(Z)=\{2\}$, $A_{out}(Z)= \{3\} $. There are no admissible cycles with this connectivity.

\noindent
\textbf{Case 8:} $A_{in}(Z)=\{3\}$, $A_{out}(Z)= \{1\}. $ This case is the most complicated, as $D^{III}_{a}$ has five admissible cycles. They are $Z_{cdgkm}$, $Z_{cdgjlm}$, $Z_{cdfilm}$, $Z_{cehkm}$, and $Z_{ceilm}$. 

By the same argument as Case 5, the edge maps includes a quasi-isomorphism from $Z_{cehkm}$ to $Z_{ceilm}$. For the remaining cycles, we will simplify the complexes until we can apply Lemma \ref{easyhomlemma} to show that the homology is trivial.

The complement of the cycle $Z_{cdgkm}$ consists of 2 strands that make a Reidemeister II move at crossings $c_{3}$ and $c_{5}$. We can therefore make cancellations within this complex until the homology is concentrated at $gr(c_{3})=1$, $gr(c_{5})=0$. The complement of $Z_{cdgjlm}$ consists of two strands with a positive crossing, while the complement of $Z_{cdfilm}$ consists of two strands with a negative crossing. Each of these cycles makes at least one turn, causing the homology of the complement to be reduced. The cycle $Z_{cdgkm}$ actually makes two turns - we will choose the reducing map from $c_{6}$. The total complex after these cancelations is shown in Figure \ref{d31}.

\begin{figure}
\centering
\begin{tikzpicture}
  \matrix (m) [matrix of math nodes,row sep=4em,column sep=6em,minimum width=2em] {
     \widetilde{H}_{F}(Z^{001100}_{cdgkm})/U_{n}=0 & \widetilde{H}_{F}(Z^{001101}_{cdgkm})/U_{n}=0 \\
          \widetilde{H}_{F}(Z^{000111}_{cdgjlm})/U_{o}=0 & \widetilde{H}_{F}(Z^{001111}_{cdgjlm})/U_{o}=0 \\
     \widetilde{H}_{F}(Z^{010011}_{cdfilm})/U_{a}=0 & \widetilde{H}_{F}(Z^{010111}_{cdfilm})/U_{a}=0 \\};
     \scriptsize
  \path[-stealth]
    (m-1-1) edge node [right] {$f_{1}$} (m-2-1)
            edge node [above] {$0$} (m-1-2)
            edge [bend right = 450] node [left] {$f_{3}$} (m-3-1)
    (m-2-1) edge node [below] {$\Delta_{+}$} (m-2-2)
    (m-1-2) edge node [right] {$f_{2}$} (m-2-2)
    	    edge [bend left = 90] node [right] {$f_{4}$} (m-3-2)
    (m-3-1) edge node [below] {$\Delta_{-}$} (m-3-2)
    (m-2-1) edge node [below] {$f_{5}$} (m-3-2);
    
\end{tikzpicture}
\caption{The complex for cycles $Z_{cdgkm}$, $Z_{cdgjlm}$, and $Z_{cdfilm}$.}\label{d31}
\end{figure}
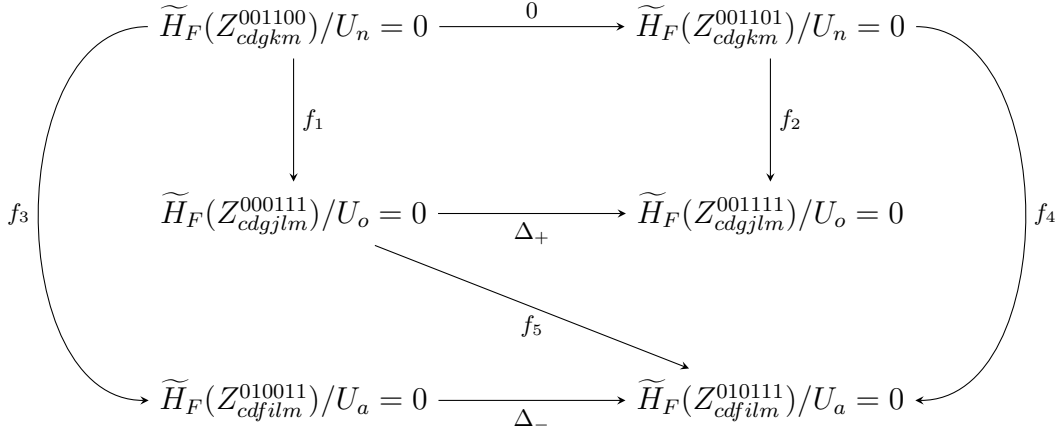

It is possible for $f_{2}$, $f_{3}$, and $f_{5}$ to be isomorphisms, as they map between isotopic diagrams and have the proper gradings. It follows that in each Alexander grading, this complex satisfies the conditions of Lemma \ref{easyhomlemma}, so the homology is trivial.

\begin{rem}

This complex corresponds to a categorification of the HOMFLY-PT skein relation

\[ aP_{H}(a,q,L_{+}) -a^{-1}P_{H}(a,q,L_{-})= (q-q^{-1})P_{H}(a,q,L_{0}) \]

\end{rem}

\noindent
\textbf{Case 9:} $A_{in}(Z)=\{3\}$, $A_{out}(Z)= \{2\} $. This case follows from the same argument as Case 5, with the only difference being that the crossings are positive instead of negative.

\noindent
\textbf{Case 10:} $A_{in}(Z)=\{3\}$, $A_{out}(Z)= \{3\} $. This case follows from the same argument as Case 2.

\noindent
\textbf{Case 11:} $A_{in}(Z)=A_{out}(Z)= \{1, 2\} $. Just as in Case 2, there is only one admissible cycle, namely the cycle defined by $Z^{C}_{cdgjo}$. For cases 11-19, it will be easiest notationally to refer to cycles by their complements, as the cycles themselves have so many edges. Observe that for these cases, the complements of all cycles will consist of a single strand, and are therefore isotopic. The arguments are therefore going to be simplified versions of the arguments in cases 2-10.

The cycle $Z^{C}_{cdgjo}$ only has contributions where $c_{1}$, $c_{2}$, $c_{4}$, and $c_{5}$ are singularized and $c_{3}$ and $c_{6}$ are smoothed. This vertex has local cube grading $3$.

\noindent
\textbf{Case 12:} $A_{in}(Z)=\{1, 2\}$, $A_{out}(Z)= \{1, 3\} $. There are no admissible cycles with this connectivity in $S^{III}_{a}$.

\noindent
\textbf{Case 13:} $A_{in}(Z)=\{1, 2\}$, $A_{out}(Z)= \{2, 3\} $. There are no admissible cycles with this connectivity in $S^{III}_{a}$.

\noindent
\textbf{Case 14:} $A_{in}(Z)=\{1, 3\}$, $A_{out}(Z)= \{1, 2\} $. There are two admissible cycles in $D^{III}_{a}$, given by $Z_{behjo}^{C}$ and $Z_{bdgjo}^{C}$, each making one turn. We know that the total homology is trivial, because there are no cycles with this connectivity in $D^{III}_{b}$. These $4$ contributions lie in $3$ adjacent cube gradings, so applying Lemma \ref{easyhomlemma} to each Alexander grading proves that the homology is trivial.

\noindent
\textbf{Case 15:} $A_{in}(Z)=A_{out}(Z)= \{1, 3\} $. This case follows from the same argument as Case 11.

\noindent
\textbf{Case 16:} $A_{in}(Z)=\{1, 3\}$, $A_{out}(Z)= \{2, 3\} $. There are no admissible cycles with this connectivity in $S^{III}_{a}$.

\noindent
\textbf{Case 17:} $A_{in}(Z)=\{2, 3\}$, $A_{out}(Z)= \{1, 2\} $. There are three admissible cycles with this connectivity - $Z_{afio}^{C}$, $Z_{afhjo}^{C}$, and $Z_{agjo}^{C}$, making one turn, two turns, and one turn, respectively. Using one turn from each complex, we can reduce each homology so that it is concentrated in $3$ adjacent cube gradings. We can then apply Lemma \ref{easyhomlemma} to each Alexander grading, proving that the homology is trivial.

\noindent
\textbf{Case 18:} $A_{in}(Z)=\{2, 3\}$, $A_{out}(Z)= \{1, 3\} $. This case follows from the same argument as Case 14.

\noindent
\textbf{Case 19:} $A_{in}(Z)=A_{out}(Z)= \{2, 3\} $. This case follows from the same argument as Case 11.

\noindent
\textbf{Case 20:} $A_{in}(Z)=A_{out}(Z)= \{1, 2, 3\} $. The only cycle with this connectivity consists of all of the edges in the local tangle. This cycle only contributes when all $6$ crossings are smoothed, which falls in local cube grading $3$.

We have shown that $H_{*}(H_{*}(\widetilde{C}_{F}(S^{III}_{a}), d^{f}_{00}), (d^{f}_{01})^{*})$ lies in local cube grading $3$, proving Reidemeister III invariance.

\subsection{Gradings}

We have shown that for all $k \ge 2$, with respect to the cube grading,

\[ E_{k}(\widetilde{C}_{F}(D_{a}^{I+})\{-1\} \cong E_{k}(\widetilde{C}_{F}(D_{b}^{I+}) \]

\[ E_{k}(\widetilde{C}_{F}(D_{a}^{I-})\{-1\} \cong E_{k}(\widetilde{C}_{F}(D_{b}^{I-}) \]

\[ E_{k}(\widetilde{C}_{F}(D_{a}^{II+})\{-1\} \cong E_{k}(\widetilde{C}_{F}(D_{b}^{II-}) \]

\[ E_{k}(\widetilde{C}_{F}(D_{a}^{II-})\{-1\} \cong E_{k}(\widetilde{C}_{F}(D_{b}^{II-}) \]

\[ E_{k}(\widetilde{C}_{F}(D_{a}^{III})\{-3\} \cong E_{k}(\widetilde{C}_{F}(D_{b}^{III}) \]

It's not hard to see that adding in the overall grading shift of $\frac{1}{2}(-c(D)-b(D))$ turns all of these grading shifts between the $D_{a}$ and $D_{b}$ complexes into $0$. Thus, for all $k \ge 2$,

\[E_{k}(\widetilde{C}_{F}(D)\{\frac{1}{2}(-c(D)-b(D))\} \]

\noindent
gives a graded invariant with respect to the cube grading. To see that the Maslov and Alexander gradings are also invariant under the braid-like Reidemeister moves, we note that the above isomorphisms are induced by Heegaard moves, and the corresponding maps are known to preserve these two gradings. This proves that 

\[E_{k}(\widetilde{C}_{F}(D)\{\frac{1}{2}(-c(D)-b(D))\} \]

\noindent
is a triply graded invariant. With regards to the graded Euler characteristic, we have already shown in Section \ref{Euler} that with respect to the triple grading $(-gr_{q}+gr_{h}, -gr_{h}, gr_{v})$,  

\[
 \sum_{i,j,k} (-1)^{(k-j)/2}q^{i}a^{j} dim(E_{2}^{i,j,k}(\widetilde{C}_{F}(D))) = P_{H}(a,q,D)
\]

\noindent
Note that with the cube grading shifted by $\{\frac{1}{2}(-c(D)-b(D))\}$, the vertical grading is exactly twice the cube grading.

\subsection{The Middle and Reduced Versions} 

We will finish by proving that the middle and reduced versions give well-defined invariants as well. We will continue to use the triple grading $(M,A,gr_{cube})$, but for ease of exposition, we will from now on assume that the cube grading in all three complexes is shifted by $\frac{1}{2}(-c(D)-b(D))$.

\subsubsection{The Middle Theory} To prove invariance of $C_{F}(D)$, we must show that the $E_{2}$ and higher pages do not depend on the choice of decorated bivalent vertex, and that it is invariant under braidlike Reidemeister moves. We will relax the assumption that the decorated edge is leftmost in the braid, even though this means that $H_{*}(C_{F}(S), d_{0}^{f})$ is not necessarily isomorphic to $H_{*}(C_{F}(S), d_{0})$. 

Let $D$ be a diagram for a knot $K$ with a decorated bivalent vertex $v$ and $v_{2}$. Let $HD_{1}$ denote the Heegaard diagram for this decorated braid diagram. Since we will be dealing with multiple Heegaard diagrams that correspond to the same knot projection $D$, we will define $C_{F}(HD)$ to be the chain complex associated to the Heegaard Diagram $HD$.

We can alter this diagram by adding an unknot to the left of the braid whose $U$ variable is set equal to zero. The local diagram for this unknot is shown in Figure \ref{HDUnknot}. Let $HD_{2}$ denote this modified diagram. It is well known that this modification doubles the homology, as the differential from $y$ to $x$ is zero.

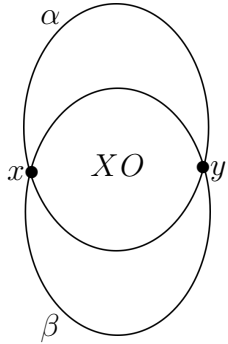
\begin{figure}[h!]
 \centering
   \begin{overpic}[width=.15\textwidth]{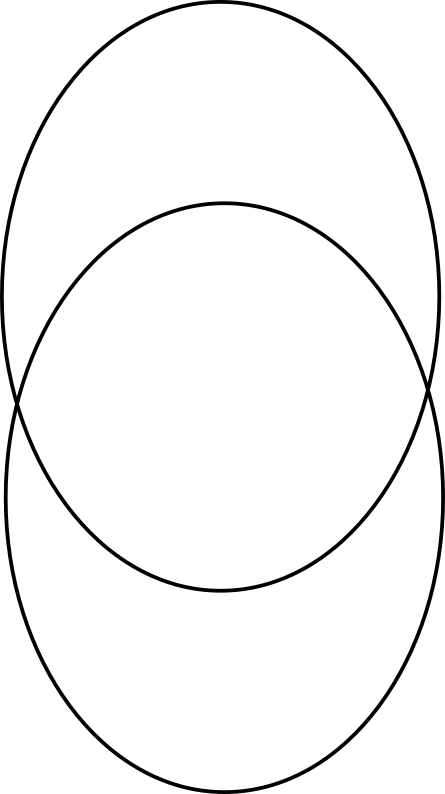}
   \put(20,48){$X$}
   \put(29,48){$O$}
   \put(0,47){$\bullet$}
   \put(-5,47){$x$}
   \put(56,48.5){$y$}
   \put(51.5,48.5){$\bullet$}
   \put(5,94){$\alpha$}
   \put(5,0){$\beta$}
   \end{overpic}
   \caption{The Heegaard diagram for the extra unknot to the left of the braid} \label{HDUnknot}
\end{figure}

With regards to the triple grading, 

\[C_{F}(HD_{2}) \cong C_{F}(HD_{1}) \otimes V \]

\noindent
where $V = \Z_{2}\{-\frac{1}{2},0,0\} \oplus \Z_{2}\{\frac{1}{2},0,0\}$ with respect to the triple grading. We now note that the Heegaard diagram $HD_{2}$ can be connected to a Heegaard diagram $HD_{3}$ in which the marking has been moved to this unknot via moves satisfying the conditions of Lemma \ref{heegaardmoves}. In particular, we handleslide $\alpha_{0}$ over all other $\alpha$ curves which aren't bounded by another $\alpha$ curve, and similarly for $\beta_{0}$. Thus, $C_{F}(HD_{2})$ is filtered chain homotopy equivalent to $C_{F}(HD_{3})$. But by definition, $C_{F}(HD_{3}) = \widetilde{C}_{F}(D)$. 

This proves that $C_{F}(HD_{1}) \otimes V$ is filtered chain homotopy equivalent to $\widetilde{C}_{F}(D)$. We have already shown that $\widetilde{C}_{F}(D)$ is an invariant of braids, and $C_{F}(HD_{1})$ is uniquely determined by $C_{F}(HD_{1}) \otimes V$. This proves that $C_{F}(D)$ is an invariant of braids and that it does not depend on the choice of marking $v$.

\subsubsection{The Reduced Theory} The reduced complex $\overline{C}_{F}(D)$ is defined as follows:

\[ \overline{C}_{F}(D)= C_{F}(D)\{-1,-1,-2\} \xrightarrow{\hspace{3mm}U_{i}\hspace{3mm}}C_{F}(D)\{0,0,0\} \]

\noindent
for some edge $e_{i}$ in $D$. For $\overline{C}_{F}(D)$ to give a well-defined homology theory, we need to show that the $E_{2}$ page does not depend on the choice of edge $e_{i}$. The standard argument is to show that edges which are diagonal at a crossing give the same action. Traversing the knot, this shows that multiplication by any two edges is the same.

\begin{lem} \label{4.12}

Let $e_{m}$ and $e_{n}$ be edges which meet diagonally at a crossing $c$ (for example, edges $e_{1}$ and $e_{4}$ in Figure \ref{HDCrossing}). Then multiplication by $U_{m}$ and $U_{n}$ is the same on $E_{2}(C_{F}(D))$.

\end{lem}

We will use the following lemma in the proof:

\begin{lem} \label{4.13}

Suppose $(C,d)$ is a graded complex that lies in 2 adjacent gradings. Let $d_{i}$ denote the component of $d$ which increases grading by $i$, so $d=d_{0}+d_{1}$. Then 

\[ H_{*}(C, d_{0}+d_{1}) \cong H_{*}(H_{*}(C, d_{0}), d_{1}^{*}) \]

\end{lem}

\begin{proof}

Since the complex lies in two adjacent gradings, there cannot be any induced differentials of length more than $1$. It follows that the spectral sequence induced by the grading must collapse at the $E_{2}$ page.

\end{proof}

\begin{proof}[Proof of Lemma \ref{4.12}]

This is known to be true for standard knot Floer homology, a fact that we will use here. Choose an ordering of the crossings $c_{1},...,c_{k}$ such that $c_{1}=c$. Define $C_{1}$ to be the homology 

\[ H_{*}(C_{F}(D), d_{0}) \]

\noindent
and let $d_{1}(j)$ denote the component of the induced edge map coming from the crossing $c_{j}$. Then, 

\[E_{2}(C_{F}(D)) = H_{*}(C_{1}, d_{1}(1)+...+d_{1}(k)) \]

We can apply different weights to the crossings to induce a filtration on this differential. We will choose to give $c_{1}$ weight $0$ and all other crossings weight $1$. In the corresponding filtration, $d_{1}(1)$ preserves filtration level while $d_{1}(j)$ increases filtration level for $j>1$. This induces a spectral sequence converging to $H_{*}(C_{1}, d_{1}(1)+...+d_{1}(k))$ whose $E_{1}$ page is $H_{*}(C_{1}, d_{1}(1))$. It suffices to show that multiplication by $U_{m}$ and $U_{n}$ are the same on $H_{*}(C_{1}, d_{1}(1))$.

The complex $H_{*}(C_{1}, d_{1}(1))$ preserves the cube filtration coming from crossings $c_{2},...,c_{k}$. Let $S$ denote a partial resolution of $D$ in which all crossings but $c_{1}$ are resolved, but $c_{1}$ remains a crossing. Then $H_{*}(C_{1}, d_{1}(1))$ decomposes as a direct sum over all such partial resolutions $S$. Let $S_{0}$ be the complete resolution in which $c_{1}$ has the 0-resolution and $S_{1}$ the 1-resolution. The complex $(C_{1}, d_{1}(1))$ can be written 

\[  \bigoplus_{S} H_{F}(S_{0}) \xrightarrow{d_{1}(1)} H_{F}(S_{1}) \]

\noindent
But applying Lemma \ref{4.13}, the homology of each of these summands is the \emph{total} homology of $S$, on which $U_{m}$ and $U_{n}$ are known to have the same multiplication.

\end{proof}

Thus, the reduced theory does not depend on the edge at which it is reduced. This also shows that $E_{2}(C_{F}(D))$ is a finitely generated (though not necessarily free) $\Z_{2}[U]$-module, and the reduced theory is finite-dimensional over $\Z_{2}$.

\begin{rem}

The proof of Lemma \ref{4.12} can be summarized as the fact that a cube filtration can't do anything interesting when you're only looking at one crossing, as the $E_{2}$ page will be the same as the total homology.

\end{rem}

This proves that the unreduced, middle, and reduced theories all give will defined invariants, which are analogous to the unreduced, middle, and HOMFLY-PT homologies. We will explore some explicit relationships with the HOMFLY-PT homologies in the next section.

\section{Relationship with HOMFLY-PT homology}

Let $D$ be a braid diagram for a knot $K$. On both knot Floer homology and HOMFLY-PT homology, the $U_{i}$ multiplication maps are all the same, making them $\Z_{2}[U]$-modules. $\mathit{HFK}^{-}(K)$ is known to be of the form

\[ \Z_{2}[U] \oplus \text{\emph{Torsion}} \]

\noindent
where the torsion is killed by some power of $U$. For HOMFLY-PT homology, on the other hand, $H_{H}(K)$ is a free $\Z_{2}[U]$-module of finite rank. In particular, the middle theory is just the reduced theory with a variable adjoined. 

\begin{thm}

Let $T$ be the torsion part of $E_{2}(C_{F}(D))$. Then

\[ E_{2}(C_{F}(D)) \cong H_{H}(K) \oplus T \]

\noindent
and similarly for the unreduced theory. This isomorphism preserves the horizontal grading  $gr_{h}$ and the vertical grading $gr_{v}$, but not necessarily the $q$-grading.

\end{thm}

\noindent
Thus, the $E_{2}$ page being torsion-free would prove that it is in fact HOMFLY-PT homology.

\begin{proof}

We know that the basepoint-filtered homology $H_{*}(H_{*}(C_{F}(D), d_{0}^{f}), (d_{1}^{f})^{*})$ splits over cycles $Z$, and that the contribution of a non-empty an admissible cycle $Z$ is 

\[ H_{H}(D-Z, n(Z)) \]

\noindent
where $n(Z)$ is the number of turns made by the cycle $Z$, and $H_{H}(D,n)$ is the HOMFLY-PT homology of $D$ reduced $n$ times, at least one on each component of $D$ (see Section \ref{filterededgemaps}).

Thus, the contribution from the empty-cycle is the full HOMFLY-PT homology $H_{H}(D)$, while the contribution from the remaining cycles is finite-dimensional over $\Z_{2}$. It follows that even after taking homology with respect to the unfiltered edge maps, the rank over $\Z_{2}[U]$ will be the rank of $H_{H}(D)$.

Note that the non-empty cycles may not be torsion, as multiplication by $U$ may map into the empty cycle. This can be seen in the example of the unknot. This ``filling in" of the cycles underneath the empty cycle is why the $q$-grading is not preserved in the map above. In fact, we don't expect the $q$-grading to be preserved, as we need the extra homology from the non-empty cycles to provide the grading shift that maps $a \mapsto aq$ on the polynomial level.

\end{proof}

\begin{cor}

The reduced theory is given by 

\[ E_{2}(\overline{C}_{F}(D)) \cong \overline{H}_{H}(D) \oplus T' \]

\noindent
where $T'$ is the homology of the mapping cone 

\[T \xrightarrow{U} T \]

\noindent
In particular, there is a rank inequality 

\[ dim(E_{2}(\overline{C}_{F}(D))) \ge dim( \overline{H}_{H}(D)) \]

\end{cor}

\bibliography{TriplyGradedHomology}{}
\bibliographystyle{plain}

\end{document}